\documentclass[final,onefignum,onetabnum]{siamart171218}
\usepackage{moreverb}
\usepackage{amsfonts}
\usepackage{rotating}
\usepackage{amssymb}
\usepackage{tikz}
\usepackage{nicefrac}
\usepackage{placeins}
\usepackage{arydshln}
\usepackage{subcaption}
\usepackage{xr-hyper}
\usepackage{pgfplots}
\pgfplotsset{compat=newest}
\pgfplotsset{plot coordinates/math parser=false}
\newlength\figureheight
\newlength\figurewidth

\newcommand\numberthis{\addtocounter{equation}{1}\tag{\theequation}}
\newsiamremark{remark}{Remark}

\newcommand\BibTeX{{\rmfamily B\kern-.05em \textsc{i\kern-.025em b}\kern-.08em
T\kern-.1667em\lower.7ex\hbox{E}\kern-.125emX}}

\newtheorem{approximaterate}{Approximation}

\title{Multilevel convergence analysis of multigrid-reduction-in-time\protect\thanks{This work performed under the auspices of the U.S. Department of Energy by Lawrence Livermore National Laboratory under Contract
DE-AC52-07NA27344, LLNL-JRNL-763460.}}

\author{
Andreas Hessenthaler\thanks{Institute for Modelling and Simulation of Biomechanical Systems, University of Stuttgart,
Pfaffenwaldring 5a, 70565 Stuttgart, Germany (\email{hessenthaler@mechbau.uni-stuttgart.de})}
\and Ben S.\ Southworth\thanks{Department of Applied Mathematics, University of Colorado at Boulder, CO, USA}
\and David Nordsletten\thanks{Division of Imaging Sciences and Biomedical Engineering,
King's College London, 4th Floor, Lambeth Wing, St.~Thomas Hospital, London, SE1 7EH, UK}
\and Oliver R\"ohrle\footnotemark[2]
\and Robert D.\ Falgout\thanks{Center for Applied Scientific Computing, Lawrence Livermore National Laboratory,
P.O.\ Box 808, L-561, Livermore, CA 94551, USA}
\and Jacob B.\ Schroder\thanks{Department of Mathematics and Statistics, University of New Mexico,
310 SMLC, Albuquerque, NM 87131, USA}
}

\ifpdf
\hypersetup{
  pdftitle={Multilevel convergence analysis of MGRIT},
  pdfauthor={A. Hessenthaler, B. S. Southworth, D. Nordsletten, O. R\"ohrle, R. D. Falgout, J. B. Schroder}
}
\fi

\begin{document}
\allowdisplaybreaks
\maketitle

\begin{abstract}
    This paper presents a multilevel convergence framework for multigrid-reduction-in-time (MGRIT)
    as a generalization of previous two-grid estimates.
    The framework provides \emph{a priori} upper bounds on the convergence of MGRIT V- and F-cycles,
    with different relaxation schemes, by deriving the respective residual and error propagation operators.
    The residual and error operators are functions of the time stepping operator,
    analyzed directly and bounded in norm, both numerically and analytically.
    We present various upper bounds of different computational cost
    and varying sharpness.
    These upper bounds are complemented by proposing analytic formulae for the approximate convergence factor
    of V-cycle algorithms that take the number of fine grid time points, the temporal coarsening factors,
    and the eigenvalues of the time stepping operator as parameters.

    The paper concludes with supporting numerical investigations of parabolic (anisotropic diffusion)
    and hyperbolic (wave equation) model problems.
    We assess the sharpness of the bounds and the quality of the approximate convergence factors.
    Observations from these numerical investigations demonstrate the value of the proposed multilevel
    convergence framework for estimating MGRIT convergence \emph{a priori} and for the design
    of a convergent algorithm.
    We further highlight that observations in the literature are captured by the theory, including
    that two-level Parareal and multilevel MGRIT with F-relaxation do not yield scalable algorithms
    and the benefit of a stronger relaxation scheme. An important observation is that
    with increasing numbers of levels
    MGRIT convergence deteriorates for the hyperbolic model problem, while constant convergence factors
    can be achieved for the diffusion equation.
    The theory also indicates that L-stable Runge-Kutta schemes are more amendable to multilevel
    parallel-in-time integration with MGRIT than A-stable Runge-Kutta schemes.
\end{abstract}


\section{Introduction}\label{introduction-sec}
Modern computer architectures enable massively parallel computations for systems under numerical investigation.
While clock rates of recent high-performance computing architectures have largely become stagnant, increased concurrency
continues to reduce the time-to-solution, allowing for increased complexity of computational models
and accuracy of computed quantities.

Spatial domain decomposition (DD) methods are a wide-spread class of parallelization techniques
to exploit parallelism in numerical simulations.
Many DD methods are straightforward to implement and scalable in parallel up to the point that communication tasks become
dominant over computation tasks. Thus, spatial parallelism may saturate without exploiting the full potential
of the available hardware.

Parallel-in-time methods \cite{Nievergelt1964,Gander2015} increase the amount of parallelism
that can be exploited by introducing parallelism in the temporal domain.
Many such methods exist, including
waveform relaxation \cite{LubichOstermann1987,VandewalleVandevelde1994},
space-time multigrid \cite{HortonVandewalle1995},
parallel implicit time-integrator \cite{FarhatChandesris2003,FarhatCortial2006},
revisionist integral deferred correction \cite{ChristliebMacdonaldOng2010},
spectral deferred correction (SDC) \cite{SpeckRuprechtKrauseEmmettMinionWinkelGibbon2012,EmmettMinion2012,HamonSchreiberMinion2019},
Parareal \cite{LionsMadayTurinici2001}
and multigrid-reduction-in-time \cite{FriedhoffFalgoutKolevMaclachlanSchroder2012,
FalgoutFriedhoffKolevMaclachlanSchroder2014}.
These methods have been developed for various application areas and with varying degree of intrusiveness,
ease of implementation, level of parallelism, and potential for speedup.
For an extensive review, see \cite{Gander2015}.

In this paper, we focus on multigrid-reduction-in-time (MGRIT), a recently developed
iterative, multilevel algorithm, which introduces parallelism in the temporal
domain by employing a parallel, iterative coarse-grid correction in time based on multigrid reduction.
MGRIT has been explored for various application areas, including the numerical solution of parabolic and hyperbolic partial differential
equations (PDEs) \cite{FriedhoffFalgoutKolevMaclachlanSchroder2012,FalgoutManteuffelOneillSchroder2017,
HowseDesterckFalgoutMaclachlanSchroder2019,HessenthalerNordslettenRoehrleSchroderFalgout2018},
investigations of power systems \cite{LecouvezFalgoutWoodwardTop2016,SchroderFalgoutWoodwardTopLecouvez2018},
solving adjoint and optimization problems \cite{GuentherGaugerSchroder2018a,GuentherGaugerSchroder2018b},
and neural network training \cite{Schroder2017}.
Two-level convergence theory for MGRIT was developed in \cite{DobrevKolevPeterssonSchroder2017} for time
integration on a uniform time grid, under the assumption of linear, simultaneously diagonalizable time-stepping operators
(see Section \ref{sec:MGRIT:diag}). The derived \emph{a priori} bounds
were shown to be quite accurate when compared
with convergence observed in practice. Southworth \cite{Southworth2018} generalized this framework, deriving necessary
and sufficient conditions and tight two-level convergence bounds for general two-level MGRIT for linear PDEs on
a uniform time grid. Some extensions to the case of non-uniform time grids are also provided in
\cite{Southworth2018,YueShuXuBuPan2018}.
In a special two-level case, MGRIT and Parareal are equivalent and within this setting, convergence theory
was developed for the linear and nonlinear case \cite{GanderVandewalle2007,GanderHairer2008}.
However, no work has been done on convergence
theory for the general multilevel setting, which is often far superior in practice.
The selection of an appropriate cycling strategy and relaxation
scheme is fundamental to achieve scalable multilevel performance and, ultimately, for achieving parallel speedup.
A theoretical framework for multilevel convergence of MGRIT can help guide these decisions in a rigorous,
a priori manner.

This paper introduces a framework for multilevel convergence analysis of MGRIT,
laying the groundwork for a better understanding of MGRIT in theory and in practice.
While the convergence framework is developed for linear PDEs,
we note that MGRIT employs full approximation storage (FAS) multigrid
(similar to methods like multilevel SDC,
e.g., \cite{EmmettMinion2012,HamonSchreiberMinion2019})
and is thus applicable to the general nonlinear case.
It was shown in \cite{DobrevKolevPeterssonSchroder2017}, however,
that the investigation of linear PDEs can illustrate the strenghts and weaknesses
of two-level MGRIT. Similarly, the new multilevel convergence framework
in this work highlights how the use of a stronger relaxation scheme
and carefully selected time integration schemes for the diffusion equation
can benefit MGRIT convergence significantly.
On the other hand, it highlights problematic areas for MGRIT, for example,
if A-stable Runge-Kutta schemes are used for the parallel-in-time integration
of the second-order wave equation.
Thus, we expect the analysis presented here to help guide
the development and improvement of MGRIT through ideas such as
coarsening in integration order as opposed to step size ($p$-MGRIT;
similar to coarsening in collocation order for multilevel SDC \cite{SpeckRuprechtEmmettMinionBoltenKrause2015}).
Furthermore, we provide a parallel C++ implementation of all derived bounds and approximate convergence
factors in the Supplementary Materials and
as open-source software\footnote{Github repository: \url{github.com/XBraid/XBraid-convergence-est}}
to guide parameter choices for a particular application in an \emph{a priori}
manner.

The paper proceeds as follows. In Section \ref{MGRIT-sec}, notation for a general linear time-stepping problem
is introduced, and in Section \ref{MGRIT-operators-sec}, the MGRIT algorithm and operators are reviewed.
The important assumption of simultaneously diagonalizable operators is discussed in Section \ref{sec:MGRIT:diag}
and the connection between two-level and multilevel convergence is discussed in Section \ref{why-multilevel-is-harder-sec}.
In Section \ref{residual-and-error-propagation-sec} and \ref{bounds-MGRIT-error-propagation-sec},
we first generalize previous two-level convergence theory of MGRIT
to the case with an arbitrary number of relaxation steps, referred to as $r$FCF-relaxation.
We then extend the convergence framework to the multilevel setting, presenting analytic formulae to
approximate the worst-case convergence factor of MGRIT algorithms, for multiple MGRIT cycling strategies
and types of relaxation. The upper bounds on residual and error propagation derived here
are able to analyze multilevel performance of MGRIT \emph{a priori}, both numerically and analytically,
and with varying degree of sharpness and computational cost.
Section \ref{numerical-results-sec} demonstrates the sharpness of the derived theoretical bounds for model
parabolic (including new analysis of the anisotropic diffusion equation) and hyperbolic PDEs,
highlighting the benefits of the theory derived in this work.
\section{Multigrid-reduction-in-time (MGRIT)}\label{MGRIT-sec}
For linear problems, sequential time stepping based on a single-step integration operator $\Phi$ can be written as
\begin{equation}
    \mathbf{u}_n = \Phi_n \mathbf{u}_{n-1} + \mathbf{g}_n, \qquad \text{for } n = 1, \ldots, N_t - 1,
    \label{one-step-eqn}
\end{equation}
with state vector $\mathbf{u}_n \in \mathbb{R}^{N_x}$ at time $t_n \in (0, T]$,
initial condition $\mathbf{u}_0$ at time $t_0 = 0$, and forcing function $\mathbf{g}_n$.
Here, $N_x$ refers to the number of degrees of freedom at one point in time,
and $N_t$ refers to the number of time points.\footnote{Note that in contrast to
\cite{DobrevKolevPeterssonSchroder2017}, we include the initial time point.}
For the theoretical analysis, we consider equidistant time points, $\delta_{t_n} = t_n - t_{n-1} = \delta_t$,
and a time-independent one-step integrator, $\Phi_n = \Phi$, for all $n$.\footnote{Multistep time integration
schemes can be addressed in a similar way; see \cite{FalgoutLecouvezWoodward2017}.}

In matrix form, \eqref{one-step-eqn} can be written as,
\begin{equation}
    A \mathbf{u} =
    \begin{bmatrix}
        I \\
        -\Phi & I \\
              & -\Phi & I \\
              &       & \ddots & \ddots
    \end{bmatrix} \begin{bmatrix}
        \mathbf{u}_0 \\
        \mathbf{u}_1 \\
        \mathbf{u}_2 \\
        \vdots
    \end{bmatrix}
    = \mathbf{g},
    \label{space-time-system-eqn}
\end{equation}
where sequential time-stepping is identified as a block-forward solve of \eqref{space-time-system-eqn}.

Multigrid-reduction-in-time (MGRIT)
\cite{FriedhoffFalgoutKolevMaclachlanSchroder2012,FalgoutFriedhoffKolevMaclachlanSchroder2014}
solves \eqref{space-time-system-eqn} iteratively
and, like sequential time-stepping, is an $O(N_t)$ method, that is, the total number of (block) operations
to solution is linear or near-linear with the number of time steps (assuming MGRIT is applicable/convergent).
MGRIT introduces a multilevel hierarchy of $n_\ell$ time grids of varying step size
to achieve parallelism in the temporal domain, employing a coarse-grid correction based on
multigrid reduction.
The fine grid (referred to as level $\ell = 0$) is composed of all time points
$t_n$ ($n = 0, \ldots, N_t - 1$)
and the coarser grids (referred to as levels $\ell = 1, \ldots, n_\ell -1$)
are derived from a uniform coarsening of the fine grid (see Figure~\ref{time-grid-hierarchy-fig}).
The temporal coarsening factors are denoted as $m_\ell \in \mathbb{N}$
(for $\ell = 0, \ldots, n_\ell - 2$),\footnote{Note that $m_\ell = 1$ for some or all $\ell$
is a valid choice, e.g., for a p-multigrid-like approach.}
such that the number of time points on each grid level is given by
\begin{equation}
    N_\ell = \frac{N_{\ell - 1} - 1}{m_{\ell - 1}} + 1, \qquad \text{for } \ell = 1, \ldots, n_\ell - 1,
    \label{Nl-eqn}
\end{equation}
with corresponding time step size $\delta_{t,\ell}$.
On each  grid level $\ell$, time points are partitioned into F-points (black) and C-points (red),
and the C-points on level $\ell$ compose all points on the next coarser grid level, $\ell + 1$.
\begin{figure}[ht!]
    \centering
    \begin{tikzpicture}[scale=0.8]
        \def\gf{1.75};
        \def\gc{0};
        \def\nx{16}
        \def\df{0.75}
        \def\m{4.0}
        \def\dc{\df * \m}
        \def\add{0.1*\df}
                \draw[thick,black] (0,\gf) -- (\nx*\df,\gf);
        \draw[thick,red] (0,\gc) -- (\m*\dc,\gc);
        \draw[right,black] (\nx*\df+2*\df,\gf) node {level $0$};
        \draw[right,red] (\nx*\df+2*\df,\gc) node {level $1$};
                        \node at (0*\dc,\gf) [rectangle, draw, ultra thick, red, fill=white, minimum width=0.025cm, minimum height=0.4cm] {};
        \node at (1*\dc,\gf) [rectangle, draw, ultra thick, red, fill=white, minimum width=0.025cm, minimum height=0.4cm] {};
        \node at (2*\dc,\gf) [rectangle, draw, ultra thick, red, fill=white, minimum width=0.025cm, minimum height=0.4cm] {};
        \node at (3*\dc,\gf) [rectangle, draw, ultra thick, red, fill=white, minimum width=0.025cm, minimum height=0.4cm] {};
        \node at (4*\dc,\gf) [rectangle, draw, ultra thick, red, fill=white, minimum width=0.025cm, minimum height=0.4cm] {};
        \node at (0*\dc,\gc) [rectangle, draw, ultra thick, red, fill=white, minimum width=0.025cm, minimum height=0.4cm] {};
        \node at (1*\dc,\gc) [rectangle, draw, ultra thick, red, fill=white, minimum width=0.025cm, minimum height=0.4cm] {};
        \node at (2*\dc,\gc) [rectangle, draw, ultra thick, red, fill=white, minimum width=0.025cm, minimum height=0.4cm] {};
        \node at (3*\dc,\gc) [rectangle, draw, ultra thick, red, fill=white, minimum width=0.025cm, minimum height=0.4cm] {};
        \node at (4*\dc,\gc) [rectangle, draw, ultra thick, red, fill=white, minimum width=0.025cm, minimum height=0.4cm] {};
                                                                                                        \draw[thick,black] (1*\df,\gf+-0.15) -- (1*\df,\gf+0.15);
        \draw[thick,black] (2*\df,\gf+-0.15) -- (2*\df,\gf+0.15);
        \draw[thick,black] (3*\df,\gf+-0.15) -- (3*\df,\gf+0.15);
        \draw[thick,black] (5*\df,\gf+-0.15) -- (5*\df,\gf+0.15);
        \draw[thick,black] (6*\df,\gf+-0.15) -- (6*\df,\gf+0.15);
        \draw[thick,black] (7*\df,\gf+-0.15) -- (7*\df,\gf+0.15);
        \draw[thick,black] (9*\df,\gf+-0.15) -- (9*\df,\gf+0.15);
        \draw[thick,black] (10*\df,\gf+-0.15) -- (10*\df,\gf+0.15);
        \draw[thick,black] (11*\df,\gf+-0.15) -- (11*\df,\gf+0.15);
        \draw[thick,black] (13*\df,\gf+-0.15) -- (13*\df,\gf+0.15);
        \draw[thick,black] (14*\df,\gf+-0.15) -- (14*\df,\gf+0.15);
        \draw[thick,black] (15*\df,\gf+-0.15) -- (15*\df,\gf+0.15);
                \draw[above,black] (0*\df,\gf+0.5) node {$t_0$};
        \draw[above,black] (1*\df,\gf+0.5) node {$t_1$};
        \draw[above,black] (16*\df,\gf+0.5) node {$t_{N_0 - 1}$};
        \draw[above,black] (2*\df,\gf+0.5) node {$\ldots$};
                \draw[above,red] (0*\dc,\gc+0.5) node {$t_0$};
        \draw[above,red] (4*\dc,\gc+0.5) node {$t_{N_1 - 1}$};
        \draw[above,red] (1*\dc,\gc+0.5) node {$\ldots$};
                \draw[<->,thick,black,below] (2*\dc-1*\df,\gf-0.5) -- (2*\dc-2*\df,\gf-0.5);
        \draw[below,black] (2*\dc-1.5*\df,\gf-0.5) node {$\delta_0$};
        \draw[<->,ultra thick,red,below] (1*\dc,\gc-0.5) -- (0*\dc,\gc-0.5);
        \draw[below,red] (0.5*\dc,\gc-0.5) node {$\delta_1 = m_0 \cdot \delta_0$};
                \draw[<-,ultra thick,black,below,bend right] (3*\dc-1*\df,\gf-0.5) -- (3*\dc-1.5*\df,\gf-0.75) -- (3*\dc-2*\df,\gf-0.5);
        \draw[below,black] (3*\dc-0.5*\df,\gf-0.5) node {$\Phi_0$};
        \draw[<-,ultra thick,red,below,bend right] (3*\dc,\gc-0.5) -- (2.5*\dc,\gc-0.75) -- (2*\dc,\gc-0.5);
        \draw[below,red] (3*\dc-0.5*\df,\gc-0.5) node {$\Phi_1$};
        \end{tikzpicture}
        \caption{Two-grid hierarchy: time points $t_n$, fine-/coarse-grid step sizes $\delta_0$
            and $\delta_1$, and coarsening factor $m_0 = 4$. On level $0$, F-points are denoted as vertical lines
            and C-points are denoted as squares.}
    \label{time-grid-hierarchy-fig}
\end{figure}
\subsection{MGRIT Operators}\label{MGRIT-operators-sec}
MGRIT approximates the exact coarse-grid time-stepping operator\footnote{Time-stepping on the coarse-grid
is referred to as exact, if it yields the same solution as sequential time-stepping on the fine-grid.}
on level $\ell$ by introducing,
\begin{equation*}
    \Phi_\ell \approx \Phi_{\ell-1}^{m_{\ell-1}}, \qquad \text{for } \ell = 1, \ldots, n_\ell - 1,
\end{equation*}
and we write,
\begin{equation}
    A_\ell =
    \begin{bmatrix}
        I \\
        -\Phi_\ell & I \\
                   & -\Phi_\ell & I \\
                &               & \ddots & \ddots
    \end{bmatrix} \in \mathbb{R}^{N_x N_\ell \times N_x N_\ell}, \qquad \text{for } \ell = 1, \ldots, n_\ell - 1.
    \label{Al-eqn}
\end{equation}

MGRIT constructs coarse-grids from a Schur complement decomposition of \eqref{Al-eqn},
relative to the F/C-splitting from Figure \ref{time-grid-hierarchy-fig} \cite{DobrevKolevPeterssonSchroder2017}.
The Schur complement arises from certain so-called ``ideal'' multigrid restriction and interpolation operators.
Define ideal restriction, $R_\ell  \in \mathbb{R}^{N_x N_{\ell+1} \times N_x N_\ell}$
(level $\ell$ to $\ell+1$, for $\ell = 0, \ldots, n_\ell - 2$) as
{\small
\begin{alignat}{4}
    R_\ell &= \begin{bmatrix}
        I \\
        & \Phi_\ell^{m_\ell - 1} & \Phi_\ell^{m_\ell - 2} & \cdots & \Phi_\ell & I \\
        &                        &                        &        &           &   & \ddots \\
        &                        &                        &        &           &   &        & \Phi_\ell^{m_\ell - 1} & \Phi_\ell^{m_\ell - 2} & \cdots & \Phi_\ell & I \\
    \end{bmatrix},
    \label{Rl-eqn}
\end{alignat}
}and ideal interpolation,
$P_\ell \in \mathbb{R}^{N_x N_\ell \times N_x N_{\ell+1}}$
(level $\ell+1$ to $\ell$, for $\ell = 0, \ldots, n_\ell - 2$),
along with an auxillary operator $S_\ell \in \mathbb{R}^{N_x N_\ell \times N_x (N_\ell - N_{\ell+1})}$
for $\ell = 1, \ldots, n_\ell - 2$, as
{\small
\begin{alignat}{4}
    P_\ell &= \resizebox{.3\hsize}{!}{$\begin{bmatrix}
        I \\
        \Phi_\ell \\
        \vdots \\
        \Phi_\ell^{m_\ell - 1} \\
        & I \\
        & \Phi_\ell \\
        & \vdots \\
        & \Phi_\ell^{m_\ell - 1} \\
        && \ddots \\
        &&& I \\
    \end{bmatrix}$}, \qquad
    S_\ell &= \begin{bmatrix}
		0 \\
		I \\
		& I \\
		&	& \ddots \\
		&	&	& I \\
		&	&	& 0 \\
		&	&	& 0	& I \\
		&	&	&	&	& \ddots \\
		&	&	&	&	&		& I	\\
		&	&	&	&	&		& 0	\\
	\end{bmatrix}.
    \label{Pl-eqn}
\end{alignat}
}

Note that the interpolation operator is not defined as the transpose of the restriction operator,
since in general $\Phi_\ell \neq \Phi_\ell^T$
(for example, see \cite[Equation (28)]{HessenthalerNordslettenRoehrleSchroderFalgout2018}),
and thus, $P_\ell \neq R_\ell^T$.
The number of block columns in $S_\ell$ corresponds to the total number of F-points on level $\ell$.
If the operator $S_\ell$ is applied to the right of $A_\ell$, the result $A_\ell S_\ell$ is composed of
all block rows in $A_\ell$, and all block columns in $A_\ell$ that correspond to F-points on level $\ell$
(zeroing out the respective C-point block columns). Thus, $S_\ell^T A_\ell S_\ell$ is composed of all block rows
and block columns in $A_\ell$ that correspond to F-points.

Using the above definitions, it is straightforward to work out that the multigrid coarse-grid operator,
$R_\ell A_\ell P_\ell$, is then given by the Schur complement \cite{DobrevKolevPeterssonSchroder2017}
of $A_\ell$ \eqref{Al-eqn},
\begin{align}
    R_\ell A_\ell P_\ell
    &= R_{I_\ell} A_\ell P_\ell
    = \begin{bmatrix}
        I \\
        - \Phi_\ell^{m_\ell} & I \\
        & - \Phi_\ell^{m_\ell} & I \\
        && \ddots & \ddots \\
    \end{bmatrix},
    \label{RlAlPl-eqn}
\end{align}
where restriction by injection is given by the operator,
\begin{alignat}{4}
    R_{I_\ell} &= \begin{bmatrix}
        I \\
        & 0 & 0 & \cdots & I \\
        &   &   &        &   & \ddots \\
        &   &   &        &   &        & 0 & 0 & \cdots & 0 & I \\
    \end{bmatrix}, \qquad \text{for } \ell = 1, \ldots, n_\ell - 2.
    \label{RIl-eqn}
\end{alignat}
Here, $R_{I_\ell} \in \mathbb{R}^{N_x N_{\ell+1} \times N_x N_\ell}$
has a similar block structure as $R_\ell$, but
with all blocks $\Phi_\ell^{d}$ (for $d = 1, \ldots, m_\ell - 1$) set to zero.
Thus, $R_{I_\ell}$ restricts the C-points from level $\ell$ to level $\ell + 1$, omitting the respective F-points.
The number of block rows corresponds to the total number of C-points on level $\ell$, i.e. $N_{\ell + 1}$.
Also note that the inverse of $A_\ell$ is given analytically by \cite{DobrevKolevPeterssonSchroder2017}
{\small
\begin{alignat}{4}
    A_\ell^{-1} &= \begin{bmatrix}
        I \\
        \Phi_\ell            & I \\
        \Phi_\ell^2          & \Phi_\ell            & \ddots \\
        \vdots               & \vdots               &        & I \\
        \Phi_\ell^{N_\ell-1} & \Phi_\ell^{N_\ell-2} &        & \Phi_\ell & I \\
    \end{bmatrix}    \in \mathbb{R}^{N_x N_\ell \times N_x N_\ell}, \qquad \text{for } \ell = 1, \ldots, n_\ell - 1.
    \label{invAl-suppeqn}
\end{alignat}
}

In a typical multigrid fashion, MGRIT uses a complementary relaxation process to reduce error that is
not adequately reduced on coarser grids. Because MGRIT is a reduction-based method, coarse-grid correction
(should) eliminate error effectively at C-points, so this is coupled with an F-relaxation scheme to eliminate
error at F-points. F-relaxation can be seen as a block-Jacobi like method, where in this case each block
consists of a set of contiguous F-points in the time domain (that is, F-relaxation updates all F-points based
on sequential time integration from the closest (previous) C-point). Algebraically, this is equivalent to an
application of the idempotent operator
\begin{equation}
    F_\ell = P_\ell R_{I_\ell} = I - S_\ell (S_\ell^T A_\ell S_\ell)^{-1} S_\ell^T.
    \label{F-relaxation-eqn}
\end{equation}

When F-relaxation alone is insufficient, a stronger relaxation scheme can be used.
Define $T_\ell = R_{I_\ell}^T$. Then, C-relaxation updates a C-point based on taking one time step
from the previous F-point (equivalent to block Jacobi applied to the C-point block rows of $A_\ell$).
Algebraically, this corresponds to an application of
\begin{equation}
    C_\ell
    = I - T_\ell (T_\ell^T A_\ell T_\ell)^{-1} T_\ell^T A_\ell,
   \label{C-relaxation-eqn}
\end{equation}
where the zero columns correspond to C-points.
Working through the algebra, FCF-relaxation (subsequent F-, C- and F-relaxation steps) can be written as
\begin{alignat}{4}
    F_\ell C_\ell F_\ell = P_\ell (I - R_\ell A_\ell P_\ell) R_{I_\ell}.
    \label{FCF-relaxation-eqn}
\end{alignat}

\subsection{Simultaneous diagonalization of $\{\Phi_\ell\}$}\label{sec:MGRIT:diag}

Let $\Phi_\ell$ denote the time-stepping operator on level $\ell$. Primary results in this paper rest on the assumption that
$\{\Phi_\ell\}$ are diagonalizable with the same set of eigenvectors, for all levels $\ell=0,...,n_\ell-1$.
This is equivalent to saying that the set $\{\Phi_\ell\}$ commutes, that is, $\Phi_i\Phi_j = \Phi_j\Phi_i$ for all $i,j$, and that $\Phi_\ell$
is diagonalizable for all $\ell$. The concept of simultaneous diagonalization (albeit, in the Fourier basis) was introduced
in \cite{FriedhoffMaclachlan2015}, and was modified and used as the basis for the improved two-grid convergence bounds developed
in \cite{DobrevKolevPeterssonSchroder2017}.

In terms of when such an assumption is valid, let $\mathcal{L}$ be a time-independent operator (such as a spatial discretization)
that is propagated through time by operators $\{\Phi_\ell\}$. Note that all rational functions of $\mathcal{L}$ commute and that if
$\mathcal{L}$ is diagonalizable, so is any rational function of $\mathcal{L}$. Indeed, nearly all standard time-integration routines,
including all single-step Runge-Kutta-type methods, consist of some rational function of $\mathcal{L}$, and all such schemes are
simultaneously diagonalizable with the eigenvectors of $\mathcal{L}$. To that end, let $\mathcal{L} = UDU^{-1}$, where
$D_{kk} = \xi_k$ is a diagonal matrix containing the eigenvalues of $\mathcal{L}$ and columns of $U$ are the corresponding
eigenvectors. Denote the Butcher tableau of a general $s$-stage Runge-Kutta method as
    \begin{center}
        \begin{tabular}{ c | c }
            $\mathfrak{c}$ & $\mathfrak{A}$ \\ \hline
                           & $\mathfrak{b}^T$
        \end{tabular}.
    \end{center}
With some algebra, one can show that $\Phi_\ell$ corresponding to a given Butcher tableau is exactly given by the
Runge-Kutta stability function applied to $\mathcal{L}$ in block form,
\begin{equation}
    \begin{split}
    \Phi_\ell
    &= I + \delta_{t,\ell} (\mathfrak{b}^T\otimes I) \left(I - \delta_{t,\ell} \mathfrak{A} \otimes \mathcal{L}\right)^{-1} (\mathbf{1}_s \otimes \mathcal{L}) \\
    &= U\Big(I + \delta_{t,\ell} \left( \mathfrak{b}^T \otimes I \right) \left(I - \delta_{t,\ell} \mathfrak{A} \otimes D\right)^{-1} (\mathbf{1}_s \otimes D) \Big)U^{-1} \\
        & = U \Lambda_\ell U^{-1},
    \end{split}
    \label{phi-diagonal-form-eqn}
\end{equation}
where $(\Lambda_\ell)_{kk} = \lambda_{\ell,k}$, for $k = 1, \ldots, N_x$,
are the eigenvalues of $\Phi_\ell$, given by
\begin{equation} \label{eq:rk_eigs}
    \lambda_{\ell,k} = 1 + \delta_{t,\ell} \xi_k \mathfrak{b}^T (I - \delta_{t,\ell} \xi_k \mathfrak{A})^{-1} \mathbf{1}.
\end{equation}
Note that Equation \eqref{eq:rk_eigs} is exactly the stability function for a Runge-Kutta scheme applied
to $\delta_{t,\ell} \xi_k$, for time step $\delta_{t,\ell}$
and spatial eigenvalue $\{\xi_k\}$ \cite[Sec.\ 2.1, \S.4]{Kraaijevanger1991}. This highlights the fact
that solving the spatial eigenvalue problem also provides the eigenvalues of all $\Phi_\ell$ for arbitrary Runge-Kutta schemes
and time-step sizes.

Now suppose $\mathcal{A}$ is some matrix operator, where each entry is a rational function of time-stepping operators
in $\{\Phi_\ell\}$. In particular, this applies to error and residual propagation operators of MGRIT that are derived in
Section \ref{residual-and-error-propagation-sec}.
Let $\widetilde{U}$ denote a block-diagonal matrix with diagonal blocks
given by $U$. Then, as in \cite{Southworth2018},
\begin{align}\label{eq:u*u}
    \| \mathcal{A} (\Phi_0, \ldots, \Phi_{n_\ell-1}) \|_{(\widetilde{U}\widetilde{U}^*)^{-1}}
    &= \sup_k \| \mathcal{A} (\lambda_{0,k}, \ldots, \lambda_{n_\ell-1,k}) \|_2.
\end{align}
Thus, the $(\widetilde{U}\widetilde{U}^*)^{-1}$-norm of $\mathcal{A}(\Phi_0,...,\Phi_{n_\ell-1})$ can be computed by
maximizing the $\ell^2$-norm of $\mathcal{A}$ over eigenvalues of $\{\Phi_\ell\}$. In the case that $\{\Phi_\ell\}$ are
normal matrices, $U$ is unitary and the $(\widetilde{U}\widetilde{U}^*)^{-1}$-norm reduces to the standard Euclidean
2-norm. More generally, we have bounds on the $\ell^2$-norm of $\mathcal{A}(\Phi_0,...,\Phi_{n_\ell-1})$,
\begin{equation}
    \begin{split}
        \frac{1}{\kappa(U)} &\left(\sup_k \| \mathcal{A}(\lambda_{0,k}, \ldots, \lambda_{n_\ell-1,k})^i \|_2 \right) \\
        &\leq \| \mathcal{A}(\Phi_0, \ldots, \Phi_{n_\ell-1})^i \|_2
        \leq \kappa(U) \left( \sup_k \|\mathcal{A}(\lambda_{0,k}, \ldots, \lambda_{n_\ell-1,k})^i \|_2 \right),
    \end{split}
\end{equation}
for $i = 1, 2, \ldots$ applications of $\mathcal{A}$, where $\kappa(U)$ denotes the matrix condition number of $U$.\footnote{
A similar modified norm also occurs in the case of integrating in time with a mass matrix \cite{DobrevKolevPeterssonSchroder2017}.}

Here, we are interested in $\mathcal{A}$ corresponding to the error- and residual-propagation operators
of $n_l$-level MGRIT, denoted as $\mathcal{E}^{n_l}$ and $\mathcal{R}^{n_l}$, respectively.
Convergence of MGRIT requires
\begin{equation}
    \| \left( \mathcal{E}^{n_l} \right)^i \|, \| \left( \mathcal{R}^{n_l} \right)^i \| \to 0,
\end{equation}
as iteration $i$ increases.
To that end, bounding $\sup_k \| \left( \mathcal{E}^{n_l} (\lambda_{0,k}, \ldots, \lambda_{n_\ell-1,k}) \right)^i \| < 1$
for all $k$ provides necessary and sufficient conditions
for $\| \left( \mathcal{E}^{n_l} (\Phi_0, \ldots, \Phi_{n_\ell-1}) \right)^i \| \to 0$ with $i$
(eventually), and similarly for $\mathcal{R}^{n_l} (\Phi_0, \ldots, \Phi_{n_\ell-1})$. Here, we have focused on the $\ell^2$-norm.
It is worth pointing out that, in some cases, people are interested in an $\ell^\infty$-norm. However, because
$\|\cdot\|_\infty \leq \|\cdot\|_2$, conditions developed here apply to the $\ell^\infty$-norm as well.

\subsection{Two-level results, and why multilevel is harder}
\label{why-multilevel-is-harder-sec}
For multigrid-type algorithms, it is generally the case that two-level convergence rates provide a lower bound
on attainable multilevel convergence rates; that is, a multilevel algorithm will typically observe worse convergence
than its two-level counterpart. Indeed, more expensive multilevel cycling strategies such as
W-cycles or F-cycles are used specifically to solve the coarse-grid operator more accurately, thus better approximating a two-grid
method. The non-Galerkin coarse grid used in MGRIT makes this relationship more complicated, and it is not definitive that
two-grid convergence provides a lower bound on multilevel in \textit{all} cases. However, in practice, it is consistently the
case that two-level convergence is better than multilevel. To that end, a two-level method which converges every
iteration is a heuristic necessary condition for multilevel convergence.

Bounds on two-grid convergence obtained in \cite{DobrevKolevPeterssonSchroder2017,Southworth2018} are tight to
$\mathcal{O}(1/N_1)$. However, these bounds only apply to either (i) error/residual on all points and all iterations
\textit{except} the first, or (ii) error/residual for all iterations, but \textit{only on C-points}. In the multilevel setting, it is necessary
to consider convergence over \textit{all} points for \textit{one} iteration. To understand why this is, consider a three-level
MGRIT V-cycle, with levels $0$, $1$ and $2$. On level $1$, a single iteration of two-level MGRIT is applied as an approximate
residual correction for level 0. Suppose conditions in \cite{DobrevKolevPeterssonSchroder2017,Southworth2018} are
satisfied, ensuring a decrease in C-point error, but a possible increase in F-point error on level $1$. If the total error on
level $1$ has increased, then a correction is interpolated to level $0$ that is a \textit{worse approximation} to the desired
exact residual correction than no correction at all (corresponding to the zero initial guess used for coarse-grid correction
in multigrid). In general, if divergent behavior is observed for iterations in the middle of the hierarchy, it is likely the
case that the whole multilevel scheme will diverge.

Extensions to the theory developed in \cite{Southworth2018} can be derived to place tight bounds on error/residual
propagation for all points and one iteration \cite{tight2}. It turns out that indeed convergence factors can be larger
and the region of convergence with respect to $\delta_{t,\ell} \xi_k$ smaller compared with bounds on C-point error or later
iterations \cite{tight2}. Here, we do not analyze the two-level setting further and, rather, use this as motivation
to consider the multilevel setting in detail. The remainder of this paper derives analytical multilevel error and
residual propagation operators and proceeds to develop upper bounds on convergence in the $\ell^2$-norm.

\section{Multilevel residual and error propagation}\label{residual-and-error-propagation-sec}

As noted in \cite{Southworth2018}, residual and error propagation are formally similar,
that is,
\begin{alignat}{4}
    \mathcal{R}^{n_\ell} = A_0 \mathcal{E}^{n_\ell} A_0^{-1}
    =  A_0 (I - M^{-1} A_0) A_0^{-1}
    =  I - A_0 M^{-1},
    \label{R-E-similar-eqn}
\end{alignat}
where $M^{-1}$ denotes the MGRIT preconditioner for $A_0^{-1}$. Noting that there is a closed form
for $A^{-1}$ (see Equation \eqref{invAl-suppeqn}),
it follows that if the error propagation operator of a particular MGRIT algorithm is known, the residual
propagation operator can be easily found by the relation in \eqref{R-E-similar-eqn}, and vice-versa.
In this section, we derive the error propagation operator for generalized two-level MGRIT
and multilevel (V-cycle) MGRIT with F- and FCF-relaxation,
which are then used to develop analytic \emph{a priori} bounds on MGRIT convergence
in Section \ref{bounds-MGRIT-error-propagation-sec}, as well as to construct the
error propagation operator and compute its norm directly in numerical tests in Section \ref{numerical-results-sec}.

In the remainder of this work, we use the following convention for sums and products:
for
$b < a$,
$\sum_{i = a}^b f_i = 0$
and
$\prod_{i = a}^b f_i = 1$.
We further write $\mathcal{E}^{n_\ell = 2}$
to refer to the two-grid error propagation operator and similarly for other numbers of levels $n_\ell$.
\subsection{Two-level MGRIT with \lowercase{$r$}FCF-relaxation}\label{multilevel-propagators-rFCF-relax-sec}

Here, we generalize the two-level error propagator, as given in
\cite{DobrevKolevPeterssonSchroder2017}, to two-level MGRIT with $r$FCF-relaxation.
$r$FCF-relaxation refers to F-relaxation followed by $r$ CF-relaxation steps.
A similar result can be found in \cite{GanderKwokZhang2018}, where MGRIT was interpreted as Parareal
with overlap in time.

The error propagator for an exact iterative two-grid method with $r$FCF-relaxation and $r \geq 0$
is given as,
\begin{align}
    \begin{split}
        0 &= I - A_0^{-1} A_0
        = (I - P_0 (R_0 A_0 P_0)^{-1} R_0 A_0) (F_0 C_0)^r F_0 \\
        &= (I - P_0 (R_0 A_0 P_0)^{-1} R_0 A_0) P_0 (I - R_0 A_0 P_0)^r R_{I_0},
    \end{split}
    \label{exact-E-rFCF-nl2-eqn}
\end{align}
and MGRIT approximates the coarse-grid operator as $A_1 \approx R_0 A_0 P_0$.
\begin{lemma}
    The error propagator of two-level MGRIT with $r$FCF-relaxation and $r \geq 0$ is given as,
    \begin{equation}
        \mathcal{E}_{rFCF}^{n_\ell = 2}= (I - P_0 A_1^{-1} R_0 A_0) P_0 (I - R_0 A_0 P_0)^r R_{I_0}.
        \label{E-rFCF-nl2-eqn}
    \end{equation}
\end{lemma}
    \begin{proof}
        This follows by substituting the coarse-grid operator $A_1 \approx R_0 A_0 P_0$
        in to Equation \eqref{exact-E-rFCF-nl2-eqn}.
    \end{proof}

\subsection{Multilevel V-cycles with F-relaxation}\label{multilevel-V-cycle-propagators-F-relax-sec}

The error propagator of a multilevel V-cycle method with F-relaxation can be derived
from the error propagator of the exact two-level method on level $\ell$,
\begin{equation}
    \begin{split}
        0 &= I - A_\ell^{-1} A_\ell
        = (I - P_\ell (R_\ell A_\ell P_\ell)^{-1} R_\ell A_\ell) (I - S_\ell (S_\ell^T A_\ell S_\ell)^{-1} S_\ell^T A_\ell) \\
        &= I - (P_\ell (R_\ell A_\ell P_\ell)^{-1} R_\ell + S_\ell (S_\ell^T A_\ell S_\ell)^{-1} S_\ell^T) A_\ell
    \end{split}
\end{equation}
and the additional relation
\begin{alignat*}{4}
    A_\ell^{-1} &= P_\ell (R_\ell A_\ell P_\ell)^{-1} R_\ell + S_\ell (S_\ell^T A_\ell S_\ell)^{-1} S_\ell^T. \numberthis
    \label{Al-inv-V-F-relax-eqn}
\end{alignat*}

\begin{lemma}    \label{E-F-nl-theo}
    The error propagator of a multilevel V-cycle method with F-relaxation
    is given as,
    \begin{align}
        \begin{split}
        \mathcal{E}_{F}^{n_\ell} = P_0 R_{I_0}
        & - \left( \prod_{k = 0}^{n_\ell - 2} P_k \right) A_{n_\ell - 1}^{-1} \left( \prod_{k = n_\ell - 2}^{0} R_k \right) A_0 P_0 R_{I_0} \\
        & - \sum_{i = 0}^{n_\ell - 3} \left( \prod_{k = 0}^{i} P_k \right) S_{i + 1} (S_{i + 1}^T A_{i + 1} S_{i + 1})^{-1} S_{i + 1}^T \left( \prod_{k = i}^{0} R_k \right) A_0 P_0 R_{I_0},
        \end{split}
        \label{E-F-nl-eqn}
    \end{align}
    for $n_\ell \geq 2$ levels.
\end{lemma}
    \begin{proof}        For $n_\ell = 2$, we have,
        \begin{equation*}
            \mathcal{E}_{F}^{n_\ell = 2} = P_0 R_{I_0} - P_0 A_1^{-1} R_0 A_0 P_0 R_{I_0} = (I - P_0 A_1^{-1} R_0 A_0) P_0 R_{I_0},
        \end{equation*}
        which is equivalent to \eqref{E-rFCF-nl2-eqn} for $r = 0$.
        Now, assume it is true for $n_\ell = n$ levels. Substituting an exact two-level method on the
        coarse grid, that is \eqref{Al-inv-V-F-relax-eqn}, yields,
        \begin{alignat*}{4}
            \mathcal{E}_{F}^{n_\ell = n} &= P_0 R_{I_0}
            && - \left( \prod_{k = 0}^{n - 2} P_k \right) \bigg[ P_{n - 1} (R_{n - 1} A_{n - 1} P_{n - 1})^{-1} R_{n - 1} \\
            &&& + S_{n - 1} (S_{n - 1}^T A_{n - 1} S_{n - 1})^{-1} S_{n - 1}^T \bigg] \left( \prod_{k = n - 2}^{0} R_k \right) A_0 P_0 R_{I_0} \\
            &&& - \sum_{i = 0}^{n - 3} \left( \prod_{k = 0}^{i} P_k \right) S_{i + 1} (S_{i + 1}^T A_{i + 1} S_{i + 1})^{-1} S_{i + 1}^T \left( \prod_{k = i}^{0} R_k \right) A_0 P_0 R_{I_0} \\
            &= P_0 R_{I_0}
            && - \left( \prod_{k = 0}^{n - 1} P_k \right) (R_{n - 1} A_{n - 1} P_{n - 1})^{-1} \left( \prod_{k = n - 1}^{0} R_k \right) A_0 P_0 R_{I_0} \\
            &&& - \sum_{i = 0}^{n - 2} \left( \prod_{k = 0}^{i} P_k \right) S_{i + 1} (S_{i + 1}^T A_{i + 1} S_{i + 1})^{-1} S_{i + 1}^T \left( \prod_{k = i}^{0} R_k \right) A_0 P_0 R_{I_0}.
        \end{alignat*}
        Approximating the exact coarse grid operator on level $n + 1$ by
        $A_n \approx R_{n - 1} A_{n - 1} P_{n - 1}$ completes the proof.
    \end{proof}
\subsection{Multilevel V-cycles with FCF-relaxation}\label{multilevel-V-cycle-propagators-FCF-relax-sec}

The error propagator of a multilevel V-cycle method with FCF-relaxation can be derived
from the error propagator of the exact two-level method on level $\ell$,
\begin{align}
    \begin{split}
        0 =&~I - A_\ell^{-1} A_\ell
        = (I - P_\ell (R_\ell A_\ell P_\ell)^{-1} R_\ell A_\ell) F_\ell C_\ell F_\ell \\
        =&~I - P_\ell (R_\ell A_\ell P_\ell)^{-1} R_\ell A_\ell
        - S_\ell (S_\ell^T A_\ell S_\ell)^{-1} S_\ell^T A_\ell
        - T_\ell (T_\ell^T A_\ell T_\ell)^{-1} T_\ell^T A_\ell \\
        &+ S_\ell (S_\ell^T A_\ell S_\ell)^{-1} S_\ell^T A_\ell T_\ell (T_\ell^T A_\ell T_\ell)^{-1} T_\ell^T A_\ell \\
        &+ P_\ell (R_\ell A_\ell P_\ell)^{-1} R_\ell A_\ell T_\ell (T_\ell^T A_\ell T_\ell)^{-1} T_\ell^T A_\ell,
    \end{split}
\end{align}
and the additional relation
\begin{equation}
    \begin{split}
        A_\ell^{-1}
        =&~T_\ell (T_\ell^T A_\ell T_\ell)^{-1} T_\ell^T \\
        &+ \left[ S_\ell (S_\ell^T A_\ell S_\ell)^{-1} S_\ell^T + P_\ell (R_\ell A_\ell P_\ell)^{-1} R_\ell \right] \left[ I - A_\ell T_\ell (T_\ell^T A_\ell T_\ell)^{-1} T_\ell^T \right].
    \end{split}
\end{equation}

\begin{lemma}
    \label{E-FCF-nl-theo}
    The error propagator of a multilevel V-cycle method with\newline FCF-relaxation is given as,
    \begin{alignat*}{4}
        \mathcal{E}_{FCF}^{n_\ell} \numberthis
        &= P_0 (I - (T_0^T A_0 T_0)^{-1} R_{I_0} A_0 P_0) R_{I_0} \\
        & - \left( \prod_{k = 0}^{n_\ell - 2} P_k \right) A_{n_\ell - 1}^{-1} \left( \prod_{k = n_\ell - 2}^{0} R_k \left[ I - A_k T_k (T_k^T A_k T_k)^{-1} T_k^T \right] \right) A_0 P_0 R_{I_0} \\
        & - \sum_{i = 1}^{n_\ell - 2} \left( \prod_{k = 0}^{i - 1} P_k \right) \bigg[
        S_i (S_i^T A_i S_i)^{-1} S_i^T \left[ I - A_i T_i (T_i^T A_i T_i)^{-1} T_i^T \right] \\
        & + T_i (T_i^T A_i T_i)^{-1} T_i^T \bigg] \left( \prod_{k = i - 1}^{0} R_k \left[ I - A_k T_k (T_k^T A_k T_k)^{-1} T_k^T \right] \right) A_0 P_0 R_{I_0},
        \label{E-FCF-nl-eqn}
    \end{alignat*}
    with $n_\ell \geq 2$ levels.
\end{lemma}
    \begin{proof}
        The proof is analogous to the proof of Lemma \ref{E-F-nl-theo}.
    \end{proof}
\subsection{Multilevel F-cycles with \lowercase{$r$}FCF-relaxation}\label{multilevel-F-cycle-propagators-rFCF-relax-sec}

Similar to the notation used for V-cycle error propagation, let $\mathcal{F}^{n_\ell}$ denote error propagation of MGRIT
F-cycles with $n_\ell$ levels, with a subscript denoting relaxation scheme.
Following \cite[pg. 53]{HackbuschTrottenberg2006}, error propagation of MGRIT
for a multilevel F-cycle with $r$FCF-relaxation can be defined recursively,
\begin{equation}
    \mathcal{F}_{rFCF}^{n_\ell} = M_{rFCF,0}^F, \qquad \text{for } n_\ell \geq 2,
    \label{F-rFCF-nl-eqn}
\end{equation}
with
\begin{alignat*}{4}
    M_{rFCF,\ell - 1}^F
    &= P_{\ell - 1} \left( I - \left( I - M_\ell^V M_\ell^F \right) A_\ell^{-1} R_{\ell - 1} A_{\ell - 1} P_{\ell-1} \right) (I - R_{\ell - 1} A_{\ell - 1} P_{\ell - 1})^r R_{I_{\ell - 1}},    \\
    M_{rFCF,\ell - 1}^V
    &= P_{\ell - 1} \left( I - \left( I - M_\ell^V \right) A_\ell^{-1} R_{\ell - 1} A_{\ell - 1} P_{\ell - 1} \right) (I - R_{\ell - 1} A_{\ell - 1} P_{\ell - 1})^r R_{I_{\ell - 1}},
\end{alignat*}
for $l = 1, \ldots, n_\ell - 2$, and,
\begin{align*}
    M_{rFCF,n_\ell - 2}^F
    &= M_{rFCF,n_\ell - 2}^V \\
    &= P_{n_\ell - 2} \left( I - A_{n_\ell - 1}^{-1} R_{n_\ell - 2} A_{n_\ell - 2} P_{n_\ell - 2} \right) (I - R_{n_\ell - 2} A_{n_\ell - 2} P_{n_\ell - 2})^r R_{I_{n_\ell - 2}}.
\end{align*}
It is easy to verify, that for $n_\ell = 2$, the recursive formulae result in
    $\mathcal{F}_{rFCF}^{n_\ell = 2} = \mathcal{E}_{rFCF}^{n_\ell = 2}$.
For $n_\ell = 3$ and $r = 0$, we can write,
\begin{alignat}{4}
    \mathcal{F}_F^{n_\ell = 3}
    &= \mathcal{E}_F^{n_\ell = 2}
    + P_0 P_1 \left( I - A_2^{-1} R_1 A_1 P_1 \right)^2 R_{I_1} A_1^{-1} R_0 A_0 P_0 R_{I_0}.
    \label{F3_F-E2_F-relation-eqn}
\end{alignat}
However, it is not straightforward to convert the recursive definition in \eqref{F-rFCF-nl-eqn}
into a summation similar to \eqref{E-F-nl-eqn} or \eqref{E-FCF-nl-eqn}, for arbitrary $n_\ell$.
Nevertheless, this formula is still useful for numerically computing bounds of $\mathcal{F}_{rFCF}^{n_\ell}$
and is, thus, included for completeness.

\section{Bounds for MGRIT residual and error propagation}\label{bounds-MGRIT-error-propagation-sec}
Following the work in \cite{DobrevKolevPeterssonSchroder2017},
we assume that operators $\Phi_\ell$, $\ell = 0, \ldots, n_\ell - 1$, can be diagonalized by the same set of
eigenvectors (see Equation \eqref{phi-diagonal-form-eqn}), with eigenvalues denoted $\{\lambda_{\ell,k}\}$,
for $1 \leq k \leq N_x$.
We also assume that $\Phi_\ell$ are strongly stable time stepping operators, that is, $\|\Phi_\ell\| < 1$, which implies
$| \lambda_{\ell,k} | < 1$ for all $\ell = 0, \ldots, n_\ell - 1$ and $k = 1, \ldots, N_x$.\footnote{Note, it is possible to have a
stable time integration scheme with $\|\Phi_\ell\| > 1$ if $\| \Phi_\ell^i \| < 1$ for some $i$
\cite[Section 9.5, Equation (9.22)]{Leveque2007, LaxRichtmyer1956}, but we do not consider such schemes.}
To simplify notation in the
following derivations, we use $\Phi_\ell$ to denote the diagonalized time stepping operator moving forward.
Results then follow in a $(\widetilde{U}\widetilde{U}^*)^{-1}$-norm, which (as discussed in Section \ref{sec:MGRIT:diag}) is equivalent
to the $\ell^2$-norm if $\Phi_\ell$ is normal (and, thus, $U$ is unitary).

For ease of presentation and because many of the derivations are fairly involved, but repetitive, a number
of steps are moved to the \emph{Supplementary Materials}.
We refer the interested reader to \ref{bounds-MGRIT-error-propagation-appsec}.

\subsection{Residual and error on level $0$ and level $1$}\label{residual-and-error-on-level-0-1-sec}
It is typically difficult or impossible in practical applications to precisely measure the error propagation
of an iterative method or the error itself.
It is, however, possible to measure the residual, and stopping criteria for iterative methods
are often based on a residual {tolerance}. In the case of MGRIT, there is a nice relation between
error and residual propagation. The norm of residual and error propagation operators are equal
in the $(\widetilde{U}\widetilde{U}^*)^{-1}$-norm (recall, $\widetilde{U}$ is a block diagonal
matrix of eigenvectors, $U$).\footnote{Although \cite{Southworth2018} specifically addresses two-grid
bounds, equality of error and residual propagation in the $(\widetilde{U}\widetilde{U}^*)^{-1}$-norm
follows if $\Phi_\ell$ is simultaneously diagonalizable for all levels $\ell$.} If $\{ \Phi_\ell \}$ are normal operators,
they are diagonalizable by unitary transformation, in which case $\widetilde{U}\widetilde{U}^* = I$, and error
and residual propagation are equal in the $\ell^2$-norm.

Similar to Section \ref{residual-and-error-propagation-sec},
let $\mathcal{E}_{rFCF}^{n_\ell}$ be the $n_l$-level error propagator, acting on all points
on level $0$.
We further refer to $\mathcal{E}_{rFCF}^{n_\ell,\Delta}$ as the error propagator
that acts on all points on level $1$,
i.e. on the error at the C-points
on level~$0$
In the two-grid setting, we also refer to $\mathcal{E}_{rFCF}^{n_\ell,\Delta}$ as the coarse-grid error propagator.

To quantify how \emph{fast} MGRIT converges in the \emph{worst case},
we can bound the convergence factor of the fine grid residual \cite{DobrevKolevPeterssonSchroder2017}
$\mathbf{r}_{i+1}$ at iteration $i + 1$, $i \in \mathbb{N}_0$, by the norm of the error propagator on level $1$ (in the unitary case),
\begin{equation}
    \| \mathbf{r}_{i+1} \|_2 / \| \mathbf{r}_i \|_2
    = \| A_1 \mathbf{e}_{i+1}^\Delta \|_2 / \| A_1 \mathbf{e}_i^\Delta \|_2
    \leq \| A_1 \mathcal{E}_{rFCF}^{n_\ell, \Delta} A_1^{-1} \|_2
    = \| \mathcal{E}_{rFCF}^{n_\ell, \Delta} \|_2,
    \label{r-convergence-factor-eq}
\end{equation}
where $\mathbf{e}_{i+1}^\Delta$ is the error on level $1$ or equivalently, error at C-points on level $0$. With,
\begin{equation}
    \begin{split}
        \mathbf{e}_{i+1}^\Delta &= \mathcal{E}_{rFCF}^{n_\ell,\Delta} \mathbf{e}_i^\Delta = \mathcal{E}_{rFCF}^{n_\ell,\Delta} R_{I_0} \mathbf{e}_i, \\
        &\Leftrightarrow \quad P_0 \mathbf{e}_{i+1}^\Delta = P_0 \mathcal{E}_{rFCF}^{n_\ell,\Delta} \mathbf{e}_i^\Delta = P_0 \mathcal{E}_{rFCF}^{n_\ell,\Delta} R_{I_0} \mathbf{e}_i,
    \end{split}
\end{equation}
we can identify,
$    \mathcal{E}_{rFCF}^{n_\ell,\Delta} = R_{I_0} \mathcal{E}_{rFCF}^{n_\ell} P_0$,
which is a generalization of the approach in \cite{DobrevKolevPeterssonSchroder2017},
where the operators $P_0$ and $R_{I_0}$ are pulled out to the left and right of the error propagator.
Thus, in general we analyze the error propagator on level $1$ to bound residual propagation on level $0$,
as given in \eqref{r-convergence-factor-eq}.

This raises the question of how the error develops at the F-points on the fine grid.
Considering error propagation on level $0$ over $i$ iterations,
\begin{equation}
    \begin{split}
        \mathbf{e}_{i+1}
        &= \mathcal{E}_{rFCF}^{n_\ell} \mathbf{e}_{i}
        = \ldots
        = \left( \mathcal{E}_{rFCF}^{n_\ell} \right)^{i+1} \mathbf{e}_0 \\
        &= \left( P_0 \mathcal{E}_{rFCF}^{n_\ell,\Delta} R_{I_0} \right)^{i+1} \mathbf{e}_0
        = P_0 \left( \mathcal{E}_{rFCF}^{n_\ell,\Delta} \right)^{i+1} R_{I_0} \mathbf{e}_0,
    \end{split}
\end{equation}
we find that error propagation at the F-points of the fine grid can be bounded by error propagation
at the respective C-points times a constant.
\begin{lemma}
    Error propagation on level $0$ for an MGRIT V-cycle method can be bounded by error propagation
    on level $1$,
    \begin{equation}
        \| \mathcal{E}_{rFCF}^{n_\ell} \|_2 \leq \sqrt{m_0} \| \mathcal{E}_{rFCF}^{n_\ell,\Delta} \|_2,
    \end{equation}
    with temporal coarsening factor $m_0$ on level $0$.
    \label{P0-bound-lemma}
\end{lemma}
\begin{proof}    This follows from,
    \begin{alignat*}{4}
        \| \mathcal{E}_{rFCF}^{n_\ell} \|_2
        &= \| P_0 \mathcal{E}_{rFCF}^{n_\ell,\Delta} R_{I_0} \|_2
        \leq \| P_0 \|_2 \| \mathcal{E}_{rFCF}^{n_\ell,\Delta} \|_2 \| R_{I_0} \|_2 \\
        &\leq \sqrt{\| P_0 \|_1 \| P_0 \|_\infty} \| \mathcal{E}_{rFCF}^{n_\ell,\Delta} \|_2 \sqrt{\| R_{I_0} \|_1 \| R_{I_0} \|_\infty}
        \leq \sqrt{m_0} \| \mathcal{E}_{rFCF}^{n_\ell,\Delta} \|_2,
    \end{alignat*}
    with submultiplicativity and the inequality
    $ \| D \|_2 \leq \sqrt{ \| D \|_1 \| D \|_\infty }$ (see \cite{Higham1992}).
\end{proof}

\begin{remark}
    It is clear from Lemma \ref{P0-bound-lemma}, that if the error at C-points on level $0$ converges,
    then the error at F-points on level $0$ converges as well.
    This is the basis for the theory developed in the rest of this work,
    where convergence is attained by bounding $\mathcal{E}_{rFCF}^{n_\ell,\Delta} $ in norm.
\end{remark}

Lemma \ref{P0-bound-lemma} is intuitive in the sense that the fine grid error propagation
is a direct result of the level $1$ error propagation; it is simply ideal interpolation applied to the level $1$ error;
that is, the operator $P_0$ propagates the error at the C-points on level $0$ to the subsequent F-points.
A similar result was presented in \cite{Southworth2018} for two-level convergence of Parareal
and MGRIT.

Based on the formulae derived in Section \ref{residual-and-error-propagation-sec}, we can construct
residual and error propagators numerically and bound the \emph{worst case} convergence factor
of MGRIT (\emph{a priori}) from above via
\begin{align}
    c_f = \max_i \| \mathbf{r}_{i+1} \|_2 / \| \mathbf{r}_i \|_2
    \leq \| \mathcal{E}_{rFCF}^{n_\ell, \Delta} \|_2,
    \label{error-l2-upper-bound-eqn}
\end{align}
which corresponds to the maximum singular value of $\mathcal{E}_{rFCF}^{n_\ell, \Delta}$.
In practice, the dimension of $\mathcal{E}_{rFCF}^{n_\ell}$ grows with the problem size in space and time,
$N_x$ and $N_0$. Similarly, $\mathcal{E}_{rFCF}^{n_\ell,\Delta}$ grows with $N_x$ and $N_1$.
Depending on the available resources, numerical construction and investigation
of these operators may be limited by memory consumption and/or compute time.
To that end, it is desirable to derive further \emph{cheaper} upper bounds that enable fast assessment
of MGRIT convergence for larger space-time problem sizes. In the following,
we present several \emph{a priori} bounds
and approximate convergence factors for fine-grid residual propagation and error propagation on level~$1$.
\subsection{Upper bound using inequality}\label{upper-bound-using-inequality-sec}

One straightforward way to reduce computational cost by roughly one order of magnitude is bounding
the $\ell^2$-norm of the error propagator on level $1$ using the well-known inequality \cite{Higham1992},
\begin{equation}
    \| \mathcal{E}_{rFCF}^{n_\ell,\Delta} \|_2^2
    \leq \| \mathcal{E}_{rFCF}^{n_\ell,\Delta} \|_1 \| \mathcal{E}_{rFCF}^{n_\ell,\Delta} \|_\infty.
    \label{inequality-eq}
\end{equation}
In \cite{DobrevKolevPeterssonSchroder2017}, this was used to develop an upper-bound on
two-grid convergence, which was proven to be sharp in \cite{Southworth2018}.
This section extends this approach to three and four grid levels based on analytic formulae.
Although the sharpness of \eqref{inequality-eq} suffers as the number of levels increases (see Section
\ref{numerical-results-sec}), we show that it is still reasonably sharp and provides a useful tool to
analyze MGRIT convergence a priori.

\subsubsection{Two-level MGRIT with \lowercase{$r$}FCF-relaxation}
\label{two-level-MGRIT-with-rFCF-relaxation-sec}

The coarse-grid error propagator follows from Equation \eqref{E-rFCF-nl2-eqn} with $n_\ell = 2$
(see \ref{two-level-MGRIT-with-rFCF-relaxation-appsec}),
\begin{equation}
\resizebox{0.91\textwidth}{!}{$
    \mathcal{E}_{rFCF}^{n_\ell = 2, \Delta}
    = \begin{bmatrix}
        0 \\
        \vdots \\
        0 \\
        (\Phi_0^{m_0} - \Phi_1) \Phi_0^{r m_0} \\
        \Phi_1 (\Phi_0^{m_0} - \Phi_1) \Phi_0^{r m_0} & (\Phi_0^{m_0} - \Phi_1) \Phi_0^{r m_0} \\
        \vdots \\
        \Phi_1^{N_1 - r - 2} (\Phi_0^{m_0} - \Phi_1) \Phi_0^{r m_0} & \Phi_1^{N_1 - r - 3} (\Phi_0^{m_0} - \Phi_1) \Phi_0^{r m_0} & \cdots & (\Phi_0^{m_0} - \Phi_1) \Phi_0^{r m_0} & 0 & \cdots & 0
    \end{bmatrix},
$}
    \label{E-rFCF-nl2-evaluated-eqn}
\end{equation}
where the first $r+1$ rows and last $r+1$ columns are zero.
\begin{theorem}
    Let $\{ \lambda_{\ell,k} \}$ be the eigenvalues of $\{ \Phi_\ell \}$.
    Then, the \emph{worst case} convergence factor of the fine-grid residual
    of two-level MGRIT with $r$FCF-relaxation is bounded by
    \begin{equation}
        c_f \leq \max_{1 \leq k \leq N_x} | \lambda_{0,k}^{m_0} - \lambda_{1,k} | | \lambda_{0,k} |^{r m_0} \frac{1 - | \lambda_{1,k} |^{N_1 - 1 - r}}{1 - | \lambda_{1,k} |}.
        \label{E-rFCF-nl2-inequality-eqn}
    \end{equation}
    \label{cf-2-level-rFCF-theo}
\end{theorem}
\begin{proof}    This follows from Equation \eqref{E-rFCF-nl2-evaluated-eqn} and inequality \eqref{inequality-eq},
    \begin{alignat*}{4}
        \| \mathcal{E}_{rFCF}^{n_\ell = 2, \Delta} \|_2
        &\leq \sqrt{\| \mathcal{E}_{rFCF}^{n_\ell = 2, \Delta} \|_1 \| \mathcal{E}_{rFCF}^{n_\ell = 2, \Delta} \|_\infty }
        = \| \mathcal{E}_{rFCF}^{n_\ell = 2} \|_1 \\
        &= \max_{1 \leq k \leq N_x} | \lambda_{0,k}^{m_0} - \lambda_{1,k} | | \lambda_{0,k} |^{r m_0} \frac{1 - | \lambda_{1,k} |^{N_1 - 1 - r}}{1 - | \lambda_{1,k} |},
    \end{alignat*}
    where the relationship $\sum_{i = 0}^{N - 1} a^i = (1 - a^N) / (1 - a)$ was used.
\end{proof}

\begin{remark}
The cases of F- and FCF-relaxation (i.e. $r = 0$ and $r = 1$),
yield the result in \cite{DobrevKolevPeterssonSchroder2017},
\begin{alignat*}{4}
    \| \mathcal{E}_{F}^{n_\ell = 2, \Delta} \|_2
    &\leq \max_{1 \leq k \leq N_x} | \lambda_{0,k}^{m_0} - \lambda_{1,k} | \frac{1 - | \lambda_{1,k} |^{N_1 - 1}}{1 - | \lambda_{1,k} |},
    \\
    \| \mathcal{E}_{FCF}^{n_\ell = 2, \Delta} \|_2
    &\leq \max_{1 \leq k \leq N_x} | \lambda_{0,k}^{m_0} - \lambda_{1,k} | | \lambda_{0,k} |^{m_0} \frac{1 - | \lambda_{1,k} |^{N_1 - 2}}{1 - | \lambda_{1,k} |}.
\end{alignat*}
    In \cite{Southworth2018}, it was shown that the bound in Theorem \ref{cf-2-level-rFCF-theo} is exact to $O (1 / N_1)$ for F- and FCF-relaxation.
\end{remark}

An interesting observation of \eqref{E-rFCF-nl2-evaluated-eqn} is the fact that the coarse-grid error propagator
is nilpotent and that each block can be diagonalized by the same unitary transformation.
This implies that we can re-order the rows and columns of the coarse-grid error propagator,
yielding a block diagonal form with lower triangular nilpotent blocks.

\begin{lemma}
    Let $\{\Phi_\ell\}$ be simultaneously diagonalizable
    by the same unitary transformation,
    with eigenvalues $\{\lambda_{\ell,k}\}$, such that $| \lambda_{\ell,k} | < 1$.
    Then, the $\ell^2$-norm of the coarse-grid error propagator of two-level MGRIT with $r$FCF-relaxation
    can be computed as,
        \begin{equation}
        \| \mathcal{E}_{rFCF}^{n_\ell = 2, \Delta} \|_2
        = \sup_{1 \leq k \leq N_x} \| \mathcal{\tilde E}_{rFCF}^{n_\ell = 2, \Delta} (k) \|_2,
    \end{equation}
        with the coarse-grid error propagator $\mathcal{\tilde E}_{rFCF}^{n_\ell = 2,\Delta} (k)$ for a single spatial mode $k$
    with $1 \leq k \leq N_x$.
    \label{individual-modes-lem}
    \begin{proof}
        This follows from the discussion above and the fact that the spectral norm of a block diagonal operator
        with lower triangular blocks can be computed as the supremum of the spectral norm of all lower triangular
        blocks. See also \cite{DobrevKolevPeterssonSchroder2017}, Remark 3.1.
    \end{proof}
\end{lemma}

\begin{remark}
    Lemma \ref{individual-modes-lem} implies that computing a bound of the form of
    \eqref{error-l2-upper-bound-eqn} can be parallelized over the number of spatial modes.
    Thus, the time complexity of evaluating \eqref{error-l2-upper-bound-eqn} is $O(N_x N_1^3 / p)$
    with $1 \leq p \leq N_x$ parallel processors.
\end{remark}

\begin{remark}
    Lemma \ref{individual-modes-lem} formalizes and generalizes the discussion
    for two-level MGRIT with F- and FCF-relaxation in \cite[Section 4.2]{FriedhoffMaclachlan2015}.
\end{remark}

\begin{remark}
    The result in Lemma \ref{individual-modes-lem} is not limited to $n_\ell = 2$ and can be applied
    to all subsequent convergence results.
\end{remark}
\subsubsection{Three-level V-cycles with F-relaxation}\label{three-level-V-cycles-with-F-relaxation-sec}

Evaluating the error propagator in Equation \eqref{E-F-nl-eqn} for a three-level V-cycle with F-relaxation
on level $1$ (see Equation \eqref{E-F-3-level-evaluated-eq}) and comparison
with the two-level error propagator for F-relaxation in \cite{DobrevKolevPeterssonSchroder2017}
highlights a slight complication:
In general, the maximum absolute column sum (and similarly, for the maximum absolute row sum) is no longer
given by the first column.\footnote{Note, that additional relaxation steps did not break symmetry
of $\mathcal{E}_{rFCF}^{n_l = 2,\Delta}$ in Equation \eqref{E-rFCF-nl2-evaluated-eqn}.}
Instead, the maximum absolute column sum is given by the maximum of the first $m_1$ absolute column sums,
corresponding to the first CF-interval (first C-point and first $m_1 - 1$ F-points) on level $1$.
This structure arises because of the recursive partitioning of time points into F- and C-points on each level.
\begin{theorem}
    Let $\{ \lambda_{\ell,k} \}$ be the eigenvalues of $\{ \Phi_\ell \}$.
    Then, the \emph{worst case} convergence factor of three-level MGRIT with F-relaxation
    is bounded by
    \begin{equation}
        c_f \leq \sqrt{\| \mathcal{E}_{F}^{n_\ell = 3, \Delta} \|_1 \| \mathcal{E}_{F}^{n_\ell = 3, \Delta} \|_\infty }.
        \label{E-F-nl3-inequality-eqn}
    \end{equation}
    and $\| \mathcal{E}_{F}^{n_\ell = 3, \Delta} \|_1$ and $\| \mathcal{E}_{F}^{n_\ell = 3, \Delta} \|_\infty$
    are given analytically as,
    \begin{alignat}{4}
        \| \mathcal{E}_{F}^{n_\ell = 3, \Delta} \|_1
        &=
\resizebox{0.7\hsize}{!}{$
        \max_{1 \leq k \leq N_x} \begin{cases}
            | \lambda_{2,k} - \lambda_{0,k}^{m_0} \lambda_{1,k}^{m_1 - 1} |
            \left(
            | \lambda_{2,k} |^{N_2 - 2}
            + \frac{1 - | \lambda_{2,k} |^{N_2 - 2}}{1 - | \lambda_{2,k} |}
            \frac{1 - | \lambda_{1,k} |^{m_1}}{1 - | \lambda_{1,k} |}
            \right) \\
            \qquad + | \lambda_{1,k} - \lambda_{0,k}^{m_0} | \frac{1 - | \lambda_{1,k} |^{m_1 - 1}}{1 - | \lambda_{1,k} |}    \\[1ex]\hdashline\vspace{-0.3cm}\\
            | \lambda_{1,k} |^{j - 1}
            | \lambda_{1,k} - \lambda_{0,k}^{m_0} |
            \left[
            | \lambda_{2,k} |^{N_2 - 2}
            + \frac{1 - | \lambda_{2,k} |^{N_2 - 2}}{1 - | \lambda_{2,k} |}
            \frac{1 - | \lambda_{1,k} |^{m_1}}{1 - | \lambda_{1,k} |}
            \right] \\
            \qquad + | \lambda_{1,k} - \lambda_{0,k}^{m_0} | \frac{1 - | \lambda_{1,k} |^{m_1 - 2}}{1 - | \lambda_{1,k} |} \qquad \text{for } j = 1, \ldots, m_1 - 1,
        \end{cases}
        $}
    \end{alignat}
    and
    \begin{alignat}{4}
        \| \mathcal{E}_{F}^{n_\ell = 3, \Delta} \|_\infty
        &=
\resizebox{0.7\hsize}{!}{$
             \max_{1 \leq k \leq N_x} \begin{cases}
            | \lambda_{2,k} - \lambda_{0,k}^{m_0} \lambda_{1,k}^{m_1 - 1} |
            \frac{1 - | \lambda_{2,k} |^{N_2 - 1}}{1 - | \lambda_{2,k} |}
            \\
            \qquad + | \lambda_{1,k} - \lambda_{0,k}^{m_0} |
            \frac{1 - | \lambda_{2,k} |^{N_2 - 1}}{1 - | \lambda_{2,k} |}
            \frac{1 - | \lambda_{1,k} |^{m_1 - 1}}{1 - | \lambda_{1,k} |}
                \\[1ex]\hdashline\vspace{-0.3cm}\\
            | \lambda_{1,k} - \lambda_{0,k}^{m_0} |
            \frac{1 - | \lambda_{1,k} |^j}{1 - | \lambda_{1,k} |}
            + | \lambda_{1,k} |^j
            \frac{1 - | \lambda_{2,k} |^{N_2 - 2}}{1 - | \lambda_{2,k} |}
            | \lambda_{2,k} - \lambda_{0,k}^{m_0} \lambda_{1,k}^{m_1 - 1} | \\
            \qquad + | \lambda_{1,k} |^j
            \frac{1 - | \lambda_{2,k} |^{N_2 - 2}}{1 - | \lambda_{2,k} |}
            | \lambda_{1,k} - \lambda_{0,k}^{m_0} |
            \frac{1 - | \lambda_{1,k} |^{m_1 - 1}}{1 - | \lambda_{1,k} |}
            \qquad \text{for } j = 1, \ldots, m_1 - 1.
        \end{cases}
        $}
    \end{alignat}
        \label{cf-3-level-F-theo}
\end{theorem}
\begin{proof}    The proof is analogous to Theorem \ref{cf-2-level-rFCF-theo}.
\end{proof}

The benefit of Theorem \ref{cf-3-level-F-theo} is that evaluating the $2 m_1$ analytic formulae
is significantly cheaper than constructing $\mathcal{E}_F^{n_\ell = 3, \Delta}$ numerically and
directly computing $\| \mathcal{E}_F^{n_\ell = 3, \Delta} \|_1$ and $\| \mathcal{E}_F^{n_\ell = 3, \Delta} \|_\infty$.
    \label{time-complexity-analytic-formulae-three-level-V-F-relax-remark}

\subsubsection{Analytic formulae for other cases}\label{analytic-formulae-for-other-cases-sec}

Analogous to Section \ref{two-level-MGRIT-with-rFCF-relaxation-sec}
and Section \ref{three-level-V-cycles-with-F-relaxation-sec},
analytic formulae for a four-level V-cycle with F-relaxation
and a three-level V-cycle with FCF-relaxation are derived
in Section \ref{four-level-V-cycles-with-F-relaxation-appsec}
and \ref{four-level-V-cycles-with-F-relaxation-appsec}.

\subsection{Approximate convergence factor of multilevel V-cycle algorithm}
\label{approximate-convergence-factor-of-multilevel-MGRIT-sec}

Section \ref{upper-bound-using-inequality-sec} presented analytic formulae for the inequality
bound \eqref{inequality-eq} as the maximum of a certain function over eigenvalues of $\{\Phi_\ell\}$.
These \emph{a priori} convergence bounds reduce memory consumption and computational cost significantly.
It is, however, increasingly difficult to derive such analytic formulae for larger numbers of levels.
Here, we propose an analytic approximate convergence factor for multilevel V-cycles with F- and FCF-relaxation,
as a function of eigenvalues of the time stepping operators $\{\lambda_{\ell,k}\}$, number of time points $\{N_\ell\}$,
and temporal coarsening factors for each level, $\{m_\ell\}$.
This yields approximate \emph{a priori} convergence factors with linear memory
and time complexity.\footnote{The generalization of Lemma \ref{individual-modes-lem} implies
that time complexity is in fact $O(\nicefrac{N_x}{p})$ with $1 \leq p \leq N_x$ parallel processors.}

The proposed approximate convergence factors are based on approximating the inequality bound
\eqref{inequality-eq}, and therefore, are expected to be a conservative upper bound in a large number of
cases. More specifically, in the case of multilevel V-cycles with F-relaxation
the approximate convergence factor is derived by identifying the recursive structure
in the analytic formulae for two, three and four levels (see \eqref{E-rFCF-nl2-inequality-eqn},
\eqref{E-F-nl3-inequality-eqn}, and \eqref{E-F-nl4-inequality-eqn}) and estimating how this recursion
continues for $n_l > 4$ levels (and similarly for FCF-relaxation with \eqref{E-rFCF-nl2-inequality-eqn}
and \eqref{E-FCF-nl3-inequality-eqn}).

First, we present the approximate convergence factor for multilevel V-cycles with F-relaxation.
\begin{approximaterate}
        Let $\{ \lambda_{\ell,k} \}$ be the eigenvalues of $\{ \Phi_\ell \}$.
    Then, an approximate worst-case convergence factor of multilevel MGRIT V-cycles with
    F-relaxation is given by
        \begin{alignat}{4}
        \tilde{c}_{f,F}
        &:= \max_{1 \leq k \leq N_x} \sqrt{ s_0^{\text{row}} (k, n_\ell) ~ s_{N_1 - 1}^{\text{col}} (k, n_\ell) }
        \approx \sqrt{\| \mathcal{E}_{F}^{n_\ell, \Delta} \|_1 \| \mathcal{E}_{F}^{n_\ell, \Delta} \|_\infty },
        \label{E-F-nl-approximate-inequality-eqn}
    \end{alignat}
        with approximate maximum absolute column and row sums
        \begin{alignat*}{4}
        &\quad s_0^{\text{col}} (k, n_\ell) \approx \sum_{\ell = 1}^{n_\ell - 1}
        \left| \lambda_{\ell,k} - \lambda_{0,k}^{m_0} \left( \prod_{p = 1}^{l - 1} \lambda_{p,k}^{\tilde{m}_p - 1} \right) \right|
        \left( \prod_{q = 1}^l \frac{1 - | \lambda_{q,k} |^{\tilde{m}_q - 1}}{1 - | \lambda_{q,k} |} \right) \numberthis \\
        &\qquad + (n_\ell > 2)
        \cdot | \lambda_{n_\ell - 1,k} |^{\tilde{m}_{n_\ell - 1} - 1}
        \left| \lambda_{n_\ell - 1,k} - \lambda_{0,k}^{m_0} \left( \prod_{p = 1}^{n_\ell - 2} \lambda_{p,k}^{\tilde{m}_p - 1} \right) \right|,    \\
        &s_{N_1 - 1}^{\text{row}} (k, n_\ell) \approx\sum_{\ell = 1}^{n_\ell - 1}
        \left| \lambda_{\ell,k} - \lambda_{0,k}^{m_0} \left( \prod_{p = 1}^{l - 1} \lambda_{p,k}^{\tilde{m}_p - 1} \right) \right|
        \left( \prod_{q = l}^{n_\ell - 1} \frac{1 - | \lambda_{q,k} |^{\tilde{m}_q}}{1 - | \lambda_{q,k} |} \right), \numberthis
    \end{alignat*}
        for $\tilde{m}_\ell = [m_0, \ldots, m_{n_\ell - 2}, N_{n_\ell - 1} - 1]^T$.  In many cases,
    $\| \mathcal{E}_{F}^{n_\ell, \Delta} \|_2 \leq \tilde{c}_{f,F}$
    because $\tilde{c}_{f,F}$ directly approximate an upper bound on $\| \mathcal{E}_{F}^{n_\ell, \Delta} \|_2$ \eqref{inequality-eq}.
        \label{cf-nl-F-app}
\end{approximaterate}
A similar result can be formulated for multilevel V-cycles with FCF-relaxation.
\begin{approximaterate}
        Let $\{ \lambda_{\ell,k} \}$ be the eigenvalues of $\{ \Phi_\ell \}$.
        Then, an approximate worst-case convergence factor of multilevel MGRIT V-cycles with FCF-relaxation is given by
        \begin{alignat}{4}
        \tilde{c}_{f,FCF}
        &:= \max_{1 \leq k \leq N_x} \sqrt{ s_0^{\text{row}} (k, n_\ell) ~ s_{N_1 - 1}^{\text{col}} (k, n_\ell) }
        \approx \sqrt{\| \mathcal{E}_{FCF}^{n_\ell, \Delta} \|_1 \| \mathcal{E}_{FCF}^{n_\ell, \Delta} \|_\infty },
        \label{E-FCF-nl-approximate-inequality-eqn}
    \end{alignat}
        with approximate maximum absolute column and row sum,
        \begin{alignat*}{4}
        &s_0^{\text{col}} (k, n_\ell)
        \approx (n_\ell > 2)
        \cdot | \lambda_{0,k} |^{m_0}
        | \lambda_{1,k} - \lambda_{0,k}^{m_0} |
        \frac{1 - | \lambda_{1,k} |^{m_1}}{1 - | \lambda_{1,k} |}    \\
        &\quad
\resizebox{0.95\hsize}{!}{$
        + \frac{1}{n_\ell - 1}
        | \lambda_{0,k} |^{m_0}
        \left[
        \sum_{p = 2}^{n_\ell - 2}
        \left( \prod_{j = 1}^{p - 1} | \lambda_{j,k} | \right)
        \left| \lambda_{p,k} - \lambda_{0,k}^{m_0} \left( \prod_{j = 1}^{p - 1} \lambda_{j,k}^{m_j - 1} \right) \right|
        \left( \prod_{j = 1}^p \frac{1 - | \lambda_{j,k} |^{m_j}}{1 - | \lambda_{j,k} |} \right)
        \right] $}    \\
        &\quad
\resizebox{0.95\hsize}{!}{$
        + \frac{1}{n_\ell - 1}
        \frac{1 - | \lambda_{n_\ell-1,k} |^{N_{n_\ell-1} - 1}}{1 - | \lambda_{n_\ell-1,k} |}
        | \lambda_{0,k} |
        \left( \prod_{j = 0}^{n_\ell - 2} | \lambda_{j,k} |^{m_j - 1} \right)
        \left( \prod_{j = 1}^{n_\ell - 2} \frac{1 - | \lambda_{j,k} |^{m_j}}{1 - | \lambda_{j,k} |} \right)
        | \lambda_{1,k} - \lambda_{0,k}^{m_0} | $}    \\
        &\quad
\resizebox{0.95\hsize}{!}{$
        + \frac{1}{n_\ell - 1}
        \frac{1 - | \lambda_{n_\ell-1,k} |^{N_{n_\ell-1} - 1}}{1 - | \lambda_{n_\ell-1,k} |}
        | \lambda_{0,k} |^{m_0}
        \left( \prod_{j = 1}^{n_\ell - 2} | \lambda_{j,k} | \right)
        \left( \sum_{p = 2}^{n_\ell - 1} \left| \lambda_{p,k} - \lambda_{0,k}^{m_0} \left( \prod_{j = 1}^{p - 1} \lambda_{j,k}^{m_j - 1} \right) \right| \right)
        \left( \prod_{j = 1}^{n_\ell - 2} \frac{1 - | \lambda_{j,k} |^{m_j}}{1 - | \lambda_{j,k} |} \right)$},    \\
                &
\resizebox{0.975\hsize}{!}{$
        s_{N_1 - 1}^{\text{row}} (k, n_\ell)
        \approx
        | \lambda_{0,k} |^{m_0}
        \frac{1 - | \lambda_{n_\ell-1,k} |^{N_{n_\ell-1} - 1}}{1 - | \lambda_{n_\ell-1,k} |}
        \left[
        \sum_{p = 1}^{n_\ell - 1}
        \left( \prod_{j = 1}^{p - 1} | \lambda_{j,k} | \right)
        \left| \lambda_{p,k} - \lambda_{0,k}^{m_0} \left( \prod_{j = 1}^{p - 1} \lambda_{j,k}^{m_j - 1} \right) \right|
        \left( \prod_{j = p}^{n_\ell - 2} \frac{1 - | \lambda_{j,k} |^{m_j}}{1 - | \lambda_{j,k} |} \right)
        \right]$}.
    \end{alignat*}
         In many cases,
    $\| \mathcal{E}_{FCF}^{n_\ell, \Delta} \|_2 \leq \tilde{c}_{f,FCF}$,  because $\tilde{c}_{f,FCF}$ directly approximates
    an upper bound on $\| \mathcal{E}_{FCF}^{n_\ell, \Delta} \|_2$ \eqref{inequality-eq}.
       \label{cf-nl-FCF-app}
\end{approximaterate}
\section{Numerical results}\label{numerical-results-sec}
All numerical, analytic and approximate bounds on convergence from Section \ref{bounds-MGRIT-error-propagation-sec}
are implemented in MPI/C++,\footnote{Github repository: \url{github.com/XBraid/XBraid-convergence-est}.
For more details, see Supplementary Materials \ref{code-suppsec}.}
using the open-source library Armadillo \cite{SandersonCurtin2016,SandersonCurtin2018}.
In this section, we evaluate these bounds for various model problems.
Analytic formulae, e.g., for the inequality bound \eqref{inequality-eq}
are employed whenever available:
for example, for a two-, three- and four-level V-cycle with F-relaxation,
we evaluate the analytic formulae derived in Section \ref{upper-bound-using-inequality-sec},
while for more than four levels, we construct the
error propagator numerically
and directly compute its $1$-/$\infty$-norm bounds.

This section assesses how sharp the various upper bounds are and how much sharpness is
sacrificed by employing a bound that is cheaper to compute numerically. For all results, we consider
Runge-Kutta time-integration schemes \cite{HairerNorsettWanner1993,HairerWanner1996}
of orders 1-4 (Butcher tableaux provided in \ref{butcher-tableaux-appsec}).
In \cite{DobrevKolevPeterssonSchroder2017}, it was noted that in the two-level setting,
L-stable schemes seem to be better suited for parallel-in-time integration than A-stable schemes.
Here, we review this observation in the multilevel setting. We further investigate the difference between
V- and F-cycle convergence, as well as the effect of F- and FCF-relaxation.

For all cases, the number of time grids varies between two and six levels.
The fine grid is composed of $N_0 = 1025$ time points and the temporal coarsening factor is $m_\ell = 2$
between all levels. The spatial domain is two-dimensional and discretized using $11$ nodes in each coordinate
direction (grid spacing $\delta_x$).
Derived bounds and approximate convergence factors are compared with the maximum observed convergence factor
in numerical simulations, in terms of the $\ell^2$-norm of the residual (see
Equation \eqref{r-convergence-factor-eq}),
\begin{equation}
    \max_i \| \mathbf{r}_{i+1} \|_2 / \| \mathbf{r}_i \|_2.
\end{equation}
All test cases {are} implemented in MPI/C++, using the open-source libraries
Armadillo \cite{SandersonCurtin2016,SandersonCurtin2018} and XBraid \cite{XBraid}.
The absolute stopping tolerance for MGRIT is selected as $\| \mathbf{r}_{i} \|_2 < 10^{-11}$
and the initial global space-time guess is random.
\subsection{Diffusion equation}\label{diffusion-equation-sec}

Consider the general time-dependent diffusion equation in two spatial dimensions over domain
$\mathbf{x}\in\Omega = (0, 2\pi) \times (0, 2\pi)$,
\begin{align*}
    \partial_t u &= \nabla \cdot \left[ K \nabla u \right] \qquad
    \text{for } \mathbf{x} \in \Omega,~t \in (0, 2\pi],
    \label{anisotropic-diffusion-eqn}
\end{align*}
with homogeneous boundary and discontinuous initial condition (see Figure \ref{diffusion-wave-ic-suppfig}),
\begin{align*}
    u (\mathbf{x}, \cdot) &= 0 \qquad\qquad~ \text{for } \mathbf{x} \in \partial\Omega,
    \\
    u (\cdot, 0) &=
\resizebox{0.8\hsize}{!}{$
        1 - \max{
        \left\{
        \text{sign}{
        \left(
        \left(
        4 - (x_1 - \pi + 1)^2 - 4 (x_2 - \pi)^2
        \right)^2
        + 1.2 (1 + \pi - x_1)^3
        - 10
        \right)}
        , 1
        \right\}}$}
        \\
        &\qquad\qquad\qquad \text{for } \mathbf{x} \in \Omega \cup \partial \Omega, \nonumber
\end{align*}
for a scalar solution $u(\mathbf{x},t)$ and boundary $\partial \Omega$.
Here, $K = \text{diag}(k_1, k_2) = \text{const}$ is the grid-aligned conductivity tensor. If
$k_1 = k_2$, the problem is isotropic, while if $k_1 \ll k_2$ or $k_2\ll k_1$, the problem is anisotropic.
The spatial problem is discretized using second-order centered finite differences, in which case the time-stepping
operators $\Phi_\ell$ are unitarily diagonalizable.
\subsubsection{Isotropic diffusion}\label{isotropic-diffusion-sec}
First, we consider the isotropic case with $k_1 = k_2 = 10$.
The CFL number on each level,
\begin{equation*}
    \text{CFL}_\ell
    = 2 \pi / (N_\ell - 1) \left(k_1 / \delta_x^2 + k_2 / \delta_x^2 \right)
    = 4 \pi k_1 / [ \delta_x^2 (N_\ell - 1) ],
    \label{isotropic-diffusion-CFL-eq}
\end{equation*}
ranges between $\text{CFL}_0 \approx 0.376$ on level $0$ and $\text{CFL}_5 \approx 12.036$ on level $5$.
Results for F-relaxation are shown in Figure \ref{diffusion-isotropic-V-F-cycle-nt1024-F-relaxation-nn11x11-fig}
and FCF relaxation in Figure \ref{diffusion-isotropic-V-F-cycle-nt1024-FCF-relaxation-nn11x11-fig} (note the
difference in y-axis limits; results for SDIRK3 can be found in Supplementary
Figures \ref{diffusion-isotropic-V-F-cycle-nt1024-F-relaxation-nn11x11-suppfig}
and \ref{diffusion-isotropic-V-F-cycle-nt1024-FCF-relaxation-nn11x11-suppfig}).

In the case of F-relaxation, there is a considerable difference in convergence behavior between the A-stable and
L-stable Runge-Kutta schemes.
For A-stable schemes, convergence of MGRIT deteriorates with a growing number of time grid levels,
which corresponds to a growing CFL number on the coarse grid, and eventually diverges.
On the other hand, L-stable schemes show a less dramatic increase in the convergence factor.
In fact, the estimated and observed convergence factors plateau for V-cycle algorithms with L-stable time integration.
For F-cycle algorithms with F-relaxation and L-stable schemes,
observed convergence is flat for all considered time grid hierarchies and only
a slight increase can be observed in the upper bound values and approximate convergence factor.

In the case of FCF-relaxation,
all observed convergence factors for SDIRK orders 2-4 are constant with respect to number of levels,
and only a slight increase in convergence factor occurs for SDIRK1.
FCF-relaxation was shown to be a critical ingredient for a scalable multilevel solver
in \cite{FalgoutFriedhoffKolevMaclachlanSchroder2014}.
An important observation for F-cycle convergence is that all upper bounds predict constant convergence factors,
suggesting that an MGRIT algorithm with F-cycles and FCF-relaxation yields a robust and scalable multilevel solver
for the isotropic diffusion equation.

In general, all upper bounds and approximate convergence factors provide good qualitative a priori estimates
of the observed convergence. These estimates become less sharp for larger numbers of time grid levels, but
the estimates do appear to be robust across changes in time integration order. Furthermore, note that Approximation
\ref{cf-nl-F-app} and Approximation \ref{cf-nl-FCF-app} estimate observed convergence as well or better than more
expensive upper bounds, demonstrating their applicability and efficacy.
Overall, results in this section demonstrate that theoretical results presented in this work
provide a valuable tool for designing robust and scalable multilevel solvers.
It further provides guidance to avoid less optimal parameter choices for MGRIT, such as F-relaxation with
A-stable RK schemes.
\FloatBarrier
\begin{figure}[h!]
    \centering
    \hspace{-1.75cm}
    \setlength{\figurewidth}{0.3\linewidth}
    \setlength{\figureheight}{0.25\linewidth}
    \begin{subfigure}[b]{0.4\textwidth}
        \raisebox{0.5cm}{\definecolor{mycolor1}{rgb}{0.00000,0.44700,0.74100}\definecolor{mycolor2}{rgb}{0.85000,0.32500,0.09800}\definecolor{mycolor3}{rgb}{0.92900,0.69400,0.12500}\definecolor{mycolor4}{rgb}{0.49400,0.18400,0.55600}\begin{tikzpicture}

\begin{axis}[width=\figurewidth,
height=0.8\figureheight,
at={(0\figurewidth,0\figureheight)},
hide axis,
unbounded coords=jump,
xmin=2,
xmax=6,
ymin=0,
ymax=1,
axis background/.style={fill=white},
title style={font=\bfseries},
legend style={at={(-0.03,0.5)}, anchor=center, legend cell align=left, align=left, fill=none, draw=none}
]

\addlegendimage{color=mycolor1, line width=1pt, mark size=1.5pt, mark=*}
\addlegendimage{color=mycolor2, line width=1pt, mark size=1.5pt, mark=*}
\addlegendimage{color=mycolor3, line width=1pt, mark size=1.5pt, mark=*}
\addlegendimage{color=mycolor4, line width=1.5pt, mark size=2.5pt, dotted, mark=x}
\addlegendimage{color=mycolor1, line width=1pt, dashed, mark size=1.5pt, mark=square, mark options={solid, mycolor1}}
\addlegendimage{color=mycolor2, line width=1pt, dashed, mark size=1.5pt, mark=square, mark options={solid, mycolor2}}
\addlegendimage{color=mycolor3, line width=1pt, dashed, mark size=1.5pt, mark=square, mark options={solid, mycolor3}}
\addlegendimage{color=mycolor4, line width=1pt, dashed, mark size=1.5pt, mark=square, mark options={solid, mycolor4}}

\addlegendentry{~$\max_i \|\mathbf{r}_{i+1}\|_2 / \|\mathbf{r}_i\|_2$ (V-cycle)}
\addlegendentry{~$\| \mathcal{E}_{F}^{n_\ell,\Delta} \|_2$}
\addlegendentry{~$\sqrt{ \| \mathcal{E}_{F}^{n_\ell,\Delta} \|_1 \| \mathcal{E}_{F}^{n_\ell,\Delta} \|_\infty }$}
\addlegendentry{~Approximation \ref{cf-nl-F-app}}
\addlegendentry{~$\max_i \|\mathbf{r}_{i+1}\|_2 / \|\mathbf{r}_i\|_2$ (F-cycle)}
\addlegendentry{~$\| \mathcal{F}_{F}^{n_\ell,\Delta} \|_2$}
\addlegendentry{~$\sqrt{ \| \mathcal{F}_{F}^{n_\ell,\Delta} \|_1 \| \mathcal{F}_{F}^{n_\ell,\Delta} \|_\infty }$}
\end{axis}
\end{tikzpicture}}
    \end{subfigure}\qquad
    \setlength{\figurewidth}{0.4\linewidth}
    \setlength{\figureheight}{0.25\linewidth}
    \begin{subfigure}[b]{0.4\textwidth}
        \definecolor{mycolor1}{rgb}{0.00000,0.44700,0.74100}\definecolor{mycolor2}{rgb}{0.85000,0.32500,0.09800}\definecolor{mycolor3}{rgb}{0.92900,0.69400,0.12500}\definecolor{mycolor4}{rgb}{0.49400,0.18400,0.55600}\begin{tikzpicture}

\begin{axis}[width=\figurewidth,
height=0.849\figureheight,
at={(0\figurewidth,0\figureheight)},
scale only axis,
unbounded coords=jump,
xmin=2,
xmax=6,
xtick={2, 3, 4, 5, 6},
xlabel style={font=\color{white!15!black}},
xlabel={Number of levels},
ymode=log,
ymin=0.01,
ymax=10,
yminorticks=true,
ylabel style={font=\color{white!15!black}},
ylabel={Convergence factor},
axis background/.style={fill=white},
title style={font=\bfseries},
title={L-stable SDIRK1}
]
\addplot [color=mycolor1, line width=1pt, mark size=1.5pt, mark=*, mark options={solid, mycolor1}, forget plot]
  table[row sep=crcr]{2	0.123253524686953\\
3	0.228102752878238\\
4	0.317447496950488\\
5	0.390266781667733\\
6	0.442051569154287\\
};
\addplot [color=mycolor2, line width=1pt, mark size=1.5pt, mark=*, mark options={solid, mycolor2}, forget plot]
  table[row sep=crcr]{2	0.124992967091733\\
3	0.236426427029217\\
4	0.351157173676124\\
5	0.483889508273665\\
6	0.66029420903876\\
};
\addplot [color=mycolor3, line width=1pt, mark size=1.5pt, mark=*, mark options={solid, mycolor3}, forget plot]
  table[row sep=crcr]{2	0.124994752067944\\
3	0.278825308858026\\
4	0.468880619298315\\
5	0.711764295193257\\
6	1.02963197704418\\
};
\addplot [color=mycolor4, line width=1.5pt, mark size=2.5pt, dotted, mark=x, mark options={solid, mycolor4}, forget plot]
  table[row sep=crcr]{2	0.124994752067944\\
3	0.262722408265537\\
4	0.395028008992868\\
5	0.512692791330555\\
6	0.611656979743394\\
};
\addplot [color=mycolor1, dashed, line width=1pt, mark size=1.5pt, mark=square, mark options={solid, mycolor1}, forget plot]
  table[row sep=crcr]{2	0.123253524686953\\
3	0.114084822202608\\
4	0.110632091474148\\
5	0.109380524808187\\
6	0.109080728067996\\
};
\addplot [color=mycolor2, dashed, line width=1pt, mark size=1.5pt, mark=square, mark options={solid, mycolor2}, forget plot]
  table[row sep=crcr]{2	0.124992967091733\\
3	0.118819269526873\\
4	0.12146057750914\\
5	0.122980219166386\\
6	0.122540102110416\\
};
\addplot [color=mycolor3, dashed, line width=1pt, mark size=1.5pt, mark=square, mark options={solid, mycolor3}, forget plot]
  table[row sep=crcr]{2	0.124994752067944\\
3	0.131639242389536\\
4	0.143329866345795\\
5	0.152265161572694\\
6	0.157884004191816\\
};
\addplot [color=mycolor4, dashed, line width=1pt, mark size=1.5pt, mark=square, mark options={solid, mycolor4}, forget plot]
  table[row sep=crcr]{2	nan\\
3	nan\\
4	nan\\
5	nan\\
6	nan\\
};
\end{axis}
\end{tikzpicture}
    \end{subfigure}\\[1ex]
    \setlength{\figurewidth}{0.4\linewidth}
    \setlength{\figureheight}{0.25\linewidth}
    \begin{subfigure}[b]{0.4\textwidth}
        \definecolor{mycolor1}{rgb}{0.00000,0.44700,0.74100}\definecolor{mycolor2}{rgb}{0.85000,0.32500,0.09800}\definecolor{mycolor3}{rgb}{0.92900,0.69400,0.12500}\definecolor{mycolor4}{rgb}{0.49400,0.18400,0.55600}\begin{tikzpicture}

\begin{axis}[width=\figurewidth,
height=0.847\figureheight,
at={(0\figurewidth,0\figureheight)},
scale only axis,
unbounded coords=jump,
xmin=2,
xmax=6,
xtick={2, 3, 4, 5, 6},
xlabel style={font=\color{white!15!black}},
xlabel={Number of levels},
ymode=log,
ymin=0.01,
ymax=10,
yminorticks=true,
ylabel style={font=\color{white!15!black}},
ylabel={Convergence factor},
axis background/.style={fill=white},
title style={font=\bfseries},
title={A-stable SDIRK2}
]
\addplot [color=mycolor1, line width=1pt, mark size=1.5pt, mark=*, mark options={solid, mycolor1}, forget plot]
  table[row sep=crcr]{2	0.0213752257036644\\
3	0.0321260750426203\\
4	0.199561124237503\\
5	0.822192397772941\\
6	2.27738532316254\\
};
\addplot [color=mycolor2, line width=1pt, mark size=1.5pt, mark=*, mark options={solid, mycolor2}, forget plot]
  table[row sep=crcr]{2	0.023083217742097\\
3	0.0333025198027782\\
4	0.211451763796879\\
5	0.877990651706644\\
6	2.80072293866106\\
};
\addplot [color=mycolor3, line width=1pt, mark size=1.5pt, mark=*, mark options={solid, mycolor3}, forget plot]
  table[row sep=crcr]{2	0.0230832471279075\\
3	0.0405830103844399\\
4	0.245792874052059\\
5	1.14755538987387\\
6	4.17911011401904\\
};
\addplot [color=mycolor4, line width=1.5pt, mark size=2.5pt, dotted, mark=x, mark options={solid, mycolor4}, forget plot]
  table[row sep=crcr]{2	0.0230832471279075\\
3	0.0394799254589041\\
4	0.239182878803504\\
5	1.06648116348769\\
6	3.51508985361788\\
};
\addplot [color=mycolor1, dashed, line width=1pt, mark size=1.5pt, mark=square, mark options={solid, mycolor1}, forget plot]
  table[row sep=crcr]{2	0.0213752257036644\\
3	0.0213528307252153\\
4	0.0199869634255212\\
5	0.207666716394547\\
6	8.85229388008015\\
};
\addplot [color=mycolor2, dashed, line width=1pt, mark size=1.5pt, mark=square, mark options={solid, mycolor2}, forget plot]
  table[row sep=crcr]{2	0.023083217742097\\
3	0.0231001279080235\\
4	0.0267822213517374\\
5	0.339787953703544\\
6	41.8313527968614\\
};
\addplot [color=mycolor3, dashed, line width=1pt, mark size=1.5pt, mark=square, mark options={solid, mycolor3}, forget plot]
  table[row sep=crcr]{2	0.0230832471279075\\
3	0.023116908194671\\
4	0.0334686457132331\\
5	0.552370667518023\\
6	77.1790453044027\\
};
\addplot [color=mycolor4, dashed, line width=1pt, mark size=1.5pt, mark=square, mark options={solid, mycolor4}, forget plot]
  table[row sep=crcr]{2	nan\\
3	nan\\
4	nan\\
5	nan\\
6	nan\\
};
\end{axis}
\end{tikzpicture}
    \end{subfigure}\qquad\qquad\qquad
    \begin{subfigure}[b]{0.4\textwidth}
        \definecolor{mycolor1}{rgb}{0.00000,0.44700,0.74100}\definecolor{mycolor2}{rgb}{0.85000,0.32500,0.09800}\definecolor{mycolor3}{rgb}{0.92900,0.69400,0.12500}\definecolor{mycolor4}{rgb}{0.49400,0.18400,0.55600}\begin{tikzpicture}

\begin{axis}[width=\figurewidth,
height=0.849\figureheight,
at={(0\figurewidth,0\figureheight)},
scale only axis,
unbounded coords=jump,
xmin=2,
xmax=6,
xtick={2, 3, 4, 5, 6},
xlabel style={font=\color{white!15!black}},
xlabel={Number of levels},
ymode=log,
ymin=0.01,
ymax=10,
yminorticks=true,
yticklabels={,,},
ylabel style={font=\color{white!15!black}},
axis background/.style={fill=white},
title style={font=\bfseries},
title={L-stable SDIRK2}
]
\addplot [color=mycolor1, line width=1pt, mark size=1.5pt, mark=*, mark options={solid, mycolor1}, forget plot]
  table[row sep=crcr]{2	0.0695597500143883\\
3	0.166101605307977\\
4	0.250466197399407\\
5	0.296748481511212\\
6	0.316415201063143\\
};
\addplot [color=mycolor2, line width=1pt, mark size=1.5pt, mark=*, mark options={solid, mycolor2}, forget plot]
  table[row sep=crcr]{2	0.0751697083775904\\
3	0.237331843556171\\
4	0.407324979794182\\
5	0.476318491751965\\
6	0.497803221492851\\
};
\addplot [color=mycolor3, line width=1pt, mark size=1.5pt, mark=*, mark options={solid, mycolor3}, forget plot]
  table[row sep=crcr]{2	0.0751697107252279\\
3	0.290884640602893\\
4	0.566517608289803\\
5	0.776523817329218\\
6	0.942444976477013\\
};
\addplot [color=mycolor4, line width=1.5pt, mark size=2.5pt, dotted, mark=x, mark options={solid, mycolor4}, forget plot]
  table[row sep=crcr]{2	0.0751697107252279\\
3	0.290780635159006\\
4	0.559910048670137\\
5	0.755649947142099\\
6	0.879506444311686\\
};
\addplot [color=mycolor1, dashed, line width=1pt, mark size=1.5pt, mark=square, mark options={solid, mycolor1}, forget plot]
  table[row sep=crcr]{2	0.0695597500143883\\
3	0.0800776058363599\\
4	0.0869109428131518\\
5	0.089637988203308\\
6	0.0904504339946201\\
};
\addplot [color=mycolor2, dashed, line width=1pt, mark size=1.5pt, mark=square, mark options={solid, mycolor2}, forget plot]
  table[row sep=crcr]{2	0.0751697083775904\\
3	0.098258504344855\\
4	0.129718233606791\\
5	0.12647989672065\\
6	0.134839247575872\\
};
\addplot [color=mycolor3, dashed, line width=1pt, mark size=1.5pt, mark=square, mark options={solid, mycolor3}, forget plot]
  table[row sep=crcr]{2	0.0751697107252279\\
3	0.120124246346818\\
4	0.174650020373018\\
5	0.203660493368975\\
6	0.227189666593223\\
};
\addplot [color=mycolor4, dashed, line width=1pt, mark size=1.5pt, mark=square, mark options={solid, mycolor4}, forget plot]
  table[row sep=crcr]{2	nan\\
3	nan\\
4	nan\\
5	nan\\
6	nan\\
};
\end{axis}
\end{tikzpicture}
    \end{subfigure}\\[1ex]
                                    \setlength{\figurewidth}{0.4\linewidth}
    \setlength{\figureheight}{0.25\linewidth}
    \begin{subfigure}[b]{0.4\textwidth}
        \definecolor{mycolor1}{rgb}{0.00000,0.44700,0.74100}\definecolor{mycolor2}{rgb}{0.85000,0.32500,0.09800}\definecolor{mycolor3}{rgb}{0.92900,0.69400,0.12500}\definecolor{mycolor4}{rgb}{0.49400,0.18400,0.55600}\begin{tikzpicture}

\begin{axis}[width=\figurewidth,
height=0.849\figureheight,
at={(0\figurewidth,0\figureheight)},
scale only axis,
unbounded coords=jump,
xmin=2,
xmax=6,
xtick={2, 3, 4, 5, 6},
xlabel style={font=\color{white!15!black}},
xlabel={Number of levels},
ymode=log,
ymin=0.001,
ymax=10,
yminorticks=true,
ylabel style={font=\color{white!15!black}},
ylabel={Convergence factor},
axis background/.style={fill=white},
title style={font=\bfseries},
title={A-stable SDIRK4}
]
\addplot [color=mycolor1, line width=1pt, mark size=1.5pt, mark=*, mark options={solid, mycolor1}, forget plot]
  table[row sep=crcr]{2	0.0787840925076033\\
3	0.246993720186949\\
4	0.449463691312231\\
5	0.629737810441511\\
6	0.573880219725353\\
};
\addplot [color=mycolor2, line width=1pt, mark size=1.5pt, mark=*, mark options={solid, mycolor2}, forget plot]
  table[row sep=crcr]{2	0.0863662064165024\\
3	0.361955370234664\\
4	0.951941623895102\\
5	1.84138111224164\\
6	3.01450694726387\\
};
\addplot [color=mycolor3, line width=1pt, mark size=1.5pt, mark=*, mark options={solid, mycolor3}, forget plot]
  table[row sep=crcr]{2	0.0863662204430876\\
3	0.434133577086321\\
4	1.28239022106269\\
5	2.87356686995355\\
6	5.60127890206903\\
};
\addplot [color=mycolor4, line width=1.5pt, mark size=2.5pt, dotted, mark=x, mark options={solid, mycolor4}, forget plot]
  table[row sep=crcr]{2	0.0863662204430876\\
3	0.433141055865876\\
4	1.22810795517032\\
5	2.53316179803951\\
6	4.378645215548\\
};
\addplot [color=mycolor1, dashed, line width=1pt, mark size=1.5pt, mark=square, mark options={solid, mycolor1}, forget plot]
  table[row sep=crcr]{2	0.0787840925076033\\
3	0.099468400897774\\
4	0.555199898539749\\
5	1.10534144924908\\
6	3.1599053123263\\
};
\addplot [color=mycolor2, dashed, line width=1pt, mark size=1.5pt, mark=square, mark options={solid, mycolor2}, forget plot]
  table[row sep=crcr]{2	0.0863662064165024\\
3	0.154120577028314\\
4	0.634700775499768\\
5	1.54629411341763\\
6	5.1600109577354\\
};
\addplot [color=mycolor3, dashed, line width=1pt, mark size=1.5pt, mark=square, mark options={solid, mycolor3}, forget plot]
  table[row sep=crcr]{2	0.0863662204430876\\
3	0.203013871667207\\
4	0.900580937672683\\
5	2.5296465533084\\
6	9.76139556548127\\
};
\addplot [color=mycolor4, dashed, line width=1pt, mark size=1.5pt, mark=square, mark options={solid, mycolor4}, forget plot]
  table[row sep=crcr]{2	nan\\
3	nan\\
4	nan\\
5	nan\\
6	nan\\
};
\end{axis}
\end{tikzpicture}
    \end{subfigure}\qquad\qquad\qquad
    \begin{subfigure}[b]{0.4\textwidth}
        \definecolor{mycolor1}{rgb}{0.00000,0.44700,0.74100}\definecolor{mycolor2}{rgb}{0.85000,0.32500,0.09800}\definecolor{mycolor3}{rgb}{0.92900,0.69400,0.12500}\definecolor{mycolor4}{rgb}{0.49400,0.18400,0.55600}\begin{tikzpicture}

\begin{axis}[width=\figurewidth,
height=0.849\figureheight,
at={(0\figurewidth,0\figureheight)},
scale only axis,
unbounded coords=jump,
xmin=2,
xmax=6,
xtick={2, 3, 4, 5, 6},
xlabel style={font=\color{white!15!black}},
xlabel={Number of levels},
ymode=log,
ymin=0.001,
ymax=10,
yminorticks=true,
yticklabels={,,},
axis background/.style={fill=white},
title style={font=\bfseries},
title={L-stable SDIRK4}
]
\addplot [color=mycolor1, line width=1pt, mark size=1.5pt, mark=*, mark options={solid, mycolor1}, forget plot]
  table[row sep=crcr]{2	0.0077144227250298\\
3	0.0510693648925527\\
4	0.165463512996394\\
5	0.29931869549169\\
6	0.379044641634633\\
};
\addplot [color=mycolor2, line width=1pt, mark size=1.5pt, mark=*, mark options={solid, mycolor2}, forget plot]
  table[row sep=crcr]{2	0.00904014962417012\\
3	0.0558594606636554\\
4	0.175691175088398\\
5	0.316096149923038\\
6	0.407162347254685\\
};
\addplot [color=mycolor3, line width=1pt, mark size=1.5pt, mark=*, mark options={solid, mycolor3}, forget plot]
  table[row sep=crcr]{2	0.00904016977881408\\
3	0.064523188799708\\
4	0.231294806716215\\
5	0.468360183746874\\
6	0.666788743913936\\
};
\addplot [color=mycolor4, line width=1.5pt, mark size=2.5pt, dotted, mark=x, mark options={solid, mycolor4}, forget plot]
  table[row sep=crcr]{2	0.00904016977881408\\
3	0.062385771452617\\
4	0.218758972429859\\
5	0.434555063361901\\
6	0.608395419508744\\
};
\addplot [color=mycolor1, dashed, line width=1pt, mark size=1.5pt, mark=square, mark options={solid, mycolor1}, forget plot]
  table[row sep=crcr]{2	0.0077144227250298\\
3	0.00657121293229005\\
4	0.00753749275460558\\
5	0.00894728620717261\\
6	0.00893810642262411\\
};
\addplot [color=mycolor2, dashed, line width=1pt, mark size=1.5pt, mark=square, mark options={solid, mycolor2}, forget plot]
  table[row sep=crcr]{2	0.00904014962417012\\
3	0.0100528157542381\\
4	0.0138067881652422\\
5	0.0202493461430096\\
6	0.0266388598922067\\
};
\addplot [color=mycolor3, dashed, line width=1pt, mark size=1.5pt, mark=square, mark options={solid, mycolor3}, forget plot]
  table[row sep=crcr]{2	0.00904016977881408\\
3	0.0119496731289021\\
4	0.0220120806932655\\
5	0.0383963507307148\\
6	0.0551673946862427\\
};
\addplot [color=mycolor4, dashed, line width=1pt, mark size=1.5pt, mark=square, mark options={solid, mycolor4}, forget plot]
  table[row sep=crcr]{2	nan\\
3	nan\\
4	nan\\
5	nan\\
6	nan\\
};
\end{axis}
\end{tikzpicture}
    \end{subfigure}
    \caption{Isotropic diffusion: Comparison of V- and F-cycle MGRIT with F-relaxation.
        Convergence of A-stable schemes deteriorates much quicker with a growing number of time
        grid levels and V-cycle MGRIT than for L-stable schemes and V-cycle MGRIT.
        The convergence factor for L-stable schemes and F-cycle MGRIT is almost constant.}
    \label{diffusion-isotropic-V-F-cycle-nt1024-F-relaxation-nn11x11-fig}
\end{figure}
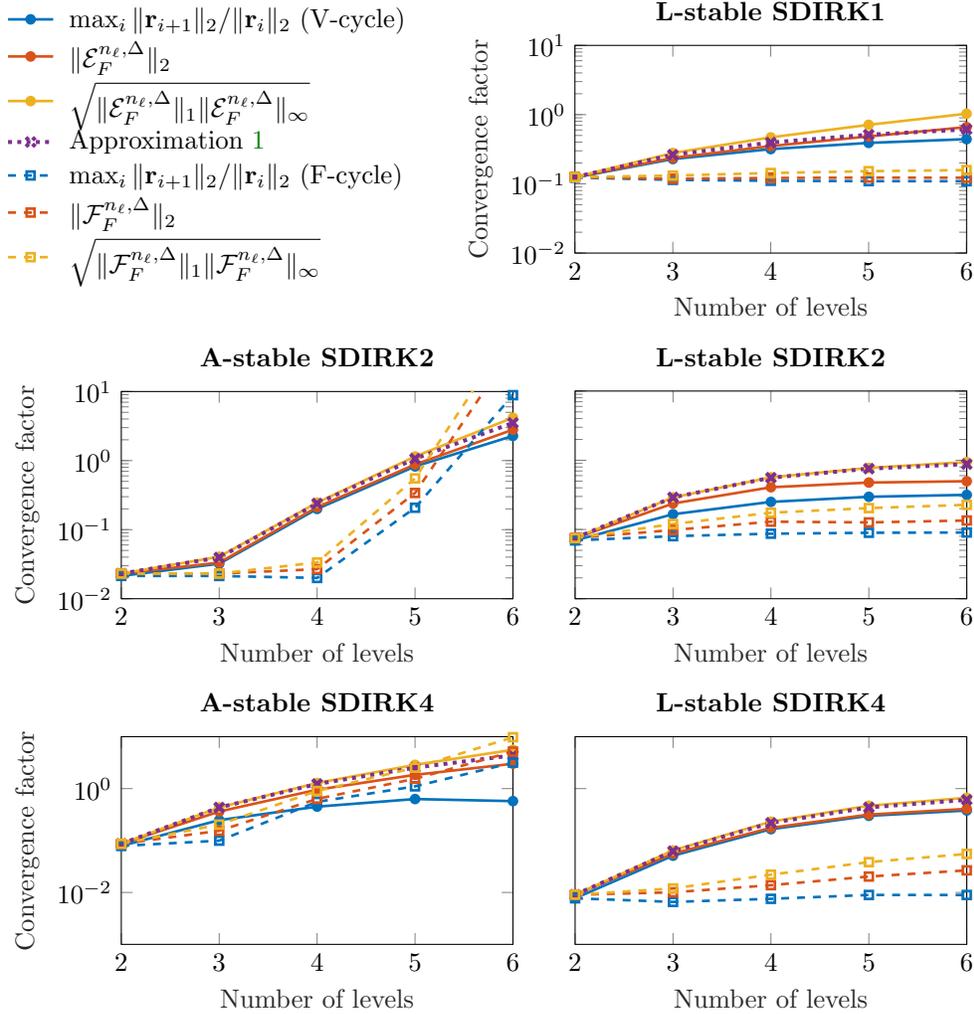
\FloatBarrier
\FloatBarrier
\begin{figure}[h!]
    \centering
    \hspace{-1.75cm}
    \setlength{\figurewidth}{0.3\linewidth}
    \setlength{\figureheight}{0.25\linewidth}
    \begin{subfigure}[b]{0.4\textwidth}
        \raisebox{0.5cm}{\definecolor{mycolor1}{rgb}{0.00000,0.44700,0.74100}\definecolor{mycolor2}{rgb}{0.85000,0.32500,0.09800}\definecolor{mycolor3}{rgb}{0.92900,0.69400,0.12500}\definecolor{mycolor4}{rgb}{0.49400,0.18400,0.55600}\begin{tikzpicture}

\begin{axis}[width=\figurewidth,
height=0.8\figureheight,
at={(0\figurewidth,0\figureheight)},
hide axis,
unbounded coords=jump,
xmin=2,
xmax=6,
ymin=0,
ymax=1,
axis background/.style={fill=white},
title style={font=\bfseries},
legend style={at={(-0.03,0.5)}, anchor=center, legend cell align=left, align=left, fill=none, draw=none}
]

\addlegendimage{color=mycolor1, line width=1pt, mark size=1.5pt, mark=*}
\addlegendimage{color=mycolor2, line width=1pt, mark size=1.5pt, mark=*}
\addlegendimage{color=mycolor3, line width=1pt, mark size=1.5pt, mark=*}
\addlegendimage{color=mycolor4, line width=1.5pt, mark size=2.5pt, dotted, mark=x}
\addlegendimage{color=mycolor1, line width=1.2pt, dashed, mark size=2.0pt, mark=square, mark options={solid, mycolor1}}
\addlegendimage{color=mycolor2, line width=1.2pt, dashed, mark size=2.0pt, mark=square, mark options={solid, mycolor2}}
\addlegendimage{color=mycolor3, line width=1.2pt, dashed, mark size=2.0pt, mark=square, mark options={solid, mycolor3}}
\addlegendimage{color=mycolor4, line width=1.2pt, dashed, mark size=2.0pt, mark=square, mark options={solid, mycolor4}}

\addlegendentry{~$\max_i \|\mathbf{r}_{i+1}\|_2 / \|\mathbf{r}_i\|_2$ (V-cycle)}
\addlegendentry{~$\| \mathcal{E}_{FCF}^{n_\ell,\Delta} \|_2$}
\addlegendentry{~$\sqrt{ \| \mathcal{E}_{FCF}^{n_\ell,\Delta} \|_1 \| \mathcal{E}_{FCF}^{n_\ell,\Delta} \|_\infty }$}
\addlegendentry{~Approximation \ref{cf-nl-FCF-app}}
\addlegendentry{~$\max_i \|\mathbf{r}_{i+1}\|_2 / \|\mathbf{r}_i\|_2$ (F-cycle)}
\addlegendentry{~$\| \mathcal{F}_{FCF}^{n_\ell,\Delta} \|_2$}
\addlegendentry{~$\sqrt{ \| \mathcal{F}_{FCF}^{n_\ell,\Delta} \|_1 \| \mathcal{F}_{FCF}^{n_\ell,\Delta} \|_\infty }$}
\end{axis}
\end{tikzpicture}}
    \end{subfigure}\qquad
    \setlength{\figurewidth}{0.4\linewidth}
    \setlength{\figureheight}{0.25\linewidth}
    \begin{subfigure}[b]{0.4\textwidth}
        \definecolor{mycolor1}{rgb}{0.00000,0.44700,0.74100}\definecolor{mycolor2}{rgb}{0.85000,0.32500,0.09800}\definecolor{mycolor3}{rgb}{0.92900,0.69400,0.12500}\definecolor{mycolor4}{rgb}{0.49400,0.18400,0.55600}\begin{tikzpicture}

\begin{axis}[width=\figurewidth,
height=0.849\figureheight,
at={(0\figurewidth,0\figureheight)},
scale only axis,
unbounded coords=jump,
xmin=2,
xmax=6,
xtick={2, 3, 4, 5, 6},
xlabel style={font=\color{white!15!black}},
xlabel={Number of levels},
ymode=log,
ymin=0.01,
ymax=1,
yminorticks=true,
ylabel style={font=\color{white!15!black}},
ylabel={Convergence factor},
axis background/.style={fill=white},
title style={font=\bfseries},
title={L-stable SDIRK1}
]
\addplot [color=mycolor1, line width=1pt, mark size=1.5pt, mark=*, mark options={solid, mycolor1}, forget plot]
  table[row sep=crcr]{2	0.0507656085880213\\
3	0.0762220161610809\\
4	0.0927548467372501\\
5	0.0955387621654756\\
6	0.105299608081205\\
};
\addplot [color=mycolor2, line width=1pt, mark size=1.5pt, mark=*, mark options={solid, mycolor2}, forget plot]
  table[row sep=crcr]{2	0.052720963365261\\
3	0.104660921371181\\
4	0.161076639878588\\
5	0.225265894986905\\
6	0.32043286103732\\
};
\addplot [color=mycolor3, line width=1pt, mark size=1.5pt, mark=*, mark options={solid, mycolor3}, forget plot]
  table[row sep=crcr]{2	0.0527247958249668\\
3	0.11844847593038\\
4	0.201417542062186\\
5	0.305775883283825\\
6	0.452874126809709\\
};
\addplot [color=mycolor4, line width=1.5pt, mark size=2.5pt, dotted, mark=x, mark options={solid, mycolor4}, forget plot]
  table[row sep=crcr]{2	0.0527247958249668\\
3	0.109520852086695\\
4	0.160873074180116\\
5	0.218529659523469\\
6	0.291222618620188\\
};
\addplot [color=mycolor1, dashed, line width=1pt, mark size=1.5pt, mark=square, mark options={solid, mycolor1}, forget plot]
  table[row sep=crcr]{2	0.0507656085880213\\
3	0.0499750896274947\\
4	0.0498901124115222\\
5	0.0498915255193069\\
6	0.0498918350190326\\
};
\addplot [color=mycolor2, dashed, line width=1pt, mark size=1.5pt, mark=square, mark options={solid, mycolor2}, forget plot]
  table[row sep=crcr]{2	0.0527209633652611\\
3	0.0518503965987351\\
4	0.0517310137613537\\
5	0.0517262247294525\\
6	0.0517261785848375\\
};
\addplot [color=mycolor3, dashed, line width=1pt, mark size=1.5pt, mark=square, mark options={solid, mycolor3}, forget plot]
  table[row sep=crcr]{2	0.0527247958249668\\
3	0.0525899597284976\\
4	0.0526591486046778\\
5	0.0526724075301763\\
6	0.0526726399335792\\
};
\addplot [color=mycolor4, dashed, line width=1pt, mark size=1.5pt, mark=square, mark options={solid, mycolor4}, forget plot]
  table[row sep=crcr]{2	nan\\
3	nan\\
4	nan\\
5	nan\\
6	nan\\
};
\end{axis}
\end{tikzpicture}
    \end{subfigure}\\[1ex]
    \setlength{\figurewidth}{0.4\linewidth}
    \setlength{\figureheight}{0.25\linewidth}
    \begin{subfigure}[b]{0.4\textwidth}
        \definecolor{mycolor1}{rgb}{0.00000,0.44700,0.74100}\definecolor{mycolor2}{rgb}{0.85000,0.32500,0.09800}\definecolor{mycolor3}{rgb}{0.92900,0.69400,0.12500}\definecolor{mycolor4}{rgb}{0.49400,0.18400,0.55600}\begin{tikzpicture}

\begin{axis}[width=\figurewidth,
height=0.849\figureheight,
at={(0\figurewidth,0\figureheight)},
scale only axis,
unbounded coords=jump,
xmin=2,
xmax=6,
xtick={2, 3, 4, 5, 6},
xlabel style={font=\color{white!15!black}},
xlabel={Number of levels},
ymode=log,
ymin=0.001,
ymax=0.1,
yminorticks=true,
ylabel style={font=\color{white!15!black}},
ylabel={Convergence factor},
axis background/.style={fill=white},
title style={font=\bfseries},
title={A-stable SDIRK2}
]
\addplot [color=mycolor1, line width=1pt, mark size=1.5pt, mark=*, mark options={solid, mycolor1}, forget plot]
  table[row sep=crcr]{2	0.00338839311398251\\
3	0.00375775453672993\\
4	0.00379400731076718\\
5	0.00401933772692649\\
6	0.00431952235024204\\
};
\addplot [color=mycolor2, line width=1pt, mark size=1.5pt, mark=*, mark options={solid, mycolor2}, forget plot]
  table[row sep=crcr]{2	0.00358908564389384\\
3	0.00618654749967819\\
4	0.0082843075215249\\
5	0.0106796241252513\\
6	0.0132568115401321\\
};
\addplot [color=mycolor3, line width=1pt, mark size=1.5pt, mark=*, mark options={solid, mycolor3}, forget plot]
  table[row sep=crcr]{2	0.00358912125421567\\
3	0.00705905541794423\\
4	0.0104958972455362\\
5	0.0146098415093414\\
6	0.0190264874019192\\
};
\addplot [color=mycolor4, line width=1.5pt, mark size=2.5pt, dotted, mark=x, mark options={solid, mycolor4}, forget plot]
  table[row sep=crcr]{2	0.00358912125421566\\
3	0.00620906901328345\\
4	0.00729492947939787\\
5	0.00853343122143058\\
6	0.00978990733423967\\
};
\addplot [color=mycolor1, dashed, line width=1pt, mark size=1.5pt, mark=square, mark options={solid, mycolor1}, forget plot]
  table[row sep=crcr]{2	0.00338839311398251\\
3	0.00338893096306971\\
4	0.00338890286254229\\
5	0.00338890286254229\\
6	0.00338890286254229\\
};
\addplot [color=mycolor2, dashed, line width=1pt, mark size=1.5pt, mark=square, mark options={solid, mycolor2}, forget plot]
  table[row sep=crcr]{2	0.00358908564389384\\
3	0.00358965214805635\\
4	0.00358965181737856\\
5	0.003589651816802\\
6	0.00358965181675937\\
};
\addplot [color=mycolor3, dashed, line width=1pt, mark size=1.5pt, mark=square, mark options={solid, mycolor3}, forget plot]
  table[row sep=crcr]{2	0.00358912125421566\\
3	0.00359013810899722\\
4	0.00359013807403896\\
5	0.0035901380742022\\
6	0.00359013807422401\\
};
\addplot [color=mycolor4, dashed, line width=1pt, mark size=1.5pt, mark=square, mark options={solid, mycolor4}, forget plot]
  table[row sep=crcr]{2	nan\\
3	nan\\
4	nan\\
5	nan\\
6	nan\\
};
\end{axis}
\end{tikzpicture}
    \end{subfigure}\qquad\qquad\qquad
    \begin{subfigure}[b]{0.4\textwidth}
        \definecolor{mycolor1}{rgb}{0.00000,0.44700,0.74100}\definecolor{mycolor2}{rgb}{0.85000,0.32500,0.09800}\definecolor{mycolor3}{rgb}{0.92900,0.69400,0.12500}\definecolor{mycolor4}{rgb}{0.49400,0.18400,0.55600}\begin{tikzpicture}

\begin{axis}[width=\figurewidth,
height=0.849\figureheight,
at={(0\figurewidth,0\figureheight)},
scale only axis,
unbounded coords=jump,
xmin=2,
xmax=6,
xtick={2, 3, 4, 5, 6},
xlabel style={font=\color{white!15!black}},
xlabel={Number of levels},
ymode=log,
ymin=0.001,
ymax=0.1,
yminorticks=true,
yticklabels={,,},
axis background/.style={fill=white},
title style={font=\bfseries},
title={L-stable SDIRK2}
]
\addplot [color=mycolor1, line width=1pt, mark size=1.5pt, mark=*, mark options={solid, mycolor1}, forget plot]
  table[row sep=crcr]{2	0.00805780361238282\\
3	0.00874145335526299\\
4	0.00856143261913144\\
5	0.00901556292441842\\
6	0.00964555796752871\\
};
\addplot [color=mycolor2, line width=1pt, mark size=1.5pt, mark=*, mark options={solid, mycolor2}, forget plot]
  table[row sep=crcr]{2	0.00843225160309777\\
3	0.0145274456119761\\
4	0.0192187315108144\\
5	0.0250238416066786\\
6	0.0316643190352325\\
};
\addplot [color=mycolor3, line width=1pt, mark size=1.5pt, mark=*, mark options={solid, mycolor3}, forget plot]
  table[row sep=crcr]{2	0.00843230356225675\\
3	0.0165225712813463\\
4	0.024541554864997\\
5	0.0342152784744969\\
6	0.0457717196989184\\
};
\addplot [color=mycolor4, line width=1.5pt, mark size=2.5pt, dotted, mark=x, mark options={solid, mycolor4}, forget plot]
  table[row sep=crcr]{2	0.00843230356225676\\
3	0.0143742695621335\\
4	0.0166147336058806\\
5	0.0194131482593446\\
6	0.0226216982757712\\
};
\addplot [color=mycolor1, dashed, line width=1pt, mark size=1.5pt, mark=square, mark options={solid, mycolor1}, forget plot]
  table[row sep=crcr]{2	0.00805780361238282\\
3	0.00805991676071779\\
4	0.00805983860159189\\
5	0.00805983860159189\\
6	0.00805983860159189\\
};
\addplot [color=mycolor2, dashed, line width=1pt, mark size=1.5pt, mark=square, mark options={solid, mycolor2}, forget plot]
  table[row sep=crcr]{2	0.00843225160309777\\
3	0.00843424784631447\\
4	0.00843425202668377\\
5	0.0084342520883057\\
6	0.00843425208892959\\
};
\addplot [color=mycolor3, dashed, line width=1pt, mark size=1.5pt, mark=square, mark options={solid, mycolor3}, forget plot]
  table[row sep=crcr]{2	0.00843230356225675\\
3	0.00843701945059571\\
4	0.0084370454421996\\
5	0.00843704664621818\\
6	0.00843704668011925\\
};
\addplot [color=mycolor4, dashed, line width=1pt, mark size=1.5pt, mark=square, mark options={solid, mycolor4}, forget plot]
  table[row sep=crcr]{2	nan\\
3	nan\\
4	nan\\
5	nan\\
6	nan\\
};
\end{axis}
\end{tikzpicture}
    \end{subfigure}\\[1ex]
                                    \setlength{\figurewidth}{0.4\linewidth}
    \setlength{\figureheight}{0.25\linewidth}
    \begin{subfigure}[b]{0.4\textwidth}
        \definecolor{mycolor1}{rgb}{0.00000,0.44700,0.74100}\definecolor{mycolor2}{rgb}{0.85000,0.32500,0.09800}\definecolor{mycolor3}{rgb}{0.92900,0.69400,0.12500}\definecolor{mycolor4}{rgb}{0.49400,0.18400,0.55600}\begin{tikzpicture}

\begin{axis}[width=\figurewidth,
height=0.849\figureheight,
at={(0\figurewidth,0\figureheight)},
scale only axis,
unbounded coords=jump,
xmin=2,
xmax=6,
xtick={2, 3, 4, 5, 6},
xlabel style={font=\color{white!15!black}},
xlabel={Number of levels},
ymode=log,
ymin=0.0001,
ymax=0.1,
yminorticks=true,
ylabel style={font=\color{white!15!black}},
ylabel={Convergence factor},
axis background/.style={fill=white},
title style={font=\bfseries},
title={A-stable SDIRK4}
]
\addplot [color=mycolor1, line width=1pt, mark size=1.5pt, mark=*, mark options={solid, mycolor1}, forget plot]
  table[row sep=crcr]{2	0.00785943192590899\\
3	0.00811177034887364\\
4	0.00808630985531714\\
5	0.00807501170270579\\
6	0.0080737537576157\\
};
\addplot [color=mycolor2, line width=1pt, mark size=1.5pt, mark=*, mark options={solid, mycolor2}, forget plot]
  table[row sep=crcr]{2	0.00825406562676913\\
3	0.0134102523698514\\
4	0.0165375365380465\\
5	0.0208047069277238\\
6	0.027246238837597\\
};
\addplot [color=mycolor3, line width=1pt, mark size=1.5pt, mark=*, mark options={solid, mycolor3}, forget plot]
  table[row sep=crcr]{2	0.008254091063939\\
3	0.0151365004296606\\
4	0.0214295902065989\\
5	0.0283046098966209\\
6	0.0383694479688892\\
};
\addplot [color=mycolor4, line width=1.5pt, mark size=2.5pt, dotted, mark=x, mark options={solid, mycolor4}, forget plot]
  table[row sep=crcr]{2	0.008254091063939\\
3	0.0133943639788404\\
4	0.0159019933698168\\
5	0.0205137925018193\\
6	0.0268587727143028\\
};
\addplot [color=mycolor1, dashed, line width=1pt, mark size=1.5pt, mark=square, mark options={solid, mycolor1}, forget plot]
  table[row sep=crcr]{2	0.00785943192590899\\
3	0.00785974461570764\\
4	0.0078597450695597\\
5	0.00785974245841552\\
6	0.00785974245841552\\
};
\addplot [color=mycolor2, dashed, line width=1pt, mark size=1.5pt, mark=square, mark options={solid, mycolor2}, forget plot]
  table[row sep=crcr]{2	0.00825406562676912\\
3	0.00825449925400225\\
4	0.00825450577423344\\
5	0.00825450604944256\\
6	0.00825450606143849\\
};
\addplot [color=mycolor3, dashed, line width=1pt, mark size=1.5pt, mark=square, mark options={solid, mycolor3}, forget plot]
  table[row sep=crcr]{2	0.008254091063939\\
3	0.00825567033808819\\
4	0.00825573484892467\\
5	0.00825575758887024\\
6	0.00825576657519351\\
};
\addplot [color=mycolor4, dashed, line width=1pt, mark size=1.5pt, mark=square, mark options={solid, mycolor4}, forget plot]
  table[row sep=crcr]{2	nan\\
3	nan\\
4	nan\\
5	nan\\
6	nan\\
};
\end{axis}
\end{tikzpicture}
    \end{subfigure}\qquad\qquad\qquad
    \begin{subfigure}[b]{0.4\textwidth}
        \definecolor{mycolor1}{rgb}{0.00000,0.44700,0.74100}\definecolor{mycolor2}{rgb}{0.85000,0.32500,0.09800}\definecolor{mycolor3}{rgb}{0.92900,0.69400,0.12500}\definecolor{mycolor4}{rgb}{0.49400,0.18400,0.55600}\begin{tikzpicture}

\begin{axis}[width=\figurewidth,
height=0.849\figureheight,
at={(0\figurewidth,0\figureheight)},
scale only axis,
unbounded coords=jump,
xmin=2,
xmax=6,
xtick={2, 3, 4, 5, 6},
xlabel style={font=\color{white!15!black}},
xlabel={Number of levels},
ymode=log,
ymin=0.0001,
ymax=0.1,
yminorticks=true,
yticklabels={,,},
axis background/.style={fill=white},
title style={font=\bfseries},
title={L-stable SDIRK4}
]
\addplot [color=mycolor1, line width=1pt, mark size=1.5pt, mark=*, mark options={solid, mycolor1}, forget plot]
  table[row sep=crcr]{2	0.00059597330930937\\
3	0.000604190470844742\\
4	0.000597479774147083\\
5	0.000592660194559346\\
6	0.000592045121194136\\
};
\addplot [color=mycolor2, line width=1pt, mark size=1.5pt, mark=*, mark options={solid, mycolor2}, forget plot]
  table[row sep=crcr]{2	0.000803016695420202\\
3	0.00128467080975186\\
4	0.00144516257928956\\
5	0.00171422344236253\\
6	0.0021725709155811\\
};
\addplot [color=mycolor3, line width=1pt, mark size=1.5pt, mark=*, mark options={solid, mycolor3}, forget plot]
  table[row sep=crcr]{2	0.000803018492711753\\
3	0.00140175740862788\\
4	0.00180486905884906\\
5	0.00232471473220589\\
6	0.00310512492069396\\
};
\addplot [color=mycolor4, line width=1.5pt, mark size=2.5pt, dotted, mark=x, mark options={solid, mycolor4}, forget plot]
  table[row sep=crcr]{2	0.000803018492711753\\
3	0.00122604306427911\\
4	0.00128050341129631\\
5	0.00135310311239575\\
6	0.0014657876002492\\
};
\addplot [color=mycolor1, dashed, line width=1pt, mark size=1.5pt, mark=square, mark options={solid, mycolor1}, forget plot]
  table[row sep=crcr]{2	0.00059597330930937\\
3	0.000595968890910351\\
4	0.000595969479168915\\
5	0.000595969479168915\\
6	0.000595969479168915\\
};
\addplot [color=mycolor2, dashed, line width=1pt, mark size=1.5pt, mark=square, mark options={solid, mycolor2}, forget plot]
  table[row sep=crcr]{2	0.000803016695420203\\
3	0.000803004620631999\\
4	0.000803004554199549\\
5	0.000803004553258015\\
6	0.000803004553245021\\
};
\addplot [color=mycolor3, dashed, line width=1pt, mark size=1.5pt, mark=square, mark options={solid, mycolor3}, forget plot]
  table[row sep=crcr]{2	0.000803018492711753\\
3	0.000803041369316796\\
4	0.000803041603361639\\
5	0.000803041609266975\\
6	0.000803041609412013\\
};
\addplot [color=mycolor4, dashed, line width=1pt, mark size=1.5pt, mark=square, mark options={solid, mycolor4}, forget plot]
  table[row sep=crcr]{2	nan\\
3	nan\\
4	nan\\
5	nan\\
6	nan\\
};
\end{axis}
\end{tikzpicture}
    \end{subfigure}
    \caption{Isotropic diffusion: Comparison of V- and F-cycle MGRIT with FCF-relaxation.
        Convergence of A-stable and L-stable schemes deteriorates only slightly for an
        MGRIT V-cycle algorithm. On the other hand, the convergence factor
        for an F-cycle MGRIT algorithm is constant for all considered RK schemes and cases.}
    \label{diffusion-isotropic-V-F-cycle-nt1024-FCF-relaxation-nn11x11-fig}
\end{figure}
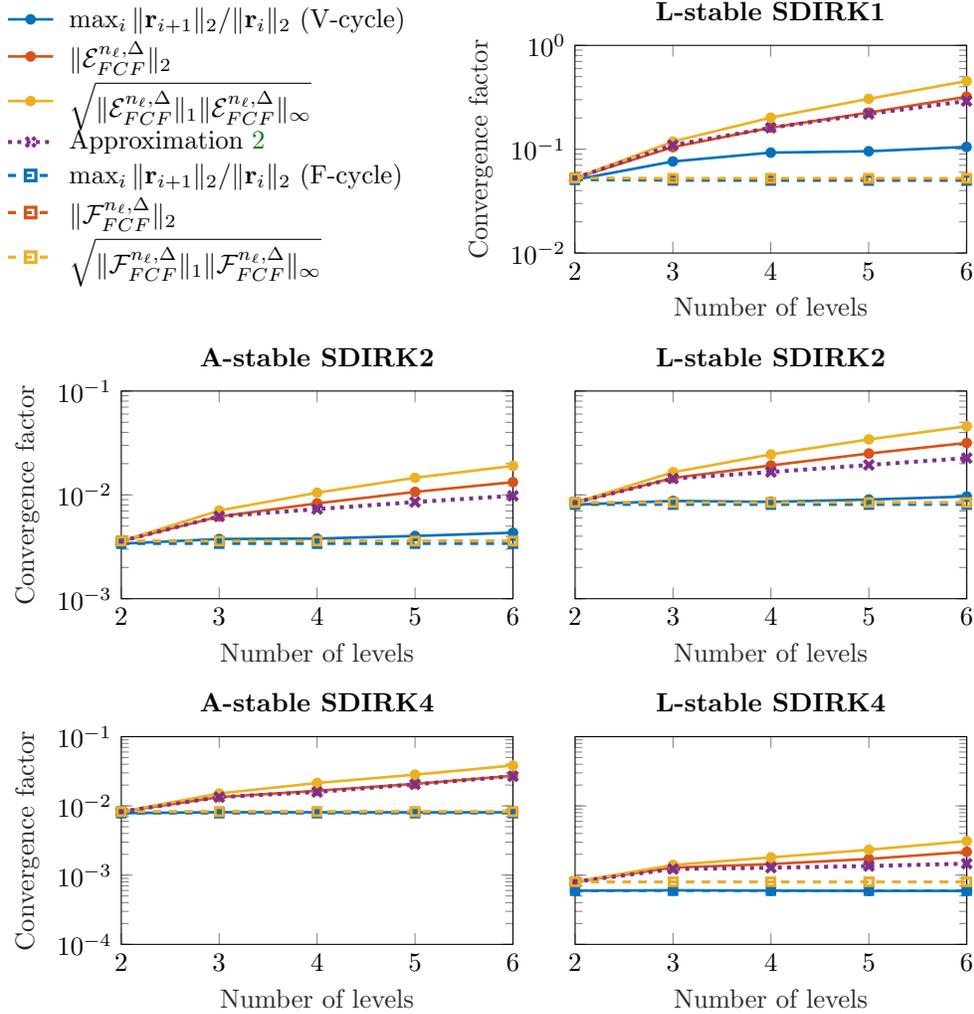
\FloatBarrier
\subsubsection{Anisotropic diffusion}
In this section, we investigate the anisotropic diffusion case for the L-stable SDIRK1 scheme (backward Euler)
to assess how sensitive the estimates are with respect to conductivity parameters.
Here, conductivity parameters are given by $k_1 = 0.5$ and $k_2 = 0.001$, and the CFL number on each level,
$    \text{CFL}_\ell
    = 2 \pi (k_1 + k_2) / [\delta_x^2 (N_\ell - 1) ],$
ranges between $\text{CFL}_0 \approx 0.009$ on level $0$ and $\text{CFL}_5 \approx 0.302$ on level $5$.
Results are presented in Supplementary Figure \ref{diffusion-anisotropic-V-F-cycle-nt1024-F-FCF-relaxation-nn11x11-suppfig}.

For V-cycle algorithms with F- and FCF-relaxation, the estimated and observed convergence factors grow
with the number of grid levels, similar to the isotropic case.
Again, FCF-relaxation yields a quicker plateauing of the observed convergence factor.
On the other hand, for F-cycle algorithms with F- and FCF-relaxtion, observed and estimated convergence
are effectively constant. This means that, for this problem, an F-cycle solves the coarse-grid problem
sufficiently accurately that residual and error reduction
is more-or-less equivalent to a two-level method.
Conversely, convergence in the case of V-cycles deteriorates due to inexact solves of the coarse-grid
problem on each level. However, the fact that solving the coarse-grid problem more accurately (such as,
with F-cycles) improves convergence, indicates that the non-Galerkin coarse-grid operator (that is, taking
larger time steps on the coarse grid using the same integration scheme) is indeed an effective preconditioner.
Note, this is in contrast to using algebraic multigrid to solve anisotropic diffusion discretizations in the spatial setting,
where stronger cycles such as F- and W-cycles often do not improve convergence \cite{ManteuffelOlsonSchroderSouthworth2017}.

The approximate bounds on convergence of F-cycles are fairly sharp for F- and FCF-relaxation and all numbers
of levels tested. In the case of V-cycles, the bounds and approximate convergence factors loose sharpness
as the number of time grid levels increases, similar to Section \ref{isotropic-diffusion-sec}, but still provide
reasonable estimates on convergence. Indeed, for V-cycles with F-relaxation, Approximation 1 is quite sharp
for all tested number of levels.
\subsection{Wave equation}

Consider the wave equation in two spatial dimensions over domain $\Omega = (0, 2\pi) \times (0, 2\pi )$,
\begin{alignat}{4}
    \partial_{tt} u &= c^2 \nabla \cdot \nabla u \qquad
    \text{for } \mathbf{x} \in \Omega,  t \in (0, 2\pi],
    \label{wave-eqn}
\end{alignat}
with scalar solution $u(\mathbf{x},t)$ and wave speed $c = \sqrt{10}$. We transform Equation \eqref{wave-eqn} into
a system of PDEs that are first-order in time,
\begin{alignat}{4}
    \partial_t u &= v, \qquad
    \partial_t v = c^2 \nabla \cdot \nabla u,\qquad
    \text{for } \mathbf{x} \in \Omega,  t \in (0, 2\pi],
    \label{wave-system-eqn}
\end{alignat}
with initial condition (see Figure \ref{diffusion-wave-ic-suppfig}) and boundary conditions,
\begin{alignat}{4}
    u (\cdot, 0) &= \sin(x) \sin(y), \quad v (\cdot, 0) = 0,\qquad &&\text{for } \mathbf{x} \in \Omega \cup \partial \Omega, \\
    u (\mathbf{x}, \cdot) &= v (\mathbf{x}, \cdot) = 0, \qquad &&\text{for } \mathbf{x} \in \partial \Omega. \label{wave-bc-eqn}
\end{alignat}

This problem corresponds to a 2D membrane with imposed non-zero initial displacement $u$
and zero initial velocity $v$. The membrane enters an oscillatory motion pattern due to initial stresses in the material.
Thus, it is a simplified representative of a hyperbolic model that shares characteristic behavior
with PDEs in solid dynamics research, such as linear elasticity \cite{HessenthalerNordslettenRoehrleSchroderFalgout2018}.
Similar to Section \ref{diffusion-equation-sec}, we use second-order centered finite differences
to discretize the spatial operator in Equation \eqref{wave-system-eqn}.
The time stepping operators $\Phi_\ell$ are then simultaneously diagonalizable and the Courant number on each level
is given by $\mathcal{C}_\ell = 2 c \pi / [\delta_x (N_\ell - 1)],$
ranging between $\mathcal{C}_0 \approx 0.034$ on level $0$ and $\mathcal{C}_5 \approx 1.087$ on level $5$.\footnote{Note,
that the Courant number is smaller than in Section \ref{isotropic-diffusion-sec}.
This is motivated by selecting a configuration that captures multiple periods of the oscillatory temporal behavior.}

An MGRIT V-cycle algorithm with FCF-relaxation shows quickly increasing convergence factors
with a growing number of time grid levels (see Figure \ref{wave-V-cycle-nt1024-FCF-relaxation-nn11x11-fig}
and Supplementary Figure \ref{wave-V-cycle-nt1024-FCF-relaxation-nn11x11-suppfig}).
The worst-case convergence factors quickly exceed $1$, and thus diverge, which is correctly predicted by all upper
bounds and Approximation \ref{cf-nl-FCF-app}.
Similarly, using an F-cycle results in a less dramatic, but still significant increase in observed and predicted
convergence factors with respect to the number of levels. For some schemes, particularly L-stable ones, an
F-cycle is able to retain convergence up to the six levels in time considered here, but the bounds and approximations
developed here do not predict these results.
\FloatBarrier
\begin{figure}[h!]
    \centering
    \setlength{\figurewidth}{0.4\linewidth}
    \setlength{\figureheight}{0.25\linewidth}
    \begin{subfigure}[b]{0.4\textwidth}
        \definecolor{mycolor1}{rgb}{0.00000,0.44700,0.74100}\definecolor{mycolor2}{rgb}{0.85000,0.32500,0.09800}\definecolor{mycolor3}{rgb}{0.92900,0.69400,0.12500}\definecolor{mycolor4}{rgb}{0.49400,0.18400,0.55600}\begin{tikzpicture}

\begin{axis}[width=\figurewidth,
height=0.849\figureheight,
at={(0\figurewidth,0\figureheight)},
scale only axis,
xmin=2,
xmax=6,
xtick={2, 3, 4, 5, 6},
xlabel style={font=\color{white!15!black}},
xlabel={Number of levels},
ymode=log,
ymin=0.01,
ymax=10,
yminorticks=true,
ylabel style={font=\color{white!15!black}},
ylabel={Convergence factor},
axis background/.style={fill=white},
title style={font=\bfseries},
title={A-stable SDIRK2}
]
\addplot [color=mycolor1, line width=1pt, mark size=1.5pt, mark=*, mark options={solid, mycolor1}, forget plot]
  table[row sep=crcr]{2	0.0126237948958803\\
3	0.0829917403735648\\
4	0.238225196343646\\
5	0.897536255881032\\
6	2.45683465444267\\
};
\addplot [color=mycolor2, line width=1pt, mark size=1.5pt, mark=*, mark options={solid, mycolor2}, forget plot]
  table[row sep=crcr]{2	0.0260990800226153\\
3	0.166092070101794\\
4	0.935281356159273\\
5	4.95670620876972\\
6	22.5251279779128\\
};
\addplot [color=mycolor3, line width=1pt, mark size=1.5pt, mark=*, mark options={solid, mycolor3}, forget plot]
  table[row sep=crcr]{2	0.0409562483157834\\
3	0.273086683554486\\
4	1.56342657983375\\
5	8.34717901296332\\
6	38.3423423076198\\
};
\addplot [color=mycolor4, line width=1.5pt, mark size=2.5pt, dotted, mark=x, mark options={solid, mycolor4}, forget plot]
  table[row sep=crcr]{2	0.0410764525880669\\
3	0.215015300138658\\
4	1.03993298906818\\
5	4.93670208755283\\
6	21.3433329825369\\
};
\end{axis}
\end{tikzpicture}    \end{subfigure}\qquad\qquad\qquad
    \begin{subfigure}[b]{0.4\textwidth}
        \definecolor{mycolor1}{rgb}{0.00000,0.44700,0.74100}\definecolor{mycolor2}{rgb}{0.85000,0.32500,0.09800}\definecolor{mycolor3}{rgb}{0.92900,0.69400,0.12500}\definecolor{mycolor4}{rgb}{0.49400,0.18400,0.55600}\begin{tikzpicture}

\begin{axis}[width=\figurewidth,
height=0.849\figureheight,
at={(0\figurewidth,0\figureheight)},
scale only axis,
xmin=2,
xmax=6,
xtick={2, 3, 4, 5, 6},
xlabel style={font=\color{white!15!black}},
xlabel={Number of levels},
ymode=log,
ymin=0.01,
ymax=10,
yminorticks=true,
yticklabels={,,},
axis background/.style={fill=white},
title style={font=\bfseries},
title={L-stable SDIRK2}
]
\addplot [color=mycolor1, line width=1pt, mark size=1.5pt, mark=*, mark options={solid, mycolor1}, forget plot]
  table[row sep=crcr]{2	0.0244843517742587\\
3	0.157756502262385\\
4	0.434494005491729\\
5	1.42660225264305\\
6	3.16908975550067\\
};
\addplot [color=mycolor2, line width=1pt, mark size=1.5pt, mark=*, mark options={solid, mycolor2}, forget plot]
  table[row sep=crcr]{2	0.0505531000941099\\
3	0.319006721670251\\
4	1.71050760210563\\
5	7.05669105132868\\
6	15.5856601797381\\
};
\addplot [color=mycolor3, line width=1pt, mark size=1.5pt, mark=*, mark options={solid, mycolor3}, forget plot]
  table[row sep=crcr]{2	0.0793185638122103\\
3	0.523940814964597\\
4	2.83541336422652\\
5	11.243491032787\\
6	21.5975279921208\\
};
\addplot [color=mycolor4, line width=1.5pt, mark size=2.5pt, dotted, mark=x, mark options={solid, mycolor4}, forget plot]
  table[row sep=crcr]{2	0.0795511657749488\\
3	0.412655086367942\\
4	1.888039464055\\
5	6.63891655230384\\
6	11.6946879181383\\
};
\end{axis}
\end{tikzpicture}    \end{subfigure}\\[1ex]
    \setlength{\figurewidth}{0.4\linewidth}
    \setlength{\figureheight}{0.25\linewidth}
    \begin{subfigure}[b]{0.4\textwidth}
        \definecolor{mycolor1}{rgb}{0.00000,0.44700,0.74100}\definecolor{mycolor2}{rgb}{0.85000,0.32500,0.09800}\definecolor{mycolor3}{rgb}{0.92900,0.69400,0.12500}\definecolor{mycolor4}{rgb}{0.49400,0.18400,0.55600}\begin{tikzpicture}

\begin{axis}[width=\figurewidth,
height=0.849\figureheight,
at={(0\figurewidth,0\figureheight)},
scale only axis,
xmin=2,
xmax=6,
xtick={2, 3, 4, 5, 6},
xlabel style={font=\color{white!15!black}},
xlabel={Number of levels},
ymode=log,
ymin=1e-07,
ymax=10,
yminorticks=true,
ylabel style={font=\color{white!15!black}},
ylabel={Convergence factor},
axis background/.style={fill=white},
title style={font=\bfseries},
title={A-SDIRK4},
legend style={at={(0.6,0)}, anchor=south, draw=none, fill=none, legend cell align=left}
]
\addplot [color=mycolor1, line width=1pt, mark size=1.5pt, mark=*, mark options={solid, mycolor1}, forget plot]
  table[row sep=crcr]{2	0.00473339681471841\\
3	0.0496473619635444\\
4	0.430149516221153\\
5	0.963588686863819\\
6	1.89448510357372\\
};
\addplot [color=mycolor2, line width=1pt, mark size=1.5pt, mark=*, mark options={solid, mycolor2}, forget plot]
  table[row sep=crcr]{2	0.00727312389153252\\
3	0.1394504667967\\
4	1.16588157962976\\
5	3.27561930133136\\
6	7.90505053733317\\
};
\addplot [color=mycolor3, line width=1pt, mark size=1.5pt, mark=*, mark options={solid, mycolor3}, forget plot]
  table[row sep=crcr]{2	0.0114075353317223\\
3	0.219741898411467\\
4	1.60928334200213\\
5	3.72715994499783\\
6	9.53557040503556\\
};
\addplot [color=mycolor4, line width=1.5pt, mark size=2.5pt, dotted, mark=x, mark options={solid, mycolor4}, forget plot]
  table[row sep=crcr]{2	0.0114409240180685\\
3	0.161758661007896\\
4	0.981989167910772\\
5	2.04295246135027\\
6	5.12541331961503\\
};

\addlegendimage{color=mycolor1, line width=1pt, mark size=1.5pt, mark=*}
\addlegendimage{color=mycolor2, line width=1pt, mark size=1.5pt, mark=*}
\addlegendimage{color=mycolor3, line width=1pt, mark size=1.5pt, mark=*}
\addlegendimage{color=mycolor4, line width=1.5pt, mark size=2.5pt, dotted, mark=x}

\addlegendentry{\footnotesize~$\max_i \|\mathbf{r}_{i+1}\|_2 / \|\mathbf{r}_i\|_2$}
\addlegendentry{\footnotesize~$\| \mathcal{E}_{FCF}^{n_\ell,\Delta} \|_2$}
\addlegendentry{\footnotesize~$\sqrt{ \| \mathcal{E}_{FCF}^{n_\ell,\Delta} \|_1 \| \mathcal{E}_{FCF}^{n_\ell,\Delta} \|_\infty }$}
\addlegendentry{\footnotesize~Approximation \ref{cf-nl-FCF-app}}

\end{axis}
\end{tikzpicture}    \end{subfigure}\qquad\qquad\qquad
    \begin{subfigure}[b]{0.4\textwidth}
        \definecolor{mycolor1}{rgb}{0.00000,0.44700,0.74100}\definecolor{mycolor2}{rgb}{0.85000,0.32500,0.09800}\definecolor{mycolor3}{rgb}{0.92900,0.69400,0.12500}\definecolor{mycolor4}{rgb}{0.49400,0.18400,0.55600}\begin{tikzpicture}

\begin{axis}[width=\figurewidth,
height=0.849\figureheight,
at={(0\figurewidth,0\figureheight)},
scale only axis,
xmin=2,
xmax=6,
xtick={2, 3, 4, 5, 6},
xlabel style={font=\color{white!15!black}},
xlabel={Number of levels},
ymode=log,
ymin=1e-07,
ymax=10,
yminorticks=true,
yticklabels={,,},
axis background/.style={fill=white},
title style={font=\bfseries},
title={L-stable SDIRK4}
]
\addplot [color=mycolor1, line width=1pt, mark size=1.5pt, mark=*, mark options={solid, mycolor1}, forget plot]
  table[row sep=crcr]{2	5.68527005677997e-07\\
3	9.77911728432672e-06\\
4	0.00462349723864827\\
5	0.0663626090077182\\
6	0.708013452181577\\
};
\addplot [color=mycolor2, line width=1pt, mark size=1.5pt, mark=*, mark options={solid, mycolor2}, forget plot]
  table[row sep=crcr]{2	3.94737354375397e-05\\
3	0.000916292932387609\\
4	0.0203700306558133\\
5	0.426149053815399\\
6	6.65825730591849\\
};
\addplot [color=mycolor3, line width=1pt, mark size=1.5pt, mark=*, mark options={solid, mycolor3}, forget plot]
  table[row sep=crcr]{2	6.19445610380176e-05\\
3	0.0014589019032384\\
4	0.032479789254962\\
5	0.679755121901639\\
6	10.6197919278232\\
};
\addplot [color=mycolor4, line width=1.5pt, mark size=2.5pt, dotted, mark=x, mark options={solid, mycolor4}, forget plot]
  table[row sep=crcr]{2	6.21263647964648e-05\\
3	0.00106978228605693\\
4	0.0196460705547786\\
5	0.362316645782796\\
6	5.26753742222048\\
};
\end{axis}
\end{tikzpicture}    \end{subfigure}
    \caption{Wave equation: The convergence factor of MGRIT with V-cycles and FCF-relaxation
        increases substantially with a growing number of time grid levels
        and eventually exceeds $1$. This means that MGRIT V-cycles will likely yield
        a divergent algorithm in practice, which is in line with observations for hyperbolic PDEs
        in the literature \cite{FarhatChandesris2003,
        FarhatCortial2006,HessenthalerNordslettenRoehrleSchroderFalgout2018}.}
    \label{wave-V-cycle-nt1024-FCF-relaxation-nn11x11-fig}
\end{figure}
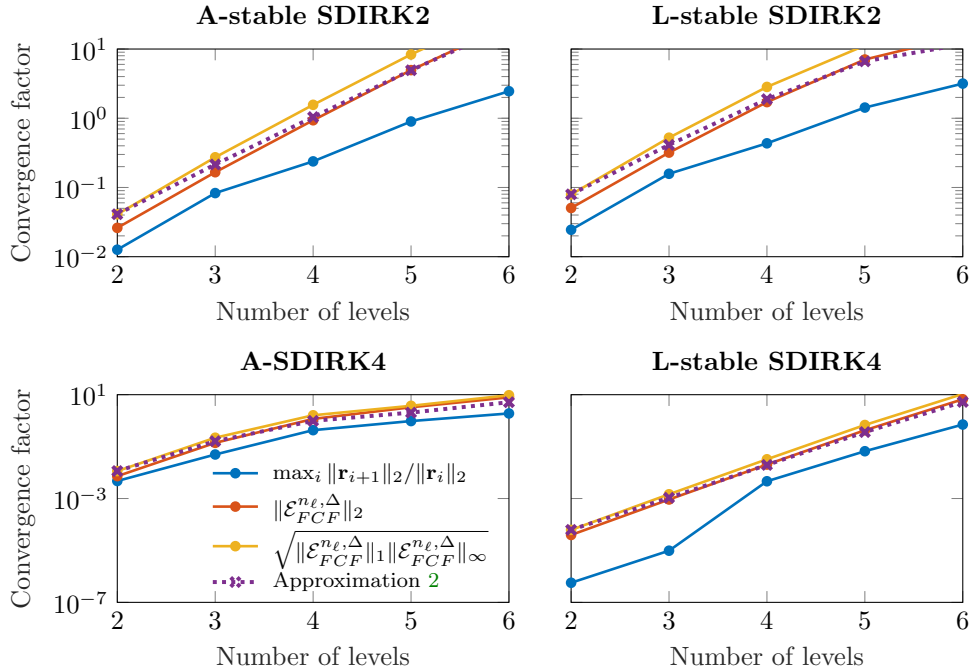
\FloatBarrier

In general, the upper bounds on convergence applied to the wave equation are significantly less sharp
compared to the diffusion equation (see Section \ref{diffusion-equation-sec}), but they are still able to
accurately represent trends.
For example, convergence factors are initially constant in most cases, then increase almost linearly
with the number of levels, such as the case of L-stable SDIRK3 in Supplementary Figure \ref{wave-F-cycle-nt1024-FCF-relaxation-nn11x11-suppfig}.
This highlights the fact that designing robust and convergent parallel-in-time algorithms for hyperbolic problems
is generally perceived as difficult, and emphasizes the benefit of the presented upper bounds for F-cycle algorithms.
For example, the convergence factor can be estimated a priori to select a time grid hierarchy
that is likely to yield a significant speedup.
In combination with performance modeling \cite{GahvariDobrevFalgoutKolevSchroderSchulzYang2016},
such a priori estimates can provide valuable guidance.

Note that for the problem considered here, the benefit of FCF-relaxation over F-relaxation
observed for the diffusion equation does not seem to apply for the wave equation (compare
Supplementary Figures \ref{wave-V-cycle-nt1024-F-relaxation-nn11x11-suppfig}
and \ref{wave-V-cycle-nt1024-FCF-relaxation-nn11x11-suppfig}, or Supplementary
Figures \ref{wave-F-cycle-nt1024-F-relaxation-nn11x11-suppfig}
and \ref{wave-F-cycle-nt1024-FCF-relaxation-nn11x11-suppfig}).
However, in some cases FCF-relaxation increases the maximum number of time grid levels
for which convergence can be achieved.
Thus, in practice one would prefer F-relaxation over FCF-relaxation to reduce the computational cost of a given algorithm.
The fact that FCF-relaxation is not sufficient to design a scalable multilevel solver for the wave equation
is a major difference to observations for the diffusion equation.

We further note, that the observed convergence factors and upper bound values are smaller with higher time integration
order, especially when L-stable SDIRK schemes are employed.
For example, the theory suggests to use five-level MGRIT with F-cycles and L-stable SDIRK4 with an estimated upper bound
on the convergence factor of $O (10^{-3})$, which is a very fast algorithm for hyperbolic PDEs.
\FloatBarrier
\begin{figure}[h!]
    \centering
    \setlength{\figurewidth}{0.4\linewidth}
    \setlength{\figureheight}{0.25\linewidth}
    \begin{subfigure}[b]{0.4\textwidth}
        \definecolor{mycolor1}{rgb}{0.00000,0.44700,0.74100}\definecolor{mycolor2}{rgb}{0.85000,0.32500,0.09800}\definecolor{mycolor3}{rgb}{0.92900,0.69400,0.12500}\definecolor{mycolor4}{rgb}{0.49400,0.18400,0.55600}\begin{tikzpicture}

\begin{axis}[width=\figurewidth,
height=0.849\figureheight,
at={(0\figurewidth,0\figureheight)},
scale only axis,
unbounded coords=jump,
xmin=2,
xmax=6,
xtick={2, 3, 4, 5, 6},
xlabel style={font=\color{white!15!black}},
xlabel={Number of levels},
ymode=log,
ymin=0.01,
ymax=10,
yminorticks=true,
ylabel style={font=\color{white!15!black}},
ylabel={Convergence factor},
axis background/.style={fill=white},
title style={font=\bfseries},
title={A-stable SDIRK2}
]
\addplot [color=mycolor1, line width=1pt, mark size=1.5pt, mark=*, mark options={solid, mycolor1}, forget plot]
  table[row sep=crcr]{2	0.0126237948958803\\
3	0.0125223995273946\\
4	0.0138907722246968\\
5	0.0173169651195102\\
6	0.261363667588777\\
};
\addplot [color=mycolor2, line width=1pt, mark size=1.5pt, mark=*, mark options={solid, mycolor2}, forget plot]
  table[row sep=crcr]{2	0.0260990800226155\\
3	0.0287320602092014\\
4	0.0795579489943677\\
5	2.52506377273171\\
6	213.94573099681\\
};
\addplot [color=mycolor3, line width=1pt, mark size=1.5pt, mark=*, mark options={solid, mycolor3}, forget plot]
  table[row sep=crcr]{2	0.0409562483157829\\
3	0.0494020068666622\\
4	0.172720594494204\\
5	5.34199571628685\\
6	443.1092172607\\
};
\addplot [color=mycolor4, line width=1pt, mark size=1.5pt, mark=*, mark options={solid, mycolor4}, forget plot]
  table[row sep=crcr]{2	nan\\
3	nan\\
4	nan\\
5	nan\\
6	nan\\
};
\end{axis}
\end{tikzpicture}    \end{subfigure}\qquad\qquad\qquad
    \begin{subfigure}[b]{0.4\textwidth}
        \definecolor{mycolor1}{rgb}{0.00000,0.44700,0.74100}\definecolor{mycolor2}{rgb}{0.85000,0.32500,0.09800}\definecolor{mycolor3}{rgb}{0.92900,0.69400,0.12500}\definecolor{mycolor4}{rgb}{0.49400,0.18400,0.55600}\begin{tikzpicture}

\begin{axis}[width=\figurewidth,
height=0.851\figureheight,
at={(0\figurewidth,0\figureheight)},
scale only axis,
unbounded coords=jump,
xmin=2,
xmax=6,
xtick={2, 3, 4, 5, 6},
xlabel style={font=\color{white!15!black}},
xlabel={Number of levels},
ymode=log,
ymin=0.01,
ymax=10,
yminorticks=true,
yticklabels={,,},
axis background/.style={fill=white},
title style={font=\bfseries},
title={L-stable SDIRK2}
]
\addplot [color=mycolor1, line width=1pt, mark size=1.5pt, mark=*, mark options={solid, mycolor1}, forget plot]
  table[row sep=crcr]{2	0.0244843517742587\\
3	0.0269869081251273\\
4	0.0301905224229283\\
5	0.133896959791831\\
6	1.02280578017734\\
};
\addplot [color=mycolor2, line width=1pt, mark size=1.5pt, mark=*, mark options={solid, mycolor2}, forget plot]
  table[row sep=crcr]{2	0.0505531000941094\\
3	0.0681299045743775\\
4	0.436411623452794\\
5	17.5611723921844\\
6	479.767273137187\\
};
\addplot [color=mycolor3, line width=1pt, mark size=1.5pt, mark=*, mark options={solid, mycolor3}, forget plot]
  table[row sep=crcr]{2	0.0793185638122101\\
3	0.126352243452018\\
4	0.914937020783199\\
5	35.7980221535127\\
6	971.134880663888\\
};
\addplot [color=mycolor4, line width=1pt, mark size=1.5pt, mark=*, mark options={solid, mycolor4}, forget plot]
  table[row sep=crcr]{2	nan\\
3	nan\\
4	nan\\
5	nan\\
6	nan\\
};
\end{axis}
\end{tikzpicture}    \end{subfigure}\\[1ex]
    \setlength{\figurewidth}{0.4\linewidth}
    \setlength{\figureheight}{0.25\linewidth}
    \begin{subfigure}[b]{0.4\textwidth}
        \definecolor{mycolor1}{rgb}{0.00000,0.44700,0.74100}\definecolor{mycolor2}{rgb}{0.85000,0.32500,0.09800}\definecolor{mycolor3}{rgb}{0.92900,0.69400,0.12500}\definecolor{mycolor4}{rgb}{0.49400,0.18400,0.55600}\begin{tikzpicture}

\begin{axis}[width=\figurewidth,
height=0.849\figureheight,
at={(0\figurewidth,0\figureheight)},
scale only axis,
unbounded coords=jump,
xmin=2,
xmax=6,
xtick={2, 3, 4, 5, 6},
xlabel style={font=\color{white!15!black}},
xlabel={Number of levels},
ymode=log,
ymin=1e-07,
ymax=10,
yminorticks=true,
ylabel style={font=\color{white!15!black}},
ylabel={Convergence factor},
axis background/.style={fill=white},
title style={font=\bfseries},
title={A-SDIRK4},
legend style={at={(0.6,0)}, anchor=south, draw=none, fill=none, legend cell align=left}
]
\addplot [color=mycolor1, line width=1pt, mark size=1.5pt, mark=*, mark options={solid, mycolor1}, forget plot]
  table[row sep=crcr]{2	0.00473339681471841\\
3	0.00570012359076608\\
4	0.042275299622819\\
5	0.393489733360033\\
6	0.24164775242275\\
};
\addplot [color=mycolor2, line width=1pt, mark size=1.5pt, mark=*, mark options={solid, mycolor2}, forget plot]
  table[row sep=crcr]{2	0.00727312389153332\\
3	0.0147625314609372\\
4	0.256344803985121\\
5	4.07424782102752\\
6	51.3717336100215\\
};
\addplot [color=mycolor3, line width=1pt, mark size=1.5pt, mark=*, mark options={solid, mycolor3}, forget plot]
  table[row sep=crcr]{2	0.0114075353317239\\
3	0.0275644643958277\\
4	0.47238551855669\\
5	6.49731471331314\\
6	96.8873295153152\\
};
\addplot [color=mycolor4, line width=1pt, mark size=1.5pt, mark=*, mark options={solid, mycolor4}, forget plot]
  table[row sep=crcr]{2	nan\\
3	nan\\
4	nan\\
5	nan\\
6	nan\\
};

\addlegendimage{color=mycolor1, line width=1pt, mark size=1.5pt, mark=*}
\addlegendimage{color=mycolor2, line width=1pt, mark size=1.5pt, mark=*}
\addlegendimage{color=mycolor3, line width=1pt, mark size=1.5pt, mark=*}

\addlegendentry{\footnotesize~$\max_i \|\mathbf{r}_{i+1}\|_2 / \|\mathbf{r}_i\|_2$}
\addlegendentry{\footnotesize~$\| \mathcal{F}_{FCF}^{n_\ell,\Delta} \|_2$}
\addlegendentry{\footnotesize~$\sqrt{ \| \mathcal{F}_{FCF}^{n_\ell,\Delta} \|_1 \| \mathcal{F}_{FCF}^{n_\ell,\Delta} \|_\infty }$}

\end{axis}
\end{tikzpicture}
    \end{subfigure}\qquad\qquad\qquad
    \begin{subfigure}[b]{0.4\textwidth}
        \definecolor{mycolor1}{rgb}{0.00000,0.44700,0.74100}\definecolor{mycolor2}{rgb}{0.85000,0.32500,0.09800}\definecolor{mycolor3}{rgb}{0.92900,0.69400,0.12500}\definecolor{mycolor4}{rgb}{0.49400,0.18400,0.55600}\begin{tikzpicture}

\begin{axis}[width=\figurewidth,
height=0.849\figureheight,
at={(0\figurewidth,0\figureheight)},
scale only axis,
unbounded coords=jump,
xmin=2,
xmax=6,
xtick={2, 3, 4, 5, 6},
xlabel style={font=\color{white!15!black}},
xlabel={Number of levels},
ymode=log,
ymin=1e-07,
ymax=10,
yminorticks=true,
yticklabels={,,},
axis background/.style={fill=white},
title style={font=\bfseries},
title={L-stable SDIRK4}
]
\addplot [color=mycolor1, line width=1pt, mark size=1.5pt, mark=*, mark options={solid, mycolor1}, forget plot]
  table[row sep=crcr]{2	5.68527005677997e-07\\
3	5.6842020787866e-07\\
4	5.73358923892536e-07\\
5	2.58410883628764e-06\\
6	0.000243520560478492\\
};
\addplot [color=mycolor2, line width=1pt, mark size=1.5pt, mark=*, mark options={solid, mycolor2}, forget plot]
  table[row sep=crcr]{2	3.94737354378682e-05\\
3	3.94796994891262e-05\\
4	3.9862557389292e-05\\
5	0.000176618155009953\\
6	0.378508156908287\\
};
\addplot [color=mycolor3, line width=1pt, mark size=1.5pt, mark=*, mark options={solid, mycolor3}, forget plot]
  table[row sep=crcr]{2	6.19445610379228e-05\\
3	6.19510424022491e-05\\
4	6.37369536838193e-05\\
5	0.000419439521747319\\
6	0.777120455015235\\
};
\addplot [color=mycolor4, line width=1pt, mark size=1.5pt, mark=*, mark options={solid, mycolor4}, forget plot]
  table[row sep=crcr]{2	nan\\
3	nan\\
4	nan\\
5	nan\\
6	nan\\
};
\end{axis}
\end{tikzpicture}    \end{subfigure}
    \caption{Wave equation: The convergence of MGRIT with F-cycles and FCF-relaxation
        deteriorates with a larger number of time grid levels compared to MGRIT with V-cycles,
        see Figure \ref{wave-V-cycle-nt1024-FCF-relaxation-nn11x11-fig}.
        Generally, convergent algorithms are given for a larger range of time grid levels
        and observed convergence is better than the predictions from the upper bounds.
        This shows that the choice of F-cycles over V-cycles is one likely ingredient
        for future improvements of MGRIT for hyperbolic-type PDEs.}
    \label{wave-F-cycle-nt1024-FCF-relaxation-nn11x11-fig}
\end{figure}
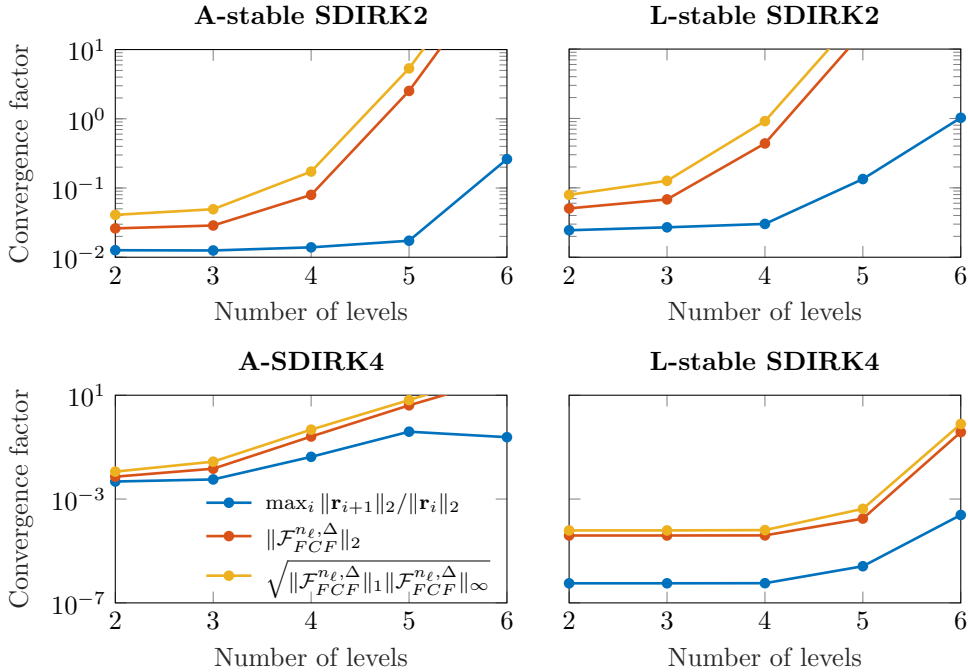
\FloatBarrier
\section{Discussion}\label{discussion-sec}
In Section \ref{numerical-results-sec},
we have compared the \emph{a priori} estimate for various developed bounds
and approximate convergence factors with observed \emph{worst-case} convergence in numerical experiments.
These investigations were performed for a fixed temporal domain (in particular, a fixed fine grid size)
and a fixed choice of the spatial discretization scheme (e.g., the large spatial step size
underresolves the discontinuous initial condition in Section \ref{diffusion-equation-sec}).
An interesting future research question is how the convergence framework can help guide the selection
of an optimal space-time discretization for the parallel-in-time integration with MGRIT
with regards to best convergence and best speedup (e.g., by combining the convergence framework
with performance models \cite{GahvariDobrevFalgoutKolevSchroderSchulzYang2016}).
Here, important aspects are the dependency of the sharpness of the bounds on material parameters,
spatial discretization and resolution, fine grid size of the temporal domain, and others.

The presented upper bounds vary in their respective time complexity and it is demonstrated that, e.g.,
using the inequality \eqref{inequality-eq} reduces the time complexity but still gives good and reasonably
sharp upper bounds.
Further, the proposed approximate convergence factors for V-cycle algorithms with F- and FCF-relaxation
(Approximation \ref{cf-nl-F-app} and \ref{cf-nl-FCF-app})
provide analytic formulae to estimate observed convergence \emph{a priori} with effectively constant time complexity
(if implemented in parallel). In the investigated cases, the approximate convergence factors yield a priori estimates
that are at least as good as more expensive bounds.
All these observations, however, rest on the assumption that the eigenvalues of the time-stepping operator
$\Phi_\ell$ can be computed, which might result in prohibitive computational cost for large-scale problems.
For Runge-Kutta schemes, the number of such operations can be reduced by only computing the eigenvalues
of the spatial operator $\mathcal{L}$ and evaluating the stability function,
as opposed to computing the eigenvalues for the family of $\Phi_\ell$ for all considered Runge-Kutta schemes.
While this might not be possible in general, the derivation
of Fourier symbols \cite{DesterckFriedhoffHowseMaclachlan2019_preprint} can provide another viable path
to reduce the time complexity of computing such bounds.
Furthermore, with prior knowledge of $\mathcal{L}$ (e.g., $\mathcal{L}$ is symmetric positive definite
or skew symmetric), the need for solving an eigenvalue problem can be avoided.
Despite the difficulty of performing multilevel convergence analysis for large-scale problems,
there is yet a lot that can be learned from investigating smaller-scale problems.

MGRIT natively supports the parallel-in-time integration of nonlinear problems
by using FAS multigrid (similar to other methods, e.g., \cite{HamonSchreiberMinion2019}).
The multilevel convergence framework developed in this work is limited to the linear case;
however, it is able to illustrate the strengths and weaknesses of the algorithm in this setting
(similar to the two-level theory developed in \cite{DobrevKolevPeterssonSchroder2017}).
Future work will investigate how theoretical results
for linear (or linearized) problems can guide the application of MGRIT for nonlinear problems.
On the other hand, (two-level) convergence theory
for Parareal \cite{GanderVandewalle2007}
was extended to the nonlinear case \cite{GanderHairer2008} and, naturally,
a similar extension of MGRIT convergence theory would be desirable.

Numerical studies demonstrate the benefit of using FCF-relaxation for parabolic model problems
(see Figure \ref{diffusion-isotropic-V-F-cycle-nt1024-FCF-relaxation-nn11x11-fig}), but
FCF-relaxation does not significantly affect convergence of MGRIT for the hyperbolic model problem.
However, theoretical results do confirm the advantage of using F-cycles over V-cycles for the hyperbolic model problem,
a result that also applies to the isotropic diffusion equation, if integrated by L-stable Runge-Kutta schemes.
Results here confirm the observation that A-stable schemes are generally less suited for parallel-in-time
integration than L-stable schemes \cite{DobrevKolevPeterssonSchroder2017}.
For the diffusion model problem, theory implies that naive multilevel Parareal (i.e. MGRIT V-cycles with
F-relaxation) does not yield a scalable algorithm and has increasing iteration counts with an increasing
number of levels for A-stable schemes. However, results here indicate that stronger cycling and relaxation,
such as F-cycles and FCF-relaxation, can alleviate this weakness.
\section{Conclusion}\label{conclusion-sec}

In this work, we develop a framework for multilevel convergence of MGRIT for linear PDEs.
This framework provides \emph{a priori} bounds and approximations for the convergence factor
of various types of MGRIT configurations, including different cycling strategies (V- and F-cycles)
and relaxation schemes ($r$FCF-relaxation).
The new theoretical results are a generalization of the two-grid theory derived
in \cite{DobrevKolevPeterssonSchroder2017} and based on similar assumptions
(for example, simultaneously diagonalizable and stable time-stepping
operators).
This work also presents a generalization of the two-level bounds
derived in \cite{DobrevKolevPeterssonSchroder2017}
to the case of arbitrary numbers of relaxation steps.

In complementary numerical studies, the theoretical results
are assessed for two different model problems, the anisotropic diffusion equation
and the second-order wave equation.
It is found that the \emph{a priori} upper bounds are relatively sharp upper bounds on observed convergence for
the diffusion equation, and accurately describe qualitative behavior for the wave equation.
Generally, these bounds are sharper for smaller numbers of time grids.

Overall, the theoretical convergence results lay the groundwork for future in-depth examination
and understanding of MGRIT.
This is especially true for the solution of hyperbolic PDEs; an application area where the
design of robust and efficient parallel-in-time algorithms has proven challenging and where
\emph{a priori} bounds can avoid (in parts) extensive numerical testing.
Thus, future development and improvement of MGRIT (different coarse-grid
operators \cite{KrzysikDesterkMaclachlanFriedhoff2019},
investigation of relation between errors and phase-shifts \cite{Ruprecht2018},
and others)
can be guided by the proposed multilevel convergence framework,
where important recommendations can be made, such as the use of higher-order L-stable
SDIRK methods and F-cycles with F-relaxation for hyperbolic problems.
\FloatBarrier
\section*{Acknowledgements}

AH would like to acknowledge the Lawrence Livermore National Laboratory, USA, and RDF/JBS
for funding and hosting a summer internship during which most of the theoretical results of this manuscript
were derived.

\bibliographystyle{siam}
\bibliography{main_strip}

\begin{thebibliography}{10}

\bibitem{BonaventuraDellarocca2015}
{\sc L.~Bonaventura and A.~Della~Rocca}, {\em Monotonicity, positivity and
  strong stability of the {TR-BDF2} method and of its {SSP} extensions}, arXiv
  preprint arXiv:1510.04303,  (2015).

\bibitem{ChristliebMacdonaldOng2010}
{\sc A.~J. Christlieb, C.~B. Macdonald, and B.~W. Ong}, {\em {Parallel
  high-order integrators}}, SIAM Journal on Scientific Computing, 32 (2010),
  pp.~818--835.

\bibitem{DesterckFriedhoffHowseMaclachlan2019_preprint}
{\sc H.~De~Sterck, S.~Friedhoff, A.~J.~M. Howse, and S.~P. MacLachlan}, {\em
  Convergence analysis for parallel-in-time solution of hyperbolic systems},
  arXiv preprint arXiv:1903.08928,  (2019).

\bibitem{DobrevKolevPeterssonSchroder2017}
{\sc V.~A. Dobrev, T.~V. Kolev, N.~A. Petersson, and J.~B. Schroder}, {\em
  Two-level convergence theory for multigrid reduction in time ({MGRIT})}, SIAM
  Journal on Scientific Computing, 39 (2017), pp.~S501--S527.

\bibitem{DuarteDobbinsSmooke2016}
{\sc M.~Duarte, R.~Dobbins, and M.~Smooke}, {\em High order implicit time
  integration schemes on multiresolution adaptive grids for stiff {PDEs}},
  arXiv preprint arXiv:1604.00355,  (2016).

\bibitem{EmmettMinion2012}
{\sc M.~Emmett and M.~L. Minion}, {\em {T}oward an {E}fficient {P}arallel in
  {T}ime {M}ethod for {P}artial {D}ifferential {E}quations}, Communications in
  Applied Mathematics and Computational Science, 7 (2012), pp.~105--132.

\bibitem{FalgoutFriedhoffKolevMaclachlanSchroder2014}
{\sc R.~D. Falgout, S.~Friedhoff, T.~V. Kolev, S.~P. MacLachlan, and J.~B.
  Schroder}, {\em Parallel time integration with multigrid}, SIAM Journal on
  Scientific Computing, 36 (2014), pp.~C635--C661.

\bibitem{FalgoutLecouvezWoodward2017}
{\sc R.~D. Falgout, M.~Lecouvez, and C.~S. Woodward}, {\em A parallel-in-time
  algorithm for variable step multistep methods}, preprint.

\bibitem{FalgoutManteuffelOneillSchroder2017}
{\sc R.~D. Falgout, T.~A. Manteuffel, B.~O'Neill, and J.~B. Schroder}, {\em
  Multigrid reduction in time for nonlinear parabolic problems: A case study},
  SIAM Journal on Scientific Computing, 39 (2017), pp.~S298--S322.

\bibitem{FarhatChandesris2003}
{\sc C.~Farhat and M.~Chandesris}, {\em Time-decomposed parallel
  time-integrators: theory and feasibility studies for fluid, structure, and
  fluid-structure applications}, International Journal for Numerical Methods in
  Engineering, 58 (2003), pp.~1397--1434.

\bibitem{FarhatCortial2006}
{\sc C.~Farhat, J.~Cortial, C.~Dastillung, and H.~Bavestrello}, {\em
  Time-parallel implicit integrators for the near-real-time prediction of
  linear structural dynamic responses}, International Journal for Numerical
  Methods in Engineering, 67 (2006), pp.~697--724.

\bibitem{FriedhoffFalgoutKolevMaclachlanSchroder2012}
{\sc S.~Friedhoff, R.~D. Falgout, T.~V. Kolev, S.~MacLachlan, and J.~B.
  Schroder}, {\em A multigrid-in-time algorithm for solving evolution equations
  in parallel}, tech. rep., Lawrence Livermore National Lab.(LLNL), Livermore,
  CA (United States), 2012.

\bibitem{FriedhoffMaclachlan2015}
{\sc S.~Friedhoff and S.~MacLachlan}, {\em A generalized predictive analysis
  tool for multigrid methods}, Numerical Linear Algebra with Applications, 22
  (2015), pp.~618--647.

\bibitem{GahvariDobrevFalgoutKolevSchroderSchulzYang2016}
{\sc H.~Gahvari, V.~A. Dobrev, R.~D. Falgout, T.~V. Kolev, J.~B. Schroder,
  M.~Schulz, and U.~M. Yang}, {\em A performance model for allocating the
  parallelism in a multigrid-in-time solver}, in Performance Modeling,
  Benchmarking and Simulation of High Performance Computer Systems (PMBS),
  International Workshop on, IEEE, 2016, pp.~22--31.

\bibitem{Gander2015}
{\sc M.~J. Gander}, {\em 50 {Y}ears of {T}ime {P}arallel {T}ime {I}ntegration},
  in Multiple Shooting and Time Domain Decomposition, Springer, 2015.

\bibitem{GanderHairer2008}
{\sc M.~J. Gander and E.~Hairer}, {\em Nonlinear convergence analysis for the
  parareal algorithm}, in {D}omain {D}ecomposition {M}ethods in {S}cience and
  {E}ngineering {XVII}, Springer, 2008, pp.~45--56.

\bibitem{GanderKwokZhang2018}
{\sc M.~J. Gander, F.~Kwok, and H.~Zhang}, {\em Multigrid interpretations of
  the parareal algorithm leading to an overlapping variant and {MGRIT}},
  Computing and Visualization in Science, 19 (2018), pp.~59--74.

\bibitem{GanderVandewalle2007}
{\sc M.~J. Gander and S.~Vandewalle}, {\em On the superlinear and linear
  convergence of the parareal algorithm}, in {D}omain {D}ecomposition {M}ethods
  in {S}cience and {E}ngineering {XVI}, Springer, 2007, pp.~291--298.

\bibitem{GuentherGaugerSchroder2018a}
{\sc S.~G{\"u}nther, N.~R. Gauger, and J.~B. Schroder}, {\em A non-intrusive
  parallel-in-time adjoint solver with the {XBraid} library}, Computing and
  Visualization in Science, 19 (2018), pp.~85--95.

\bibitem{GuentherGaugerSchroder2018b}
\leavevmode\vrule height 2pt depth -1.6pt width 23pt, {\em A non-intrusive
  parallel-in-time approach for simultaneous optimization with unsteady
  {PDEs}}, Optimization Methods and Software,  (2018), pp.~1--16.

\bibitem{HackbuschTrottenberg2006}
{\sc W.~Hackbusch and U.~Trottenberg}, {\em {M}ultigrid {M}ethods:
  {P}roceedings of the {C}onference {H}eld at {K}{\"o}ln-{P}orz, {N}ovember
  23-27, 1981}, vol.~960, Springer, 1981.

\bibitem{HairerNorsettWanner1993}
{\sc E.~Hairer, S.~P. N{\o}rsett, and G.~Wanner}, {\em {S}olving {O}rdinary
  {D}ifferential {E}quations {I}: {N}onstiff {P}roblems}, 1993.

\bibitem{HairerWanner1996}
{\sc E.~Hairer and G.~Wanner}, {\em {S}olving {O}rdinary {D}ifferential
  {E}quations {II}: {S}tiff and {D}ifferential-{A}lgebraic {P}roblems},
  Springer Series in Computational Mathematics, 14 (1996).

\bibitem{HamonSchreiberMinion2019}
{\sc F.~P. Hamon, M.~Schreiber, and M.~L. Minion}, {\em Multi-level spectral
  deferred corrections scheme for the shallow water equations on the rotating
  sphere}, Journal of Computational Physics, 376 (2019), pp.~435--454.

\bibitem{HessenthalerNordslettenRoehrleSchroderFalgout2018}
{\sc A.~Hessenthaler, D.~Nordsletten, O.~R\"ohrle, J.~B. Schroder, and R.~D.
  Falgout}, {\em Convergence of the multigrid reduction in time algorithm for
  the linear elasticity equations}, Numerical Linear Algebra with Applications,
  25 (2018), p.~e2155.
\newblock e2155 nla.2155.

\bibitem{Higham1992}
{\sc N.~J. Higham}, {\em Estimating the matrix p-norm}, Numerische Mathematik,
  62 (1992), pp.~539--555.

\bibitem{HortonVandewalle1995}
{\sc G.~Horton and S.~Vandewalle}, {\em {A} {S}pace-{T}ime {M}ultigrid {M}ethod
  for {P}arabolic {P}artial {D}ifferential {E}quations}, SIAM Journal on
  Scientific Computing, 16 (1995), pp.~848--864.

\bibitem{HowseDesterckFalgoutMaclachlanSchroder2019}
{\sc A.~J. Howse, H.~De~Sterck, R.~D. Falgout, S.~MacLachlan, and J.~Schroder},
  {\em Parallel-in-time multigrid with adaptive spatial coarsening for the
  linear advection and inviscid {B}urgers equations}, SIAM Journal on
  Scientific Computing, 41 (2019), pp.~A538--A565.

\bibitem{Kraaijevanger1991}
{\sc J.~F. B.~M. Kraaijevanger}, {\em {C}ontractivity of {R}unge-{K}utta
  methods}, BIT Numerical Mathematics, 31 (1991), pp.~482--528.

\bibitem{KrzysikDesterkMaclachlanFriedhoff2019}
{\sc O.~A. Krzysik, H.~De~Sterck, S.~P. MacLachlan, and S.~Friedhoff}, {\em
  {O}n selecting coarse-grid operators for {P}arareal and {MGRIT} applied to
  linear advection}, arXiv preprint arXiv:1902.07757,  (2019).

\bibitem{LaxRichtmyer1956}
{\sc P.~D. Lax and R.~D. Richtmyer}, {\em {S}urvey of the {S}tability of
  {L}inear {F}inite {D}ifference {E}quations}, Communications on Pure and
  Applied Mathematics, 9 (1956), pp.~267--293.

\bibitem{LecouvezFalgoutWoodwardTop2016}
{\sc M.~Lecouvez, R.~D. Falgout, C.~S. Woodward, and P.~Top}, {\em {A}
  {P}arallel {M}ultigrid {R}eduction in {T}ime {M}ethod for {P}ower {S}ystems},
  in Power and Energy Society General Meeting (PESGM), 2016, IEEE, 2016,
  pp.~1--5.

\bibitem{Leveque2007}
{\sc R.~J. LeVeque}, {\em {F}inite {D}ifference {M}ethods for {O}rdinary and
  {P}artial {D}ifferential {E}quations: {S}teady-{S}tate and {T}ime-{D}ependent
  {P}roblems}, vol.~98, {SIAM}, 2007.

\bibitem{LionsMadayTurinici2001}
{\sc J.-L. Lions, Y.~Maday, and G.~Turinici}, {\em {A "parareal" in time
  discretization of {PDE}'s}}, Comptes Rendus de l'Acad{\'e}mie des Sciences -
  Series I - Mathematics, 332 (2001), pp.~661--668.

\bibitem{LubichOstermann1987}
{\sc C.~Lubich and A.~Ostermann}, {\em Multi-grid dynamic iteration for
  parabolic equations}, BIT Numerical Mathematics, 27 (1987), pp.~216--234.

\bibitem{ManteuffelOlsonSchroderSouthworth2017}
{\sc T.~A. Manteuffel, L.~N. Olson, J.~B. Schroder, and B.~S. Southworth}, {\em
  {A} {R}oot-node {B}ased {A}lgebraic {M}ultigrid {M}ethod}, SIAM Journal on
  Scientific Computing, (accepted) (2017).

\bibitem{Nievergelt1964}
{\sc J.~Nievergelt}, {\em {P}arallel {M}ethods for {I}ntegrating {O}rdinary
  {D}ifferential {E}quations}, Commun. ACM, 7 (1964), pp.~731--733.

\bibitem{Ruprecht2018}
{\sc D.~Ruprecht}, {\em Wave propagation characteristics of parareal},
  Computing and Visualization in Science, 19 (2018), pp.~1--17.

\bibitem{SandersonCurtin2016}
{\sc C.~Sanderson and R.~Curtin}, {\em {A}rmadillo: a template-based {C}++
  library for linear algebra}, Journal of Open Source Software,  (2016).

\bibitem{SandersonCurtin2018}
\leavevmode\vrule height 2pt depth -1.6pt width 23pt, {\em {A}
  {U}ser-{F}riendly {H}ybrid {S}parse {M}atrix {C}lass in {C}++}, arXiv
  preprint arXiv:1805.03380,  (2018).

\bibitem{Schroder2017}
{\sc J.~B. Schroder}, {\em {P}arallelizing {O}ver {A}rtificial {N}eural
  {N}etwork {T}raining {R}uns with {M}ultigrid}, arXiv preprint
  arXiv:1708.02276,  (2017).

\bibitem{SchroderFalgoutWoodwardTopLecouvez2018}
{\sc J.~B. Schroder, R.~D. Falgout, C.~S. Woodward, P.~Top, and M.~Lecouvez},
  {\em {P}arallel-in-{T}ime {S}olution of {P}ower {S}ystems with {S}cheduled
  {E}vents}, 2018 Power and Energy Society General Meeting (PESGM), IEEE,
  (2018).

\bibitem{Southworth2018}
{\sc B.~S. Southworth}, {\em Necessary conditions and tight two-level
  convergence bounds for parareal and multigrid reduction in time}, SIAM
  Journal on Matrix Analysis and Applications, 40 (2019), pp.~564--608.

\bibitem{tight2}
{\sc B.~S. Southworth, W.~Mitchell, and A.~Hessenthaler}, {\em {T}ight
  {T}wo-level {C}onvergence of {L}inear {P}arareal and {MGRIT}: {E}xtensions
  and {I}mplications in {P}ractice}, (in preparation),  (2019).

\bibitem{SpeckRuprechtEmmettMinionBoltenKrause2015}
{\sc R.~Speck, D.~Ruprecht, M.~Emmett, M.~Minion, M.~Bolten, and R.~Krause},
  {\em A multi-level spectral deferred correction method}, BIT Numerical
  Mathematics, 55 (2015), pp.~843--867.

\bibitem{SpeckRuprechtKrauseEmmettMinionWinkelGibbon2012}
{\sc R.~Speck, D.~Ruprecht, R.~Krause, M.~Emmett, M.~L. Minion, M.~Winkel, and
  P.~Gibbon}, {\em {A massively space-time parallel {N}-body solver}}, in
  Proceedings of the International Conference on High Performance Computing,
  Networking, Storage and Analysis, {SC '12}, Los Alamitos, CA, USA, 2012, IEEE
  Computer Society Press, pp.~92:1--92:11.

\bibitem{VandewalleVandevelde1994}
{\sc S.~Vandewalle and E.~{Van de Velde}}, {\em Space-time concurrent multigrid
  waveform relaxation}, Annals of Numerical Mathematics, 1 (1994),
  pp.~347--360.

\bibitem{XBraid}
{\em {XB}raid: Parallel multigrid in time}.
\newblock \url{http://llnl.gov/casc/xbraid}.

\bibitem{YueShuXuBuPan2018}
{\sc X.~Yue, S.~Shu, X.~Xu, W.~Bu, and K.~Pan}, {\em {P}arallel-in-{T}ime with
  {F}ully {F}inite {E}lement {M}ultigrid for 2-{D} {S}pace-fractional
  {D}iffusion {E}quations}, arXiv preprint arXiv:1805.06688,  (2018).

\end{thebibliography}

\FloatBarrier
\clearpage

\renewcommand{\thesection}{SM\arabic{section}}
\setcounter{section}{0}%
\setcounter{subsection}{0}%
\setcounter{equation}{0}
\setcounter{figure}{0}
\setcounter{table}{0}
\setcounter{page}{1}
\renewcommand{\theequation}{SM\arabic{equation}}
\renewcommand{\thefigure}{SM\arabic{figure}}

\begin{center}
    \textbf{\normalsize\MakeUppercase{
    Supplementary Materials:
    Multilevel convergence analysis of multigrid-reduction-in-time$^*$}}\\[2ex]
    \footnotesize{
    ANDREAS HESSENTHALER$^{\dagger,*}$,
    BEN S.\ SOUTHWORTH$^\ddagger$,
    DAVID NORDSLETTEN$^\S$,
    OLIVER R\"OHRLE$^\dagger$,
    ROBERT D.\ FALGOUT$^\P$,}
    \scriptsize{AND} \footnotesize{JACOB B.\ SCHRODER$^\parallel$}\\[4ex]
    {\itshape \footnotesize
    ${}^\dagger$Institute for Modelling and Simulation of Biomechanical Systems,
    University of Stuttgart,
    Pfaffenwaldring 5a, 70569 Stuttgart, Germany\\[1ex]
    ${}^\ddagger$Department of Applied Mathematics, University of Colorado at Boulder, CO, USA\\[1ex]
    ${}^\S$School of Biomedical Engineering and Imaging Sciences,
    King's College London, 4th FL Rayne Institute,
    St Thomas Hospital, London, SE1 7EH\\[1ex]
    ${}^\P$Center for Applied Scientific Computing, Lawrence Livermore National Laboratory,
    P.O.\ Box 808, L-561, Livermore, CA 94551, USA\\[1ex]
    ${}^\parallel$Department of Mathematics and Statistics, University of New Mexico,
    310 SMLC, Albuquerque, NM 87131, USA}~\\[3ex]
    \scriptsize{$^*$This work performed under the auspices of the U.S. Department of Energy by Lawrence Livermore National Laboratory under Contract
DE-AC52-07NA27344, LLNL-JRNL-763460.}
\end{center}
~\\[1ex]

\section{Implementation of analytic and numerical bounds}\label{code-suppsec}

We provide a parallel C++ implementation of all derived bounds and approximate convergence
factors as part of this manuscript and as open-source software.\footnote{Github repository:
\url{github.com/XBraid/XBraid-convergence-est}}
The code takes the (complex or real) eigenvalues of the family of $\Phi_\ell$ as input along with a definition
of the desired MGRIT algorithm (V- or F-cycles, relaxation scheme, number of levels, coarsening factors, etc.)
and computes the bound or approximate convergence factor values.

We further implemented functionality for the user to supply the eigenvalues of a spatial operator and to compute
the respective eigenvalues of $\Phi_\ell$ based on the stability function of a given Runge-Kutta time integration scheme
and its Butcher tableau.

\section{Derivations: Bounds for MGRIT residual and error propagation}\label{bounds-MGRIT-error-propagation-appsec}

Here, we present derivations and proofs that were omitted in Section \ref{bounds-MGRIT-error-propagation-sec}.

\subsection{Two-level MGRIT with \lowercase{$r$}FCF-relaxation}\label{two-level-MGRIT-with-rFCF-relaxation-appsec}

The coarse-grid error propagator follows from Equation \eqref{E-rFCF-nl2-eqn} for $n_\ell = 2$,
\begin{equation*}
    \mathcal{E}_{rFCF}^{n_\ell = 2, \Delta} =  R_{I_0} \mathcal{E}_{rFCF}^{n_\ell = 2} P_0
    = (I - A_1^{-1} R_0 A_0 P_0) (I - R_0 A_0 P_0)^r
\end{equation*}
with,
\begin{align}
    R_0 A_0 P_0 &= \begin{bmatrix}
        I \\
        - \Phi_0^{m_0} & I \\
        & - \Phi_0^{m_0} & I \\
        && \ddots & \ddots \\
    \end{bmatrix},
    \label{R0A0P0-eq} \\
    I - A_1^{-1} R_0 A_0 P_0
    &= - \begin{bmatrix}
        0    \\
        \Phi_1 - \Phi_0^{m_0} & 0    \\
        \Phi_1 \left( \Phi_1 - \Phi_0^{m_0} \right) & \Phi_1 - \Phi_0^{m_0} & 0    \\
        \vdots &&& \ddots    \\
        \Phi_1^{N_1 - 2} \left( \Phi_1 - \Phi_0^{m_0} \right) & \cdots &&& 0    \\
    \end{bmatrix},
\end{align}
and
\begin{align*}
    (I - R_0 A_0 P_0)^r
    &= \begin{bmatrix}
        0 \\
        \Phi_0^{m_0} & 0 \\
        & \Phi_0^{m_0} & 0 \\
        && \ddots & \ddots \\
    \end{bmatrix} (I - R_0 A_0 P_0)^{r - 1} \\
    &= \begin{bmatrix}
        0 \\
        0 \\
        \Phi_0^{2 m_0} & 0 \\
        & \Phi_0^{2 m_0} & 0 \\
        && \ddots & \ddots \\
    \end{bmatrix} (I - R_0 A_0 P_0)^{r - 2} \\
    &= \ldots
    = \begin{bmatrix}
        0 \\
        \vdots \\
        0 \\
        \Phi_0^{r m_0} & 0 \\
        & \Phi_0^{r m_0} & 0 \\
        && \ddots & \ddots \\
    \end{bmatrix}.
\end{align*}

\subsection{Three-level V-cycles with F-relaxation}
\label{three-level-V-cycles-with-F-relaxation-appsec}

Evaluating the error propagator in Equation \eqref{E-F-nl-eqn} for a three-level V-cycle with F-relaxation
on the coarse-grid yields,
\begin{alignat*}{4}
    \mathcal{E}_{F}^{n_\ell = 3,\Delta}
    =  R_{I_0} \mathcal{E}_F^{n_\ell = 3} P_0
    & = I
    - \left[ P_1 A_2^{-1} R_1
    + S_1 (S_1^T A_1 S_1)^{-1} S_1^T \right] R_0 A_0 P_0,
\end{alignat*}
where,
\begin{alignat*}{4}
\resizebox{0.9\textwidth}{!}{$
    S_1 (S_1^T A_1 S_1)^{-1} S_1^T = \begin{bmatrix}
        0    \\
        & I    \\
        & \Phi_1        & I    \\
        & \Phi_1^2    & \Phi_1    & I    \\
        & \vdots            &            &    & \ddots    \\
        & \Phi_1^{m_1 - 2}    & \Phi_1^{m_1 - 3}    & \cdots    & \Phi_1    & I    \\
        &&&&&& 0    \\
        &&&&&&& I    \\
        &&&&&&& \Phi_1        & I    \\
        &&&&&&& \Phi_1^2    & \Phi_1    & I    \\
        &&&&&&& \vdots            &            &    & \ddots    \\
        &&&&&&& \Phi_1^{m_1 - 2}    & \Phi_1^{m_1 - 3}    & \cdots    & \Phi_1    & I    \\
        &&&&&&&&&&&& 0    \\
        &&&&&&&&&&&&& \ddots    \\
        &&&&&&&&&&&&&& 0    \\
    \end{bmatrix}$}
    \numberthis
    \label{invAff-eq}
\end{alignat*}
and
\begin{equation}
\resizebox{0.9\textwidth}{!}{$
    P_1 A_2^{-1} R_1
    = \begin{bmatrix}
        I    \\
        \Phi_1    \\
        \Phi_1^2    \\
        \vdots    \\
        \Phi_1^{m_1 - 1}    \\
        \Phi_2    & \Phi_1^{m_1 - 1}    & \Phi_1^{m_1 - 2}    & \cdots    & I    \\
        \Phi_1 \Phi_2    & \Phi_1 \Phi_1^{m_1 - 1}    & \Phi_1 \Phi_1^{m_1 - 2}    & \cdots    & \Phi_1    \\
        \vdots    & \vdots    & \vdots    && \vdots    \\
        \Phi_1^{m_1 - 1} \Phi_2    & \Phi_1^{m_1 - 1} \Phi_1^{m_1 - 1}    & \Phi_1^{m_1 - 1} I \Phi_1^{m_1 - 2}    & \cdots    & \Phi_1^{m_1 - 1}    \\
        \vdots    & \vdots    & \vdots    && \vdots    & \ddots    \\
        \vdots    & \vdots    & \vdots    && \vdots    \\
        \Phi_2^{N_2-2}    & \Phi_2^{N_2 - 3} \Phi_1^{m_1 - 1}    & \Phi_2^{N_2 - 3} \Phi_1^{m_1 - 2}    & \cdots    & \Phi_2^{N_2 - 3}    \\
        \Phi_1 \Phi_2^{N_2-2}    & \Phi_1 \Phi_2^{N_2 - 3} \Phi_1^{m_1 - 1}    & \Phi_1 \Phi_2^{N_2 - 3} \Phi_1^{m_1 - 2}    & \cdots    & \Phi_1 \Phi_2^{N_2 - 3}    \\
        \vdots    & \vdots    & \vdots    && \vdots    \\
        \Phi_1^{m_1 - 1} \Phi_2^{N_2-2}    & \Phi_1^{m_1 - 1} \Phi_2^{N_2 - 3} \Phi_1^{m_1 - 1}    & \Phi_1^{m_1 - 1} \Phi_2^{N_2 - 3} \Phi_1^{m_1 - 2}    & \cdots    & \Phi_1^{m_1 - 1} \Phi_2^{N_2 - 3}    \\
        \Phi_2^{N_2 - 1}    & \Phi_2^{N_2 - 2} \Phi_1^{m_1 - 1}    & \Phi_2^{N_2 - 2} \Phi_1^{m_1 - 2}    & \cdots    & \Phi_2^{N_2 - 2}    \\
    \end{bmatrix}$}.
    \label{PinvAR-eq}
\end{equation}
With \eqref{R0A0P0-eq}, \eqref{invAff-eq} and \eqref{PinvAR-eq} follows the nilpotent operator,
\begin{equation}
\resizebox{0.9\textwidth}{!}{$
    \mathcal{E}_{F}^{n_\ell = 3,\Delta}
    = \begin{bmatrix}
        0    \\
        \Phi_0^{m_0} - \Phi_1    & 0    \\
        \Phi_1 ( \Phi_0^{m_0} - \Phi_1 )    & \Phi_0^{m_0} - \Phi_1 & 0 \\
        \Phi_1^2 ( \Phi_0^{m_0} - \Phi_1 )    & \Phi_1 ( \Phi_0^{m_0} - \Phi_1 ) & \Phi_0^{m_0} - \Phi_1    \\
        \vdots    & \vdots    & \vdots    & \ddots    \\
        \Phi_1^{m_1 - 2} ( \Phi_0^{m_0} - \Phi_1 )    & \Phi_1^{m_1 - 3} ( \Phi_0^{m_0} - \Phi_1 )    & \Phi_1^{m_1 - 4} ( \Phi_0^{m_0} - \Phi_1 )    \\
        \Phi_1^{m_1 - 1} \Phi_0^{m_0} - \Phi_2    & \Phi_1^{m_1 - 2} ( \Phi_0^{m_0} - \Phi_1 )    & \Phi_1^{m_1 - 3} ( \Phi_0^{m_0} - \Phi_1 )    \\
        \Phi_1 ( \Phi_1^{m_1 - 1} \Phi_0^{m_0} - \Phi_2 )    & \Phi_1 \Phi_1^{m_1 - 2} ( \Phi_0^{m_0} - \Phi_1 )    & \Phi_1 \Phi_1^{m_1 - 3} ( \Phi_0^{m_0} - \Phi_1 )    \\
        \vdots    & \vdots    & \vdots    \\
        \Phi_1^{m_1 - 1} ( \Phi_1^{m_1 - 1} \Phi_0^{m_0} - \Phi_2 )    & \Phi_1^{m_1 - 1} \Phi_1^{m_1 - 2} ( \Phi_0^{m_0} - \Phi_1 )    & \Phi_1^{m_1 - 1} \Phi_1^{m_1 - 3} ( \Phi_0^{m_0} - \Phi_1 )    \\
        \vdots    & \vdots    & \vdots    \\
        \vdots    & \vdots    & \vdots    \\
        \Phi_2^{N_2 - 3} ( \Phi_1^{m_1 - 1} \Phi_0^{m_0} - \Phi_2 )    & \Phi_2^{N_2 - 3} \Phi_1^{m_1 - 2} ( \Phi_0^{m_0} - \Phi_1 )    & \Phi_2^{N_2 - 3} \Phi_1^{m_1 - 3} ( \Phi_0^{m_0} - \Phi_1 )    \\
        \Phi_1 \Phi_2^{N_2 - 3} ( \Phi_1^{m_1 - 1} \Phi_0^{m_0} - \Phi_2 )    & \Phi_1 \Phi_2^{N_2 - 3} \Phi_1^{m_1 - 2} ( \Phi_0^{m_0} - \Phi_1 )    & \Phi_1 \Phi_2^{N_2 - 3} \Phi_1^{m_1 - 3} ( \Phi_0^{m_0} - \Phi_1 )    \\
        \vdots    & \vdots    & \vdots    \\
        \Phi_1^{m_1 - 1} \Phi_2^{N_2 - 3} ( \Phi_1^{m_1 - 1} \Phi_0^{m_0} - \Phi_2 )    & \Phi_1^{m_1 - 1} \Phi_2^{N_2 - 3} \Phi_1^{m_1 - 2} ( \Phi_0^{m_0} - \Phi_1 )    & \Phi_1^{m_1 - 1} \Phi_2^{N_2 - 3} \Phi_1^{m_1 - 3} ( \Phi_0^{m_0} - \Phi_1 )    \\
        \Phi_2^{N_2 - 2} ( \Phi_1^{m_1 - 1} \Phi_0^{m_0} - \Phi_2 )    & \Phi_2^{N_2 - 2} \Phi_1^{m_1 - 2} ( \Phi_0^{m_0} - \Phi_1 )    & \Phi_2^{N_2 - 2} \Phi_1^{m_1 - 3} ( \Phi_0^{m_0} - \Phi_1 )    & \cdots    \\
    \end{bmatrix}.
$}
    \label{E-F-3-level-evaluated-eq}
\end{equation}

\subsubsection{Four-level V-cycles with F-relaxation}\label{four-level-V-cycles-with-F-relaxation-appsec}

Evaluating the error propagator in Equation \eqref{E-F-nl-eqn} for a four-level V-cycle with F-relaxation
on the coarse-grid yields,
\begin{alignat*}{4}
    &\mathcal{E}_{F}^{n_\ell = 4,\Delta}
    =  R_{I_0} \mathcal{E}_F^{n_\ell = 4} P_0 \\
    \quad&= I
    - \left[ P_1 P_2 A_3^{-1} R_2 R_1
    + S_1 (S_1^T A_1 S_1)^{-1} S_1^T
    + P_1 S_2 (S_2^T A_2 S_2)^{-1} S_2^T R_1 \right] R_0 A_0 P_0,
\end{alignat*}
where,
\begin{center}
\begin{sideways}
\parbox{\textheight}{
\begin{alignat*}{4}
    P_1 & S_2 (S_2^T A_2 S_2)^{-1} S_2^T R_1
    &= \resizebox{.75\hsize}{!}{$
    \begin{bmatrix}
        0 \\
        \vdots \\
        0
        & \Phi_1^{m_1 - 1}                                   & \cdots & I                                 \\
        & \vdots                                             & \cdots & \vdots                            \\
        & \Phi_1^{m_1 - 1} \Phi_1^{m_1 - 1}                  & \cdots & \Phi_1^{m_1 - 1}                  \\
        & \Phi_2 \Phi_1^{m_1 - 1}                            & \cdots & \Phi_2                            & \Phi_1^{m_1 - 1}                                   & \cdots & I                                 \\
        & \vdots                                             &        & \vdots                            & \vdots                                             &        & \vdots                            \\
        & \Phi_1^{m_1 - 1} \Phi_2 \Phi_1^{m_1 - 1}           & \cdots & \Phi_1^{m_1 - 1} \Phi_2           & \Phi_1^{m_1 - 1} \Phi_1^{m_1 - 1}                  & \cdots & \Phi_1^{m_1 - 1}                  \\
        & \vdots                                             &        & \vdots                            &                                                    &        &                                   &    \ddots    \\
        & \vdots                                             &        & \vdots                            &                                                    &        &                                   &    \\
        & \Phi_2^{m_2 - 2} \Phi_1^{m_1 - 1}                  & \cdots & \Phi_2^{m_2 - 2}                  & \Phi_2^{m_2 - 3} \Phi_1^{m_1 - 1}                  & \cdots & \Phi_2^{m_2 - 3}                  & \cdots    & \cdots    & I    \\
        & \vdots                                             &        & \vdots                            & \vdots                                             &        & \vdots                            &           &           & \vdots \\
        & \Phi_1^{m_1 - 1} \Phi_2^{m_2 - 2} \Phi_1^{m_1 - 1} & \cdots & \Phi_1^{m_1 - 1} \Phi_2^{m_2 - 2} & \Phi_1^{m_1 - 1} \Phi_2^{m_2 - 3} \Phi_1^{m_1 - 1} & \cdots & \Phi_1^{m_1 - 1} \Phi_2^{m_2 - 3} & \cdots    & \cdots    & \Phi_1^{m_1 - 1}    \\
        &&&&&&&&&& 0    \\
        &&&&&&&&&&& \ddots    \\
        &&&&&&&&&&&& 0    \\
    \end{bmatrix}
    $}
    \numberthis
    \label{PinvAffR-eq}
\end{alignat*}
}
\end{sideways}
\end{center}
and,
\begin{alignat}{4}
\resizebox{.9\hsize}{!}{$
    P_1 P_2 A_3^{-1} R_2 R_1
    =
    \begin{bmatrix}
        I                                                  \\
        \vdots                                             \\
        \Phi_1^{m_1 - 1}                                   \\
        \Phi_2                                             \\
        \vdots                                             \\
        \Phi_1^{m_1 - 1} \Phi_2^{m_2 - 1}                  \\
        \Phi_3                                             & \Phi_2^{m_2 - 1} \Phi_1^{m_1 - 1}                                                    & \cdots & \Phi_2^{m_2 - 1}                                                    & \Phi_2^{m_2 - 2} \Phi_1^{m_1 - 1}                                                    & \cdots  & I                                                  \\
        \vdots                                             & \vdots                                                                               &        & \vdots                                                              & \vdots                                                                               &         & \vdots                                             \\
        \Phi_1^{m_1 - 1} \Phi_3                            & \Phi_1^{m_1 - 1} \Phi_2^{m_2 - 1} \Phi_1^{m_1 - 1}                                   & \cdots & \Phi_1^{m_1 - 1} \Phi_2^{m_2 - 1}                                   & \Phi_1^{m_1 - 1} \Phi_2^{m_2 - 2} \Phi_1^{m_1 - 1}                                   & \cdots  & \Phi_1^{m_1 - 1}                                   \\
        \vdots                                             & \vdots                                                                               &        & \vdots                                                              & \vdots                                                                               &         & \vdots                                             \\
        \vdots                                             & \vdots                                                                               &        & \vdots                                                              & \vdots                                                                               &         & \vdots                                             \\
        \Phi_2^{m_2 - 1} \Phi_3                            & \Phi_2^{m_2 - 1} \Phi_2^{m_2 - 1} \Phi_1^{m_1 - 1}                                   & \cdots & \Phi_2^{m_2 - 1} \Phi_2^{m_2 - 1}                                   & \Phi_2^{m_2 - 1} \Phi_2^{m_2 - 2} \Phi_1^{m_1 - 1}                                   & \cdots  & \Phi_2^{m_2 - 1}                                   \\
        \vdots                                             & \vdots                                                                               &        & \vdots                                                              & \vdots                                                                               &         & \vdots                                             \\
        \Phi_1^{m_1 - 1} \Phi_2^{m_2 - 1} \Phi_3           & \Phi_1^{m_1 - 1} \Phi_2^{m_2 - 1} \Phi_2^{m_2 - 1} \Phi_1^{m_1 - 1}                  & \cdots & \Phi_1^{m_1 - 1} \Phi_2^{m_2 - 1} \Phi_2^{m_2 - 1}                  & \Phi_1^{m_1 - 1} \Phi_2^{m_2 - 1} \Phi_2^{m_2 - 2} \Phi_1^{m_1 - 1}                  & \cdots  & \Phi_1^{m_1 - 1} \Phi_2^{m_2 - 1}                  \\
        \vdots                                             & \vdots                                                                               &        & \vdots                                                              & \vdots                                                                               &         & \vdots & \ddots                                    \\
        \vdots                                             & \vdots                                                                               &        & \vdots                                                              & \vdots                                                                               &         & \vdots                                             \\
        \Phi_1^{m_1 - 1} \Phi_3^{N_3 - 2}                  & \Phi_1^{m_1 - 1} \Phi_3^{N_3 - 3} \Phi_2^{m_2 - 1} \Phi_1^{m_1 - 1}                  & \cdots & \Phi_1^{m_1 - 1} \Phi_3^{N_3 - 3} \Phi_2^{m_2 - 1}                  & \Phi_1^{m_1 - 1} \Phi_3^{N_3 - 3} \Phi_2^{m_2 - 2} \Phi_1^{m_1 - 1}                  & \cdots  & \Phi_1^{m_1 - 1} \Phi_3^{N_3 - 3}                  \\
        \Phi_2 \Phi_3^{N_3 - 2}                            & \Phi_2 \Phi_3^{N_3 - 3} \Phi_2^{m_2 - 1} \Phi_1^{m_1 - 1}                            & \cdots & \Phi_2 \Phi_3^{N_3 - 3} \Phi_2^{m_2 - 1}                            & \Phi_2 \Phi_3^{N_3 - 3} \Phi_2^{m_2 - 2} \Phi_1^{m_1 - 1}                            & \cdots  & \Phi_2 \Phi_3^{N_3 - 3}                            \\
        \vdots                                             & \vdots                                                                               &        & \vdots                                                              & \vdots                                                                               &         & \vdots                                             \\
        \Phi_1^{m_1 - 1} \Phi_2 \Phi_3^{N_3 - 2}           & \Phi_1^{m_1 - 1} \Phi_2 \Phi_3^{N_3 - 3} \Phi_2^{m_2 - 1} \Phi_1^{m_1 - 1}           & \cdots & \Phi_1^{m_1 - 1} \Phi_2 \Phi_3^{N_3 - 3} \Phi_2^{m_2 - 1}           & \Phi_1^{m_1 - 1} \Phi_2 \Phi_3^{N_3 - 3} \Phi_2^{m_2 - 2} \Phi_1^{m_1 - 1}           & \cdots  & \Phi_1^{m_1 - 1} \Phi_2 \Phi_3^{N_3 - 3}           \\
        \vdots                                             & \vdots                                                                               &        & \vdots                                                              & \vdots                                                                               &         & \vdots                                             \\
        \vdots                                             & \vdots                                                                               &        & \vdots                                                              & \vdots                                                                               &         & \vdots                                             \\
        \Phi_2^{m_2 - 1} \Phi_3^{N_3 - 2}                  & \Phi_2^{m_2 - 1} \Phi_3^{N_3 - 3} \Phi_2^{m_2 - 1} \Phi_1^{m_1 - 1}                  & \cdots & \Phi_2^{m_2 - 1} \Phi_3^{N_3 - 3} \Phi_2^{m_2 - 1}                  & \Phi_2^{m_2 - 1} \Phi_3^{N_3 - 3} \Phi_2^{m_2 - 2} \Phi_1^{m_1 - 1}                  & \cdots  & \Phi_2^{m_2 - 1} \Phi_3^{N_3 - 3}                  \\
        \vdots                                             & \vdots                                                                               &        & \vdots                                                              & \vdots                                                                               &         & \vdots                                             \\
        \Phi_1^{m_1 - 1} \Phi_2^{m_2 - 1} \Phi_3^{N_3 - 2} & \Phi_1^{m_1 - 1} \Phi_2^{m_2 - 1} \Phi_3^{N_3 - 3} \Phi_2^{m_2 - 1} \Phi_1^{m_1 - 1} & \cdots & \Phi_1^{m_1 - 1} \Phi_2^{m_2 - 1} \Phi_3^{N_3 - 3} \Phi_2^{m_2 - 1} & \Phi_1^{m_1 - 1} \Phi_2^{m_2 - 1} \Phi_3^{N_3 - 3} \Phi_2^{m_2 - 2} \Phi_1^{m_1 - 1} & \cdots  & \Phi_1^{m_1 - 1} \Phi_2^{m_2 - 1} \Phi_3^{N_3 - 3} \\
        \Phi_3^{N_3 - 1}                                   & \Phi_3^{N_3 - 2} \Phi_2^{m_2 - 1} \Phi_1^{m_1 - 1}                                   & \cdots & \Phi_3^{N_3 - 2} \Phi_2^{m_2 - 1}                                   & \Phi_3^{N_3 - 2} \Phi_2^{m_2 - 2} \Phi_1^{m_1 - 1}                                   & \cdots  & \Phi_3^{N_3 - 2}                                   \\
    \end{bmatrix}.
    $}
    \label{PPinvARR-eq}
\end{alignat}
The sum of \eqref{invAff-eq}, \eqref{PinvAffR-eq} and \eqref{PPinvARR-eq} yields,
\begin{alignat*}{4}
    P_1 & P_2 A_3^{-1} R_2 R_1
    + S_1 (S_1^T A_1 S_1)^{-1} S_1^T
    + P_1 S_2 (S_2^T A_2 S_2)^{-1} S_2^T R_1    \\
    &= \resizebox{.8\hsize}{!}{$
    \begin{bmatrix}
        I                                                  \\
        \Phi_1                                             & I                                                                                    \\
        \vdots                                             & \vdots                                                                               & \ddots \\
        \Phi_1^{m_1 - 1}                                   & \Phi_1^{m_1 - 2}                                                                     &        \\
        \Phi_2                                             & \Phi_1^{m_1 - 1}                                                                     & \cdots & I     &  \\
        \vdots                                             & \vdots                                                                               &        & \vdots & \ddots \\
        \vdots                                             & \vdots                                                                               &        & \vdots \\
        \Phi_1^{m_1 - 1} \Phi_2^{m_2 - 1}                  & \Phi_1^{m_1 - 1} \Phi_2^{m_2 - 2} \Phi_1^{m_1 - 1}                                   & \cdots & \Phi_1^{m_1 - 1} \Phi_2^{m_2 - 2}                                   & \Phi_1^{m_1 - 1} \Phi_2^{m_2 - 3} \Phi_1^{m_1 - 1}                                   & \ddots \\
        \Phi_3                                             & \Phi_2^{m_2 - 1} \Phi_1^{m_1 - 1}                                                    & \cdots & \Phi_2^{m_2 - 1}                                                    & \Phi_2^{m_2 - 2} \Phi_1^{m_1 - 1}                                                    & \cdots & I                                                  \\
        \vdots                                             & \vdots                                                                               &        & \vdots                                                              & \vdots                                                                               &        & \vdots                                             \\
        \Phi_1^{m_1 - 1} \Phi_3                            & \Phi_1^{m_1 - 1} \Phi_2^{m_2 - 1} \Phi_1^{m_1 - 1}                                   & \cdots & \Phi_1^{m_1 - 1} \Phi_2^{m_2 - 1}                                   & \Phi_1^{m_1 - 1} \Phi_2^{m_2 - 2} \Phi_1^{m_1 - 1}                                   & \cdots & \Phi_1^{m_1 - 1}                                   \\
        \vdots                                             & \vdots                                                                               &        & \vdots                                                              & \vdots                                                                               &        & \vdots                                             \\
        \vdots                                             & \vdots                                                                               &        & \vdots                                                              & \vdots                                                                               &        & \vdots                                             \\
        \Phi_2^{m_2 - 1} \Phi_3                            & \Phi_2^{m_2 - 1} \Phi_2^{m_2 - 1} \Phi_1^{m_1 - 1}                                   & \cdots & \Phi_2^{m_2 - 1} \Phi_2^{m_2 - 1}                                   & \Phi_2^{m_2 - 1} \Phi_2^{m_2 - 2} \Phi_1^{m_1 - 1}                                   & \cdots & \Phi_2^{m_2 - 1}                                   \\
        \vdots                                             & \vdots                                                                               &        & \vdots                                                              & \vdots                                                                               &        & \vdots                                             \\
        \Phi_1^{m_1 - 1} \Phi_2^{m_2 - 1} \Phi_3           & \Phi_1^{m_1 - 1} \Phi_2^{m_2 - 1} \Phi_2^{m_2 - 1} \Phi_1^{m_1 - 1}                  & \cdots & \Phi_1^{m_1 - 1} \Phi_2^{m_2 - 1} \Phi_2^{m_2 - 1}                  & \Phi_1^{m_1 - 1} \Phi_2^{m_2 - 1} \Phi_2^{m_2 - 2} \Phi_1^{m_1 - 1}                  & \cdots & \Phi_1^{m_1 - 1} \Phi_2^{m_2 - 1}                  \\
        \vdots                                             & \vdots                                                                               &        & \vdots                                                              & \vdots                                                                               &        & \vdots & \ddots                                    \\
        \vdots                                             & \vdots                                                                               &        & \vdots                                                              & \vdots                                                                               &        & \vdots                                             \\
        \Phi_1^{m_1 - 1} \Phi_3^{N_3 - 2}                  & \Phi_1^{m_1 - 1} \Phi_3^{N_3 - 3} \Phi_2^{m_2 - 1} \Phi_1^{m_1 - 1}                  & \cdots & \Phi_1^{m_1 - 1} \Phi_3^{N_3 - 3} \Phi_2^{m_2 - 1}                  & \Phi_1^{m_1 - 1} \Phi_3^{N_3 - 3} \Phi_2^{m_2 - 2} \Phi_1^{m_1 - 1}                  & \cdots & \Phi_1^{m_1 - 1} \Phi_3^{N_3 - 3}                  \\
        \Phi_2 \Phi_3^{N_3 - 2}                            & \Phi_2 \Phi_3^{N_3 - 3} \Phi_2^{m_2 - 1} \Phi_1^{m_1 - 1}                            & \cdots & \Phi_2 \Phi_3^{N_3 - 3} \Phi_2^{m_2 - 1}                            & \Phi_2 \Phi_3^{N_3 - 3} \Phi_2^{m_2 - 2} \Phi_1^{m_1 - 1}                            & \cdots & \Phi_2 \Phi_3^{N_3 - 3}                            \\
        \vdots                                             & \vdots                                                                               &        & \vdots                                                              & \vdots                                                                               &        & \vdots                                             \\
        \Phi_1^{m_1 - 1} \Phi_2 \Phi_3^{N_3 - 2}           & \Phi_1^{m_1 - 1} \Phi_2 \Phi_3^{N_3 - 3} \Phi_2^{m_2 - 1} \Phi_1^{m_1 - 1}           & \cdots & \Phi_1^{m_1 - 1} \Phi_2 \Phi_3^{N_3 - 3} \Phi_2^{m_2 - 1}           & \Phi_1^{m_1 - 1} \Phi_2 \Phi_3^{N_3 - 3} \Phi_2^{m_2 - 2} \Phi_1^{m_1 - 1}           & \cdots & \Phi_1^{m_1 - 1} \Phi_2 \Phi_3^{N_3 - 3}           \\
        \vdots                                             & \vdots                                                                               &        & \vdots                                                              & \vdots                                                                               &        & \vdots                                             \\
        \vdots                                             & \vdots                                                                               &        & \vdots                                                              & \vdots                                                                               &        & \vdots                                             \\
        \Phi_2^{m_2 - 1} \Phi_3^{N_3 - 2}                  & \Phi_2^{m_2 - 1} \Phi_3^{N_3 - 3} \Phi_2^{m_2 - 1} \Phi_1^{m_1 - 1}                  & \cdots & \Phi_2^{m_2 - 1} \Phi_3^{N_3 - 3} \Phi_2^{m_2 - 1}                  & \Phi_2^{m_2 - 1} \Phi_3^{N_3 - 3} \Phi_2^{m_2 - 2} \Phi_1^{m_1 - 1}                  & \cdots & \Phi_2^{m_2 - 1} \Phi_3^{N_3 - 3}                  \\
        \vdots                                             & \vdots                                                                               &        & \vdots                                                              & \vdots                                                                               &        & \vdots                                             \\
        \Phi_1^{m_1 - 1} \Phi_2^{m_2 - 1} \Phi_3^{N_3 - 2} & \Phi_1^{m_1 - 1} \Phi_2^{m_2 - 1} \Phi_3^{N_3 - 3} \Phi_2^{m_2 - 1} \Phi_1^{m_1 - 1} & \cdots & \Phi_1^{m_1 - 1} \Phi_2^{m_2 - 1} \Phi_3^{N_3 - 3} \Phi_2^{m_2 - 1} & \Phi_1^{m_1 - 1} \Phi_2^{m_2 - 1} \Phi_3^{N_3 - 3} \Phi_2^{m_2 - 2} \Phi_1^{m_1 - 1} & \cdots & \Phi_1^{m_1 - 1} \Phi_2^{m_2 - 1} \Phi_3^{N_3 - 3} \\
        \Phi_3^{N_3 - 1}                                   & \Phi_3^{N_3 - 2} \Phi_2^{m_2 - 1} \Phi_1^{m_1 - 1}                                   & \cdots & \Phi_3^{N_3 - 2} \Phi_2^{m_2 - 1}                                   & \Phi_3^{N_3 - 2} \Phi_2^{m_2 - 2} \Phi_1^{m_1 - 1}                                   & \cdots & \Phi_3^{N_3 - 2}                                   \\
    \end{bmatrix},
    $} \numberthis
    \label{PPinvARR-plus-invAff-plus-PinvAffR-eq}
\end{alignat*}
where we notice that the sparsity patterns are non-overlapping. Then, the error propagator is given as,

\begin{center}
\begin{sideways}
\parbox{\textheight}{
\begin{alignat*}{4}
    \mathcal{E}_{F}^{n_\ell = 4,\Delta}
    &= \resizebox{.8\hsize}{!}{$
    - \begin{bmatrix}
        0                                                                                                            \\
        \Phi_1 - \Phi_0^{m_0}                                                                                        & 0                                                                                                            \\
        \Phi_1 (\Phi_1 - \Phi_0^{m_0})                                                                               & \Phi_1 - \Phi_0^{m_0}                                                                                        \\
        \vdots                                                                                                       & \vdots                                                                                                       & \ddots \\
        \Phi_1^{m_1 - 2} (\Phi_1 - \Phi_0^{m_0})                                                                     & \Phi_1^{m_1 - 3} (\Phi_1 - \Phi_0^{m_0})                                                                     & \cdots & \\
        \Phi_2 - \Phi_1^{m_1 - 1} \Phi_0^{m_0}                                                                       & \Phi_1^{m_1 - 2} (\Phi_1 - \Phi_0^{m_0})                                                                     & \cdots & 0 \\
        \vdots                                                                                                       & \vdots                                                                                                       &        & \vdots                                                                                                      \\
        \vdots                                                                                                       & \vdots                                                                                                       &        & \vdots                                                                                                      & \ddots \\
        \Phi_1^{m_1 - 1} \Phi_2^{m_2 - 2} (\Phi_2 - \Phi_1^{m_1 - 1} \Phi_0^{m_0})                                   & \Phi_1^{m_1 - 1} \Phi_2^{m_2 - 2} \Phi_1^{m_1 - 2} (\Phi_1 - \Phi_0^{m_0})                                   & \cdots & \Phi_1^{m_1 - 1} \Phi_2^{m_2 - 3} (\Phi_2 - \Phi_1^{m_1 - 1} \Phi_0^{m_0})                                  &        & \ddots \\
        \Phi_3 - \Phi_2^{m_2 - 1} \Phi_1^{m_1 - 1} \Phi_0^{m_0}                                                      & \Phi_2^{m_2 - 1} \Phi_1^{m_1 - 2} (\Phi_1 - \Phi_0^{m_0})                                                    & \cdots & \Phi_2^{m_2 - 2} (\Phi_2 - \Phi_1^{m_1 - 1} \Phi_0^{m_0})                                                   & \cdots & \cdots & 0 \\
        \vdots                                                                                                       & \vdots                                                                                                       &        & \vdots                                                                                                      &        &        & \vdots & \ddots \\
        \Phi_1^{m_1 - 1} (\Phi_3 - \Phi_2^{m_2 - 1} \Phi_1^{m_1 - 1} \Phi_0^{m_0})                                   & \Phi_1^{m_1 - 1} \Phi_2^{m_2 - 1} \Phi_1^{m_1 - 2} (\Phi_1 - \Phi_0^{m_0})                                   & \cdots & \Phi_1^{m_1 - 1} \Phi_2^{m_2 - 2} (\Phi_2 - \Phi_1^{m_1 - 1} \Phi_0^{m_0})                                  & \cdots & \cdots & \Phi_1^{m_1 - 2} (\Phi_1 - \Phi_0^{m_0}) \\
        \vdots                                                                                                       & \vdots                                                                                                       &        & \vdots                                                                                                      &        &        & \vdots \\
        \vdots                                                                                                       & \vdots                                                                                                       &        & \vdots                                                                                                      &        &        & \vdots \\
        \Phi_2^{m_2 - 1} (\Phi_3 - \Phi_2^{m_2 - 1} \Phi_1^{m_1 - 1} \Phi_0^{m_0})                                   & \Phi_2^{m_2 - 1} \Phi_2^{m_2 - 1} \Phi_1^{m_1 - 2} (\Phi_1 - \Phi_0^{m_0})                                   & \cdots & \Phi_2^{m_2 - 1} \Phi_2^{m_2 - 2} (\Phi_2 - \Phi_1^{m_1 - 1} \Phi_0^{m_0})                                  & \cdots &        & \Phi_2^{m_2 - 2} (\Phi_2 - \Phi_1^{m_1 - 1} \Phi_0^{m_0}) \\
        \vdots                                                                                                       & \vdots                                                                                                       &        & \vdots                                                                                                      &        &        & \vdots \\
        \Phi_1^{m_1 - 1} \Phi_2^{m_2 - 1} (\Phi_3 - \Phi_2^{m_2 - 1} \Phi_1^{m_1 - 1} \Phi_0^{m_0})                  & \Phi_1^{m_1 - 1} \Phi_2^{m_2 - 1} \Phi_2^{m_2 - 1} \Phi_1^{m_1 - 2} (\Phi_1 - \Phi_0^{m_0})                  & \cdots & \Phi_1^{m_1 - 1} \Phi_2^{m_2 - 1} \Phi_2^{m_2 - 2} (\Phi_2 - \Phi_1^{m_1 - 1} \Phi_0^{m_0})                 & \cdots &        & \Phi_1^{m_1 - 1} \Phi_2^{m_2 - 2} (\Phi_2 - \Phi_1^{m_1 - 1} \Phi_0^{m_0}) \\
        \vdots                                                                                                       & \vdots                                                                                                       &        & \vdots                                                                                                      &        &        & \vdots \\
        \vdots                                                                                                       & \vdots                                                                                                       &        & \vdots                                                                                                      &        &        & \vdots \\
        \Phi_1^{m_1 - 1} \Phi_3^{N_3 - 3} (\Phi_3 - \Phi_2^{m_2 - 1} \Phi_1^{m_1 - 1} \Phi_0^{m_0})                  & \Phi_1^{m_1 - 1} \Phi_3^{N_3 - 3} \Phi_2^{m_2 - 1} \Phi_1^{m_1 - 2} (\Phi_1 - \Phi_0^{m_0})                  & \cdots & \Phi_1^{m_1 - 1} \Phi_3^{N_3 - 3} \Phi_2^{m_2 - 2} (\Phi_2 - \Phi_1^{m_1 - 1} \Phi_0^{m_0}                  & \cdots &        & \Phi_1^{m_1 - 1} \Phi_3^{N_3 - 4} (\Phi_3 - \Phi_2^{m_2 - 1} \Phi_1^{m_1 - 1} \Phi_0^{m_0}) \\
        \Phi_2 \Phi_3^{N_3 - 3} (\Phi_3 - \Phi_2^{m_2 - 1} \Phi_1^{m_1 - 1} \Phi_0^{m_0})                            & \Phi_2 \Phi_3^{N_3 - 3} \Phi_2^{m_2 - 1} \Phi_1^{m_1 - 2} (\Phi_1 - \Phi_0^{m_0})                            & \cdots & \Phi_2 \Phi_3^{N_3 - 3} \Phi_2^{m_2 - 2} (\Phi_2 - \Phi_1^{m_1 - 1} \Phi_0^{m_0}                            & \cdots &        & \Phi_2 \Phi_3^{N_3 - 4} (\Phi_3 - \Phi_2^{m_2 - 1} \Phi_1^{m_1 - 1} \Phi_0^{m_0}) \\
        \vdots                                                                                                       & \vdots                                                                                                       &        & \vdots                                                                                                      &        &        & \vdots \\
        \Phi_1^{m_1 - 1} \Phi_2 \Phi_3^{N_3 - 3} (\Phi_3 - \Phi_2^{m_2 - 1} \Phi_1^{m_1 - 1} \Phi_0^{m_0})           & \Phi_1^{m_1 - 1} \Phi_2 \Phi_3^{N_3 - 3} \Phi_2^{m_2 - 1} \Phi_1^{m_1 - 2} (\Phi_1 - \Phi_0^{m_0})           & \cdots & \Phi_1^{m_1 - 1} \Phi_2 \Phi_3^{N_3 - 3} \Phi_2^{m_2 - 2} (\Phi_2 - \Phi_1^{m_1 - 1} \Phi_0^{m_0}           & \cdots &        & \Phi_1^{m_1 - 1} \Phi_2 \Phi_3^{N_3 - 4} (\Phi_3 - \Phi_2^{m_2 - 1} \Phi_1^{m_1 - 1} \Phi_0^{m_0}) \\
        \vdots                                                                                                       & \vdots                                                                                                       &        & \vdots                                                                                                      &        &        & \vdots \\
        \vdots                                                                                                       & \vdots                                                                                                       &        & \vdots                                                                                                      &        &        & \vdots \\
        \Phi_2^{m_2 - 1} \Phi_3^{N_3 - 3} (\Phi_3 - \Phi_2^{m_2 - 1} \Phi_1^{m_1 - 1} \Phi_0^{m_0})                  & \Phi_2^{m_2 - 1} \Phi_3^{N_3 - 3} \Phi_2^{m_2 - 1} \Phi_1^{m_1 - 2} (\Phi_1 - \Phi_0^{m_0})                  & \cdots & \Phi_2^{m_2 - 1} \Phi_3^{N_3 - 3} \Phi_2^{m_2 - 2} (\Phi_2 - \Phi_1^{m_1 - 1} \Phi_0^{m_0}                  & \cdots &        & \Phi_2^{m_2 - 1} \Phi_3^{N_3 - 4} (\Phi_3 - \Phi_2^{m_2 - 1} \Phi_1^{m_1 - 1} \Phi_0^{m_0}) \\
        \vdots                                                                                                       & \vdots                                                                                                       &        & \vdots                                                                                                      &        &        & \vdots \\
        \Phi_1^{m_1 - 1} \Phi_2^{m_2 - 1} \Phi_3^{N_3 - 3} (\Phi_3 - \Phi_2^{m_2 - 1} \Phi_1^{m_1 - 1} \Phi_0^{m_0}) & \Phi_1^{m_1 - 1} \Phi_2^{m_2 - 1} \Phi_3^{N_3 - 3} \Phi_2^{m_2 - 1} \Phi_1^{m_1 - 2} (\Phi_1 - \Phi_0^{m_0}) & \cdots & \Phi_1^{m_1 - 1} \Phi_2^{m_2 - 1} \Phi_3^{N_3 - 3} \Phi_2^{m_2 - 2} (\Phi_2 - \Phi_1^{m_1 - 1} \Phi_0^{m_0} & \cdots &        & \Phi_1^{m_1 - 1} \Phi_2^{m_2 - 1} \Phi_3^{N_3 - 4} (\Phi_3 - \Phi_2^{m_2 - 1} \Phi_1^{m_1 - 1} \Phi_0^{m_0}) & \cdots & \cdots & 0 \\
        \Phi_3^{N_3 - 2} (\Phi_3 - \Phi_2^{m_2 - 1} \Phi_1^{m_1 - 1} \Phi_0^{m_0})                                   & \Phi_3^{N_3 - 2} \Phi_2^{m_2 - 1} \Phi_1^{m_1 - 2} (\Phi_1 - \Phi_0^{m_0})                                   & \cdots & \Phi_3^{N_3 - 2} \Phi_2^{m_2 - 2} (\Phi_2 - \Phi_1^{m_1 - 1} \Phi_0^{m_0}                                   & \cdots &        & \Phi_3^{N_3 - 3} (\Phi_3 - \Phi_2^{m_2 - 1} \Phi_1^{m_1 - 1} \Phi_0^{m_0})                                   & \cdots & \cdots & \Phi_1 - \Phi_0^{m_0} & 0\\
    \end{bmatrix}
    $} \numberthis
    \label{E-F-4-level-evaluated-eq}
\end{alignat*}
}
\end{sideways}
\end{center}

This yields the following result.

\begin{theorem}
    Let $\{ \Phi_\ell \}$ be simultaneously diagonalizable by the same unitary transformation $X$,
    with eigenvalues $\{ \lambda_{\ell,k} \}$, $| \lambda_{\ell,k} | < 1$.
    Then, the \emph{worst case} convergence factor of four-level MGRIT with F-relaxation
    is bounded by,
    \begin{alignat}{4}
        c_f &\leq \sqrt{\| \mathcal{E}_{F}^{n_\ell = 4, \Delta} \|_1 \| \mathcal{E}_{F}^{n_\ell = 4, \Delta} \|_\infty },
        \label{E-F-nl4-inequality-eqn}
    \end{alignat}
    and $\| \mathcal{E}_{F}^{n_\ell = 4, \Delta} \|_1$ and $\| \mathcal{E}_{F}^{n_\ell = 4, \Delta} \|_\infty$
    are given analytically as,
    \begin{alignat*}{4}
        \| \mathcal{E}_{F}^{n_\ell = 4, \Delta} \|_1
        &= \max_{\substack{1 \leq k \leq N_x\\0 \leq d \leq m_1 m_2 - 1}} s_d^{\text{col}} (k), &&\qquad
        \| \mathcal{E}_{F}^{n_\ell = 4, \Delta} \|_\infty
        &= \max_{\substack{1 \leq k \leq N_x\\0 \leq d \leq m_1 m_2 - 1}} s_d^{\text{row}} (k),
    \end{alignat*}
    where the column and row sums, $s_d^{\text{col}}$ and $s_d^{\text{row}}$ (row and column subscripts $d$),
    are defined as follows.
    The absolute column sums of the first CF-interval on level $1$ are given as,
    \begin{alignat*}{4}
        & s_0^{\text{col}} (k)
        = | \lambda_{3,k} |^{N_3 - 2}
        | \lambda_{3,k} - \lambda_{0,k}^{m_0} \lambda_{1,k}^{m_1 - 1} \lambda_{2,k}^{m_2 - 1} | \\
        &\quad + | \lambda_{3,k} - \lambda_{0,k}^{m_0} \lambda_{1,k}^{m_1 - 1} \lambda_{2,k}^{m_2 - 1} |
        \frac{1 - | \lambda_{3,k} |^{N_3 - 2}}{1 - | \lambda_{3,k} |}
        \frac{1 - | \lambda_{2,k} |^{m_2}}{1 - | \lambda_{2,k} |}
        \frac{1 - | \lambda_{1,k} |^{m_1}}{1 - | \lambda_{1,k} |}    \\
        &\quad + | \lambda_{2,k} - \lambda_{0,k}^{m_0} \lambda_{1,k}^{m_1 - 1} |
        \frac{1 - | \lambda_{2,k} |^{m_2 - 1}}{1 - | \lambda_{2,k} |}
        \frac{1 - | \lambda_{1,k} |^{m_1}}{1 - | \lambda_{1,k} |}
        + | \lambda_{1,k} - \lambda_{0,k}^{m_0} |
        \frac{1 - | \lambda_{1,k} |^{m_1 - 1}}{1 - | \lambda_{1,k} |},
    \end{alignat*}
    corresponding to the first C-point on level $1$. Next,
    \begin{alignat*}{4}
        & s_{m_1 (m_2 - j)}^{\text{col}} (k)
        =
\resizebox{0.8\hsize}{!}{$
        | \lambda_{2,k} - \lambda_{0,k}^{m_0} \lambda_{1,k}^{m_1 - 1} |
        \left(
        \sum_{p = 0}^{j - 2} | \lambda_{2,k} |^p
        \right)
        \frac{1 - | \lambda_{1,k} |^{m_1}}{1 - | \lambda_{1,k} |}
        + | \lambda_{1,k} - \lambda_{0,k}^{m_0} |
        \frac{1 - | \lambda_{1,k} |^{m_1 - 1}}{1 - | \lambda_{1,k} |}
        $} \\
        &\quad
\resizebox{0.95\hsize}{!}{$
        + | \lambda_{2,k} |^{j - 1}
        | \lambda_{2,k} - \lambda_{0,k}^{m_0}  \lambda_{1,k}^{m_1 - 1} |
        \left[
        | \lambda_{3,k} |^{N_3 - 2}
        + \frac{1 - | \lambda_{3,k} |^{N_3 - 2}}{1 - | \lambda_{3,k} |}
        \frac{1 - | \lambda_{2,k} |^{m_2}}{1 - | \lambda_{2,k} |}
        \frac{1 - | \lambda_{1,k} |^{m_1 - 1}}{1 - | \lambda_{1,k} |}
        \right]$},
    \end{alignat*}
    for $j = 1, \ldots, m_2 - 1$, corresponding to the interior level $2$ C-points of the first CF-interval
    on level $1$. Lastly,
    \begin{alignat*}{4}
        &s_{m_1 (m_2 - j) - r - 1}^{\text{col}} (k)
        =
        | \lambda_{1,k} - \lambda_{0,k}^{m_0} |
        \left[
        \left(
        \sum_{q = 0}^{m_1 - 1} | \lambda_{1,k} |^q
        \right)
        \left(
        \sum_{p = 0}^{j - 1} | \lambda_{2,k} |^p
        \right)
        + \left(
        \sum_{q = 0}^{r - 1} | \lambda_{1,k} |^q
        \right)
        \right] \\
        &\quad
\resizebox{0.95\hsize}{!}{$
        + | \lambda_{2,k} |^j
        | \lambda_{1,k} |^r
        | \lambda_{1,k} - \lambda_{0,k}^{m_0} |
        \left[
        | \lambda_{3,k} |^{N_3 - 2}
        +
        \frac{1 - | \lambda_{3,k} |^{N_3 - 2}}{1 - | \lambda_{3,k} |}
        \frac{1 - | \lambda_{2,k} |^{m_2}}{1 - | \lambda_{2,k} |}
        \frac{1 - | \lambda_{1,k} |^{m_1}}{1 - | \lambda_{1,k} |}
        \right]$},
    \end{alignat*}
    for $j = 0, \ldots, m_2 - 1$ and $r = 0, \ldots, m_1 - 2$, corresponding to the level $2$ F-points
    of the first CF-interval on level $1$.

    The absolute row sums of the last FC-interval on level $1$ are given as,
    \begin{alignat*}{4}
        s_{N_1 - 1}^{\text{row}} (k)
        =~&
        \frac{1 - | \lambda_{3,k} |^{N_3 - 1}}{1 - | \lambda_{3,k} |}
        \bigg[
        | \lambda_{1,k} - \lambda_{0,k}^{m_0} |
        \frac{1 - | \lambda_{2,k} |^{m_2}}{1 - | \lambda_{2,k} |}
        \frac{1 - | \lambda_{1,k} |^{m_1 - 1}}{1 - | \lambda_{1,k} |} \\
        &\quad + | \lambda_{2,k} - \lambda_{0,k}^{m_0}  \lambda_{1,k}^{m_1 - 1} |
        \frac{1 - | \lambda_{2,k} |^{m_2 - 1}}{1 - | \lambda_{2,k} |}
        + | \lambda_{3,k} - \lambda_{0,k}^{m_0} \lambda_{1,k}^{m_1 - 1} \lambda_{2,k}^{m_2 - 1} |
        \bigg],
    \end{alignat*}
    corresponding to the last C-point on level $1$. Next,
    \begin{alignat*}{4}
        &s_{N_1 - 1 - m_1 m_2 + j m_1}^{\text{row}} (k)
        = \frac{1 - | \lambda_{2,k} |^j}{1 - | \lambda_{2,k} |}
        \left[
        | \lambda_{1,k} - \lambda_{0,k}^{m_0} |
        \frac{1 - | \lambda_{1,k} |^{m_1 - 1}}{1 - | \lambda_{1,k} |}
        + | \lambda_{2,k} - \lambda_{0,k}^{m_0}  \lambda_{1,k}^{m_1 - 1} |
        \right]    \\
        &\quad + | \lambda_{2,k} |^j
        \frac{1 - | \lambda_{3,k} |^{N_3 - 2}}{1 - | \lambda_{3,k} |}
        \bigg[
        | \lambda_{1,k} - \lambda_{0,k}^{m_0} |
        \frac{1 - | \lambda_{2,k} |^{m_2}}{1 - | \lambda_{2,k} |}
        \frac{1 - | \lambda_{1,k} |^{m_1 - 1}}{1 - | \lambda_{1,k} |} \\
        &\quad + | \lambda_{2,k} - \lambda_{0,k}^{m_0}  \lambda_{1,k}^{m_1 - 1} |
        \frac{1 - | \lambda_{2,k} |^{m_2 - 1}}{1 - | \lambda_{2,k} |}
        + | \lambda_{3,k} - \lambda_{0,k}^{m_0} \lambda_{1,k}^{m_1 - 1} \lambda_{2,k}^{m_2 - 1} |
        \bigg],
    \end{alignat*}
    for $j = 1, \ldots, m_2 - 1$, corresponding to the interior C-points of the last FC-interval on level $1$. Lastly,
    \begin{alignat*}{4}
        &s_{N_1 - 1 - m_1 m_2 + r + j m_1}^{\text{row}} (k) = \\
        &\quad
\resizebox{0.95\hsize}{!}{$
        | \lambda_{1,k} |^r
        | \lambda_{2,k} |^j
        \frac{1 - | \lambda_{3,k} |^{N_3 - 2}}{1 - | \lambda_{3,k} |}
        \left[
        | \lambda_{2,k} - \lambda_{0,k}^{m_0}  \lambda_{1,k}^{m_1 - 1} |
        \frac{1 - | \lambda_{2,k} |^{m_2 - 1}}{1 - | \lambda_{2,k} |}
        + | \lambda_{3,k} - \lambda_{0,k}^{m_0} \lambda_{1,k}^{m_1 - 1} \lambda_{2,k}^{m_2 - 1} |
        \right] $}   \\
        &\quad
\resizebox{0.95\hsize}{!}{$ + | \lambda_{1,k} |^r
        \left( \sum_{q = 0}^{j - 1} | \lambda_{2,k} |^q \right)
        \left[
        | \lambda_{2,k} - \lambda_{0,k}^{m_0}  \lambda_{1,k}^{m_1 - 1} |
        + | \lambda_{1,k} - \lambda_{0,k}^{m_0} |
        \frac{1 - | \lambda_{1,k} |^{m_1 - 1}}{1 - | \lambda_{1,k} |}
        \right] $}   \\
        &\quad
\resizebox{0.95\hsize}{!}{$ + | \lambda_{1,k} |^r
        | \lambda_{1,k} - \lambda_{0,k}^{m_0} |
        \frac{1 - | \lambda_{3,k} |^{N_3 - 2}}{1 - | \lambda_{3,k} |}
        \frac{1 - | \lambda_{2,k} |^{m_2}}{1 - | \lambda_{2,k} |}
        \frac{1 - | \lambda_{1,k} |^{m_1 - 1}}{1 - | \lambda_{1,k} |}
        + | \lambda_{1,k} - \lambda_{0,k}^{m_0} |
        \frac{1 - | \lambda_{1,k} |^r}{1 - | \lambda_{1,k} |}$},
    \end{alignat*}
    for $j = 0, \ldots, m_2 - 1$ and $r = 1, \ldots, m_1 - 1$, corresponding to the F-points
    of the last FC-interval on level $1$.
    \label{cf-4-level-F-theo}
\end{theorem}

\begin{proof}
    The proof is analogous to Theorem \ref{cf-3-level-F-theo}.
\end{proof}

\begin{remark}
    We note, that evaluating the $2 m_1 m_2$ analytic formulae in Theorem \ref{cf-4-level-F-theo} significantly reduces
    the time complexity of evaluating Equation \eqref{E-F-nl4-inequality-eqn}
    compared to constructing $\mathcal{E}_F^{n_\ell = 4, \Delta}$ numerically and computing
    $\| \mathcal{E}_F^{n_\ell = 4, \Delta} \|_1$ and $\| \mathcal{E}_F^{n_\ell = 4, \Delta} \|_\infty$.
\end{remark}

\subsubsection{Three-level V-cycles with FCF-relaxation}
\label{three-level-V-cycles-with-FCF-relaxation-appsec}

Following the same approach as in the previous sections, we can find the following result
for three-level V-cycles with FCF-relaxation.
\begin{theorem}
    Let $\{ \Phi_\ell \}$ be simultaneously diagonalizable by the same unitary transformation $X$,
    with eigenvalues $\{ \lambda_{\ell,k} \}$, $| \lambda_{\ell,k} | < 1$.
    Then, the \emph{worst case} convergence factor of three-level MGRIT with FCF-relaxation
    is bounded by,
    \begin{alignat}{4}
        c_f &\leq \sqrt{\| \mathcal{E}_{FCF}^{n_\ell = 3, \Delta} \|_1 \| \mathcal{E}_{FCF}^{n_\ell = 3, \Delta} \|_\infty },
        \label{E-FCF-nl3-inequality-eqn}
    \end{alignat}
    and $\| \mathcal{E}_{F}^{n_\ell = 3, \Delta} \|_1$ and $\| \mathcal{E}_{F}^{n_\ell = 3, \Delta} \|_\infty$
    are given analytically as,
    \begin{alignat*}{4}
        \| \mathcal{E}_{FCF}^{n_\ell = 3, \Delta} \|_1
        &= \max_{\substack{1 \leq k \leq N_x\\0 \leq d \leq m_1 m_2 - 1}} s_d^{\text{col}} (k), &&\qquad
        \| \mathcal{E}_{FCF}^{n_\ell = 3, \Delta} \|_\infty
        &= \max_{\substack{1 \leq k \leq N_x\\0 \leq d \leq m_1 m_2 - 1}} s_d^{\text{row}} (k).
    \end{alignat*}
    The absolute column sums of the first CF-interval on level $1$ are given as,
    \begin{equation*}
\resizebox{0.975\hsize}{!}{$
        s_0^{\text{col}} (k)
        = | \lambda_{0,k} |^{m_0}
        | \lambda_{1,k} - \lambda_{0,k}^{m_0} |
        \left[
        | \lambda_{1,k} |^{m_1 - 1}
        \left(
        | \lambda_{2,k} |^{N_2 - 3}
        + \frac{1 - | \lambda_{2,k} |^{N_2 - 3}}{1 - | \lambda_{2,k} |}
        \frac{1 - | \lambda_{1,k} |^{m_1}}{1 - | \lambda_{1,k} |}
        \right)
        + \frac{1 - | \lambda_{1,k} |^{m_1 - 1}}{1 - | \lambda_{1,k} |}
        \right]$},
    \end{equation*}
    corresponding to the last F-point on level $1$. Next,
    \begin{alignat*}{4}
        &s_{m_1 - 2}^{\text{col}} (k)
        = | \lambda_{0,k} |^{m_0}
        \bigg[
        | \lambda_{1,k} - \lambda_{0,k}^{m_0} |
        \frac{1 - | \lambda_{1,k} |^{m_1}}{1 - | \lambda_{1,k} |} \\
        &\quad + | \lambda_{1,k} |
        | \lambda_{2,k} - \lambda_{0,k}^{m_0}  \lambda_{1,k}^{m_1 - 1} |
        \left(
        | \lambda_{2,k} |^{N_2 - 3}
        + \frac{1 - | \lambda_{2,k} |^{N_2 - 3}}{1 - | \lambda_{2,k} |}
        \frac{1 - | \lambda_{1,k} |^{m_1}}{1 - | \lambda_{1,k} |}
        \right)
        \bigg],
    \end{alignat*}
    corresponding to the first C-point on level $1$ if $m_1 = 2$,
    or the penultimate F-point on level $1$ if $m_1 > 2$. Lastly, if $m_1 > 2$,
    \begin{equation*}
\resizebox{0.975\hsize}{!}{$
        s_{m_1 - 2 - j}^{\text{col}} (k)
        = | \lambda_{0,k} |^{m_0}
        | \lambda_{1,k} |^j
        | \lambda_{1,k} - \lambda_{0,k}^{m_0} |
        \left[
        | \lambda_{2,k} |^{N_2 - 2}
        + \frac{1 - | \lambda_{1,k} |^j}{1 - | \lambda_{1,k} |}
        + \frac{1 - | \lambda_{2,k} |^{N_2 - 2}}{1 - | \lambda_{2,k} |}
        \frac{1 - | \lambda_{1,k} |^{m_1}}{1 - | \lambda_{1,k} |}
        \right]$},
    \end{equation*}
    for $j = 1, \ldots, m_1 - 2$,
    corresponding to the first C-point and the following F-points on level $1$.

    The absolute row sums of the last FC-interval on level $1$ are given as,
    \begin{alignat*}{4}
        &s_{N_1 - 1}^{\text{row}} (k)
        = | \lambda_{0,k} |^{m_0}
        | \lambda_{1,k} |
        | \lambda_{2,k} - \lambda_{0,k}^{m_0}  \lambda_{1,k}^{m_1 - 1} |
        \frac{1 - | \lambda_{2,k} |^{N_2 - 2}}{1 - | \lambda_{2,k} |}    \\
        &\quad
\resizebox{0.95\hsize}{!}{$
        + | \lambda_{0,k} |^{m_0}
        \left[
        | \lambda_{1,k} - \lambda_{0,k}^{m_0} |
        \left(
        1
        + | \lambda_{1,k} |
        \frac{1 - | \lambda_{2,k} |^{N_2 - 2}}{1 - | \lambda_{2,k} |}
        \frac{1 - | \lambda_{1,k} |^{m_1 - 1}}{1 - | \lambda_{1,k} |}
        + | \lambda_{1,k} |
        | \lambda_{2,k} |^{N_2 - 2}
        \left( \sum_{q = 0}^{m_1 - 3} | \lambda_{1,k} |^q \right)
        \right)
        \right]$},
    \end{alignat*}
    corresponding to the last C-point on level $1$, and,
    \begin{alignat*}{4}
        &s_{N_1 - m_1 + j}^{\text{row}} (k)
        = | \lambda_{0,k} |^{m_0}
        \bigg[
        | \lambda_{1,k} - \lambda_{0,k}^{m_0} |
        \frac{1 - | \lambda_{1,k} |^{j + 2}}{1 - | \lambda_{1,k} |} \\
        &\quad + | \lambda_{1,k} |^{j + 2}
        \frac{1 - | \lambda_{2,k} |^{N_2 - 3}}{1 - | \lambda_{2,k} |}
        \left(
        | \lambda_{2,k} - \lambda_{0,k}^{m_0}  \lambda_{1,k}^{m_1 - 1} |
        + | \lambda_{1,k} - \lambda_{0,k}^{m_0} |
        \frac{1 - | \lambda_{1,k} |^{m_1 - 1}}{1 - | \lambda_{1,k} |}
        \right)
        \bigg],
    \end{alignat*}
    for $j = 0, \ldots, m_1 - 2$,corresponding to the preceding F-points on level $1$.
    \label{cf-3-level-FCF-theo}
\end{theorem}

\begin{proof}
    The proof is analogous to Theorem \ref{cf-3-level-F-theo}.
\end{proof}

\section{Butcher tableaux of SDIRK schemes}\label{butcher-tableaux-appsec}
~
\FloatBarrier
\begin{table}[h]
    \begin{center}
        \begin{tabular}{ c | c }
            $1$ & $1$ \\ \hline
                & $1$
        \end{tabular} \qquad
        \begin{tabular}{ c | c  c }
            $1-\gamma$ & $1-\gamma$  & $0$ \\
            $\gamma$   & $2\gamma-1$ & $1-\gamma$ \\ \hline
                       & $1/2$       & $1/2$
        \end{tabular} \qquad
        \begin{tabular}{ c | c  c  c }
            $q$ & $q$   & $0$     & $0$ \\
            $s$ & $s-q$ & $q$     & $0$ \\
            $1$ & $r$   & $1-q-r$ & $q$ \\ \hline
                & $r$   & $1-q-r$ & $q$
        \end{tabular}
        \caption{Butcher tableau for L-stable SDIRK scheme of orders 1 - 3 with
            $\gamma = 1/\sqrt{2}$, $q = 0.4358665215\ldots$, $r = 1.2084966491\ldots$
            and $s = 0.7179332607\ldots$; See \cite{DobrevKolevPeterssonSchroder2017}.}
    \end{center}
    \label{butcher-tableau-L-SDIRK123-tab}
\end{table}

\begin{table}[h]
    \begin{center}
        \begin{tabular}{ c | c  c  c  c  c }
        $1/4$   & $1/4$      & $0$         & $0$      & $0$      & $0$ \\
        $3/4$   & $1/2$      & $1/4$       & $0$      & $0$      & $0$ \\
        $11/20$ & $17/50$    & $-1/25$     & $1/4$    & $0$      & $0$ \\
        $1/2$   & $371/1360$ & $-137/2720$ & $15/544$ & $1/4$    & $0$ \\
        $1$     & $25/24$    & $-49/48$    & $125/16$ & $-85/12$ & $1/4$ \\ \hline
                & $25/24$    & $-49/48$    & $125/16$ & $-85/12$ & $1/4$
        \end{tabular}
        \caption{Butcher tableau for L-stable SDIRK scheme of orders 4; See \cite{DuarteDobbinsSmooke2016}, Appendix C.}
    \end{center}
    \label{butcher-tableau-L-SDIRK4-tab}
\end{table}

\begin{table}[h]
    \begin{center}
        \begin{tabular}{ c | c  c }
            $1/4$ & $1/4$ & $0$   \\
            $3/4$ & $1/2$ & $1/4$ \\ \hline
                  & $1/2$ & $1/2$
        \end{tabular} \qquad
        \begin{tabular}{ c | c  c  c }
            $\gamma$   & $\gamma$    & $0$ \\
            $1-\gamma$ & $1-2\gamma$ & $\gamma$ \\ \hline
                       & $1/2$       & $1/2$
        \end{tabular} \qquad
        \begin{tabular}{ c | c  c  c }
            $q$   & $q$     & $0$    & $0$ \\
            $1/2$ & $1/2-q$ & $q$    & $0$ \\
            $1-q$ & $2q$    & $1-4q$ & $q$ \\ \hline
                  & $r$     & $1-2r$ & $r$
        \end{tabular}
        \caption{Butcher tableau for A-stable SDIRK scheme of orders 2 - 4 with
            $\gamma = (3 + \sqrt{3}) / 6$, $q =  \cos{(\pi / 18)} / \sqrt{3} + 1/2$
            and $r = 1 / (6 (2 q - 1)^2)$;
            See \cite{BonaventuraDellarocca2015}.}
    \end{center}
    \label{butcher-tableau-A-SDIRK234-tab}
\end{table}
\FloatBarrier

\section{Numerical results}

\subsection{Initial condition}
\label{numerical-results-ic-suppsec}
~
\FloatBarrier
\begin{figure}[h!]
    \centering
    \includegraphics[width=0.45\textwidth]{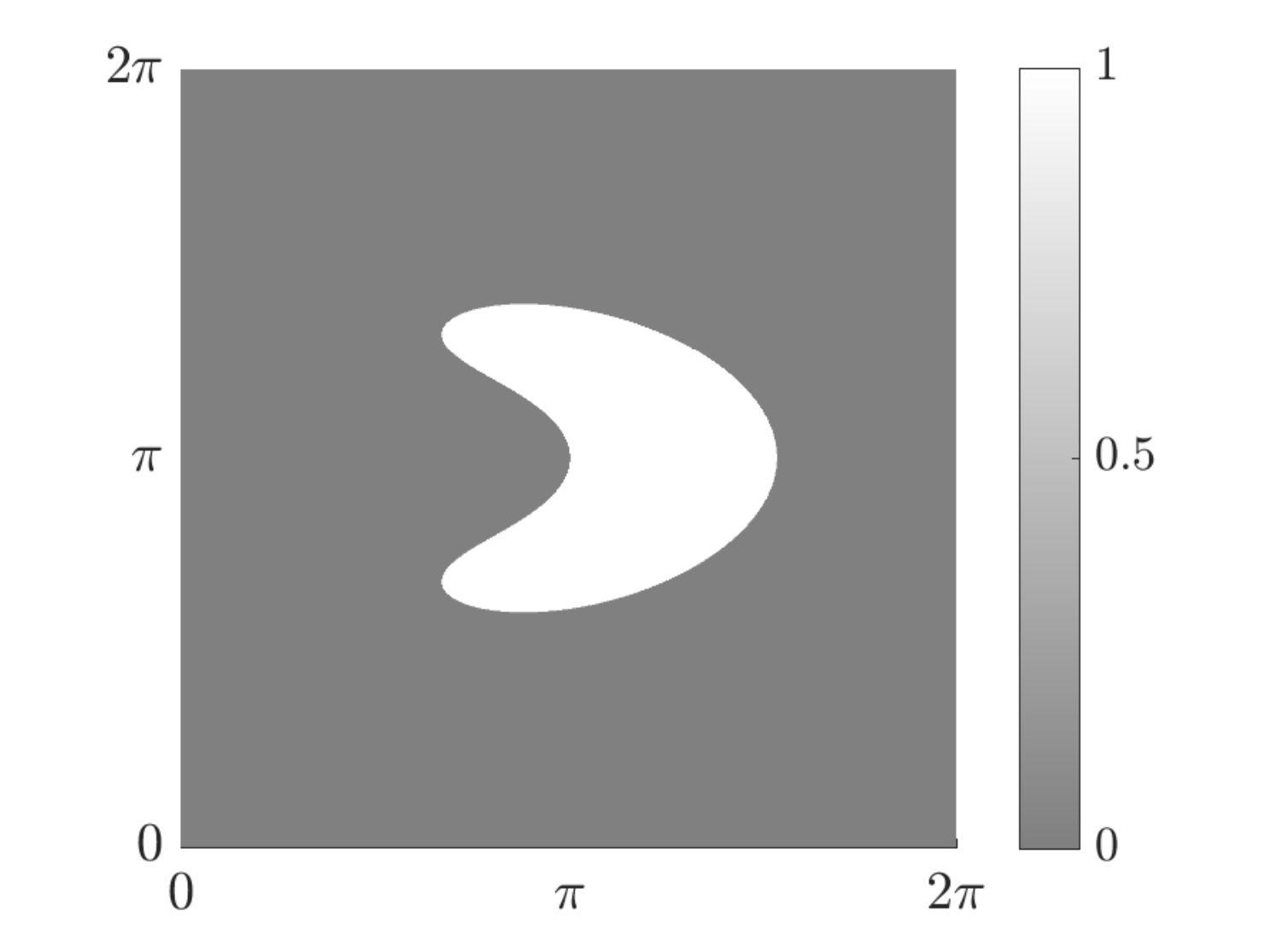}
    \includegraphics[width=0.45\textwidth]{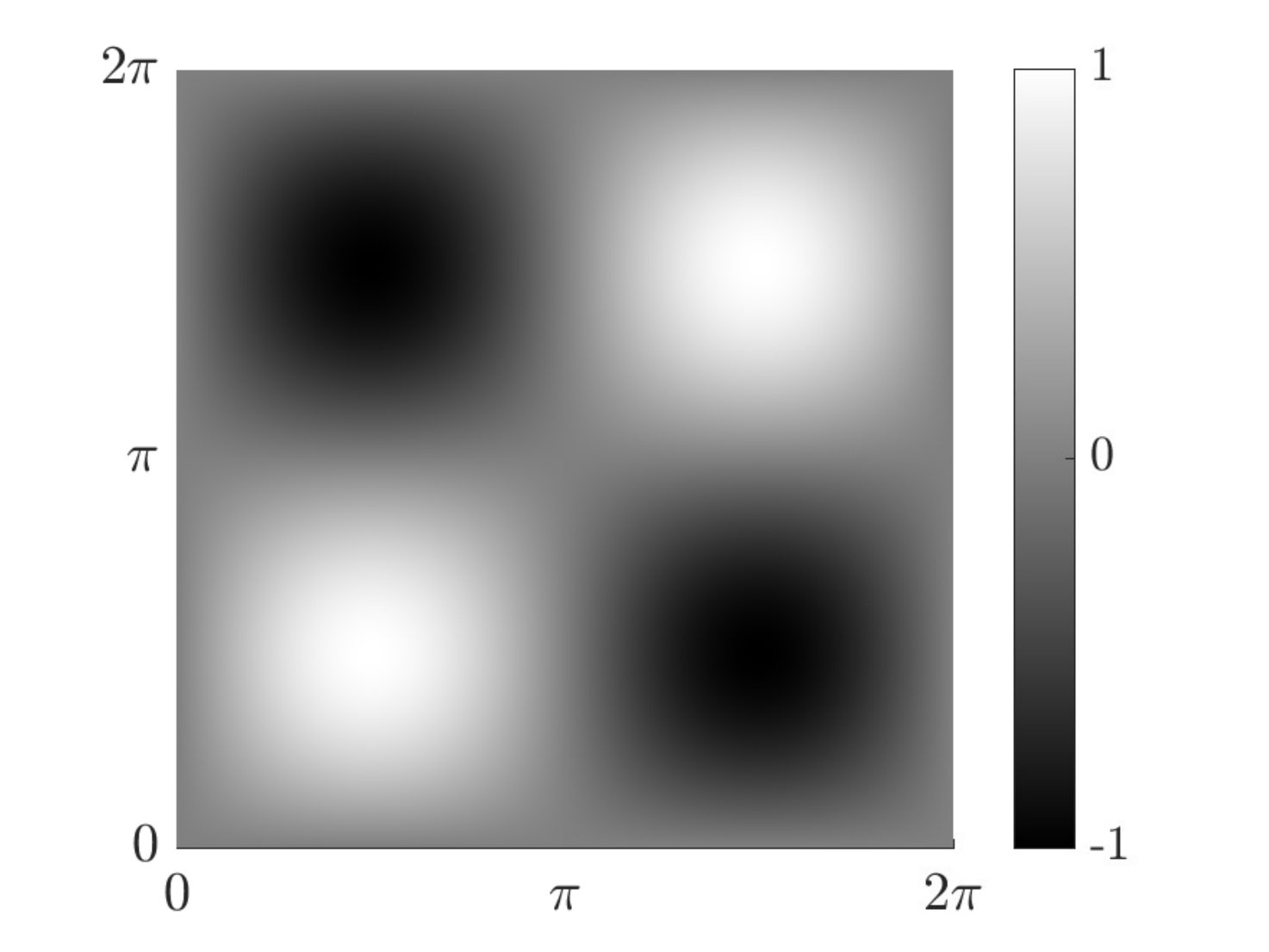}
    \caption{Initial condition for 2D wave equation (left) and 2D diffusion equation (right).}
    \label{diffusion-wave-ic-suppfig}
\end{figure}
\FloatBarrier

\subsection{Isotropic diffusion equation}
\label{isotropic-diffusion-equation-additional-numerical-results-suppsec}
\FloatBarrier
\begin{figure}[h!]
    \centering
    \hspace{-1.75cm}
    \setlength{\figurewidth}{0.3\linewidth}
    \setlength{\figureheight}{0.25\linewidth}
    \begin{subfigure}[b]{0.4\textwidth}
        \raisebox{0.5cm}{\definecolor{mycolor1}{rgb}{0.00000,0.44700,0.74100}%
        \definecolor{mycolor2}{rgb}{0.85000,0.32500,0.09800}%
        \definecolor{mycolor3}{rgb}{0.92900,0.69400,0.12500}%
        \definecolor{mycolor4}{rgb}{0.49400,0.18400,0.55600}%
%

    \end{subfigure}
    \caption{Isotropic diffusion: Comparison of V- and F-cycle MGRIT with F-relaxation.
        Convergence of A-stable schemes deteriorates much quicker with a growing number of time
        grid levels and V-cycle MGRIT than for L-stable schemes and V-cycle MGRIT.
        The convergence factor for L-stable schemes and F-cycle MGRIT is almost constant.}
    \label{diffusion-isotropic-V-F-cycle-nt1024-F-relaxation-nn11x11-suppfig}
\end{figure}
\FloatBarrier
\FloatBarrier
\begin{figure}[h!]
    \centering
    \hspace{-1.75cm}
    \setlength{\figurewidth}{0.3\linewidth}
    \setlength{\figureheight}{0.25\linewidth}
    \begin{subfigure}[b]{0.4\textwidth}
        \raisebox{0.5cm}{
            \definecolor{mycolor1}{rgb}{0.00000,0.44700,0.74100}%
            \definecolor{mycolor2}{rgb}{0.85000,0.32500,0.09800}%
            \definecolor{mycolor3}{rgb}{0.92900,0.69400,0.12500}%
            \definecolor{mycolor4}{rgb}{0.49400,0.18400,0.55600}%
%

    \end{subfigure}
    \caption{Isotropic diffusion: Comparison of V- and F-cycle MGRIT with FCF-relaxation.
        Convergence of A-stable and L-stable schemes deteriorates only slightly for an
        MGRIT V-cycle algorithm. On the other hand, the convergence factor
        for an F-cycle MGRIT algorithm is constant for all considered RK schemes and cases.}
    \label{diffusion-isotropic-V-F-cycle-nt1024-FCF-relaxation-nn11x11-suppfig}
\end{figure}
\FloatBarrier

\subsection{Anisotropic diffusion equation}
\label{anisotropic-diffusion-equation-additional-numerical-results-suppsec}
~
\FloatBarrier
\setlength{\figurewidth}{0.2\linewidth}
\setlength{\figureheight}{0.3\linewidth}
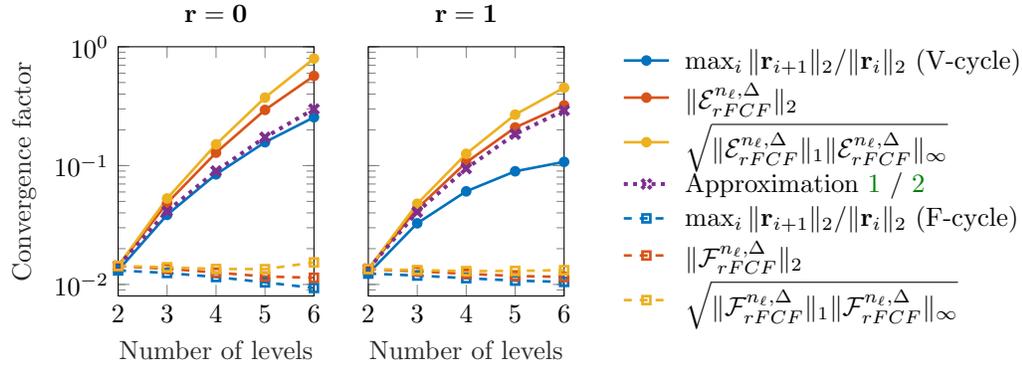
\begin{figure}[h!]
    \centering
   \hspace{-3cm}
    \begin{subfigure}[b]{0.3\textwidth}
        \definecolor{mycolor1}{rgb}{0.00000,0.44700,0.74100}%
        \definecolor{mycolor2}{rgb}{0.85000,0.32500,0.09800}%
        \definecolor{mycolor3}{rgb}{0.92900,0.69400,0.12500}%
        \definecolor{mycolor4}{rgb}{0.49400,0.18400,0.55600}%
        \begin{tikzpicture}

        \begin{axis}[%
        width=\figurewidth,
        height=0.849\figureheight,
        at={(0\figurewidth,0\figureheight)},
        scale only axis,
        unbounded coords=jump,
        xmin=2,
        xmax=6,
        xtick={2, 3, 4, 5, 6},
        xlabel style={font=\color{white!15!black}},
        xlabel={Number of levels},
        ymode=log,
        ymin=0.008,
        ymax=1,
        yminorticks=true,
        ylabel style={font=\color{white!15!black}},
        ylabel={Convergence factor},
        axis background/.style={fill=white},
        title style={font=\bfseries},
        title={$\mathbf{r = 0}$}
        ]
        \addplot [color=mycolor1, line width=1pt, mark size=1.5pt, mark=*, mark options={solid, mycolor1}, forget plot]
          table[row sep=crcr]{%
        2	0.0131195859007948\\
        3	0.0384940699580679\\
        4	0.0844207065892325\\
        5	0.157354090962062\\
        6	0.255180539447881\\
        };
        \addplot [color=mycolor2, line width=1pt, mark size=1.5pt, mark=*, mark options={solid, mycolor2}, forget plot]
          table[row sep=crcr]{%
        2	0.0142378138554095\\
        3	0.0483440920426073\\
        4	0.128465307097121\\
        5	0.294214747440343\\
        6	0.568495498824089\\
        };
        \addplot [color=mycolor3, line width=1pt, mark size=1.5pt, mark=*, mark options={solid, mycolor3}, forget plot]
          table[row sep=crcr]{%
        2	0.014309642481936\\
        3	0.0529094702345216\\
        4	0.150928864599452\\
        5	0.373541985076262\\
        6	0.79580456982689\\
        };
        \addplot [color=mycolor4, line width=1.5pt, mark size=2.5pt, dotted, mark=x, mark options={solid, mycolor4}, forget plot]
          table[row sep=crcr]{%
        2	0.0143096424819362\\
        3	0.0413139320495366\\
        4	0.0900594130484497\\
        5	0.173125913188797\\
        6	0.298906134007341\\
        };
        \addplot [color=mycolor1, dashed, line width=1pt, mark size=1.5pt, mark=square, mark options={solid, mycolor1}, forget plot]
          table[row sep=crcr]{%
        2	0.0131195859007948\\
        3	0.0124937389498109\\
        4	0.011555392018312\\
        5	0.0104082091822464\\
        6	0.00934424439275446\\
        };
        \addplot [color=mycolor2, dashed, line width=1pt, mark size=1.5pt, mark=square, mark options={solid, mycolor2}, forget plot]
          table[row sep=crcr]{%
        2	0.0142378138554095\\
        3	0.0135501078825623\\
        4	0.0125995367554516\\
        5	0.0116614098406755\\
        6	0.0113893292004962\\
        };
        \addplot [color=mycolor3, dashed, line width=1pt, mark size=1.5pt, mark=square, mark options={solid, mycolor3}, forget plot]
          table[row sep=crcr]{%
        2	0.014309642481936\\
        3	0.0139521829612406\\
        4	0.0135056111994526\\
        5	0.0134545053074463\\
        6	0.0153262052234104\\
        };
        \addplot [color=mycolor4, dashed, line width=1pt, mark size=1.5pt, mark=square, mark options={solid, mycolor4}, forget plot]
          table[row sep=crcr]{%
        2	nan\\
        3	nan\\
        4	nan\\
        5	nan\\
        6	nan\\
        };
        \end{axis}
        \end{tikzpicture}%

    \end{subfigure}\qquad
    \begin{subfigure}[b]{0.3\textwidth}
        \definecolor{mycolor1}{rgb}{0.00000,0.44700,0.74100}%
        \definecolor{mycolor2}{rgb}{0.85000,0.32500,0.09800}%
        \definecolor{mycolor3}{rgb}{0.92900,0.69400,0.12500}%
        \definecolor{mycolor4}{rgb}{0.49400,0.18400,0.55600}%
        \begin{tikzpicture}

        \begin{axis}[%
        width=\figurewidth,
        height=0.849\figureheight,
        at={(0\figurewidth,0\figureheight)},
        scale only axis,
        unbounded coords=jump,
        xmin=2,
        xmax=6,
        xtick={2, 3, 4, 5, 6},
        xlabel style={font=\color{white!15!black}},
        xlabel={Number of levels},
        ymode=log,
        ymin=0.008,
        ymax=1,
        yminorticks=true,
        yticklabels={,,},
        axis background/.style={fill=white},
        title style={font=\bfseries},
        title={$\mathbf{r = 1}$}
        ]
        \addplot [color=mycolor1, line width=1pt, mark size=1.5pt, mark=*, mark options={solid, mycolor1}, forget plot]
          table[row sep=crcr]{%
        2	0.0123457472619505\\
        3	0.0327098739476685\\
        4	0.0607412832106444\\
        5	0.0895966063932679\\
        6	0.107472590436227\\
        };
        \addplot [color=mycolor2, line width=1pt, mark size=1.5pt, mark=*, mark options={solid, mycolor2}, forget plot]
          table[row sep=crcr]{%
        2	0.0134102261971545\\
        3	0.0435819837148031\\
        4	0.106342790643953\\
        5	0.209329349119447\\
        6	0.321003132687412\\
        };
        \addplot [color=mycolor3, line width=1pt, mark size=1.5pt, mark=*, mark options={solid, mycolor3}, forget plot]
          table[row sep=crcr]{%
        2	0.0134781354311136\\
        3	0.0477738018900641\\
        4	0.125608214773449\\
        5	0.26857847246042\\
        6	0.452854067978917\\
        };
        \addplot [color=mycolor4, line width=1.5pt, mark size=2.5pt, dotted, mark=x, mark options={solid, mycolor4}, forget plot]
          table[row sep=crcr]{%
        2	0.0134781354311138\\
        3	0.0406712806437466\\
        4	0.0944569507240598\\
        5	0.184168907556853\\
        6	0.291190628110804\\
        };
        \addplot [color=mycolor1, dashed, line width=1pt, mark size=1.5pt, mark=square, mark options={solid, mycolor1}, forget plot]
          table[row sep=crcr]{%
        2	0.0123457472619505\\
        3	0.0118818705432437\\
        4	0.0112962549439031\\
        5	0.0107726271789498\\
        6	0.010508942223429\\
        };
        \addplot [color=mycolor2, dashed, line width=1pt, mark size=1.5pt, mark=square, mark options={solid, mycolor2}, forget plot]
          table[row sep=crcr]{%
        2	0.0134102261971543\\
        3	0.0128973586317559\\
        4	0.0122871514457647\\
        5	0.0117852946233255\\
        6	0.0115690400029249\\
        };
        \addplot [color=mycolor3, dashed, line width=1pt, mark size=1.5pt, mark=square, mark options={solid, mycolor3}, forget plot]
          table[row sep=crcr]{%
        2	0.0134781354311135\\
        3	0.0132128918350252\\
        4	0.0129436690545357\\
        5	0.0129886140847788\\
        6	0.0131949912778645\\
        };
        \addplot [color=mycolor4, dashed, line width=1pt, mark size=1.5pt, mark=square, mark options={solid, mycolor4}, forget plot]
          table[row sep=crcr]{%
        2	nan\\
        3	nan\\
        4	nan\\
        5	nan\\
        6	nan\\
        };
        \end{axis}
        \end{tikzpicture}%

    \end{subfigure}
   \hspace{-0.75cm}
    \begin{subfigure}[b]{0.2\textwidth}
        \raisebox{0.5cm}{
            \definecolor{mycolor1}{rgb}{0.00000,0.44700,0.74100}%
            \definecolor{mycolor2}{rgb}{0.85000,0.32500,0.09800}%
            \definecolor{mycolor3}{rgb}{0.92900,0.69400,0.12500}%
            \definecolor{mycolor4}{rgb}{0.49400,0.18400,0.55600}%
            \begin{tikzpicture}

            \begin{axis}[%
            width=\figurewidth,
            height=0.8\figureheight,
            at={(0\figurewidth,0\figureheight)},
            hide axis,
            unbounded coords=jump,
            xmin=2,
            xmax=6,
            ymin=0,
            ymax=1,
            axis background/.style={fill=white},
            title style={font=\bfseries},
            legend style={at={(-0.03,0.5)}, anchor=center, legend cell align=left, align=left, fill=none, draw=none}
            ]

            \addlegendimage{color=mycolor1, line width=1pt, mark size=1.5pt, mark=*}
            \addlegendimage{color=mycolor2, line width=1pt, mark size=1.5pt, mark=*}
            \addlegendimage{color=mycolor3, line width=1pt, mark size=1.5pt, mark=*}
            \addlegendimage{color=mycolor4, line width=1.5pt, mark size=2.5pt, dotted, mark=x}
            \addlegendimage{color=mycolor1, line width=1pt, dashed, mark size=1.5pt, mark=square, mark options={solid, mycolor1}}
            \addlegendimage{color=mycolor2, line width=1pt, dashed, mark size=1.5pt, mark=square, mark options={solid, mycolor2}}
            \addlegendimage{color=mycolor3, line width=1pt, dashed, mark size=1.5pt, mark=square, mark options={solid, mycolor3}}
            \addlegendimage{color=mycolor4, line width=1pt, dashed, mark size=1.5pt, mark=square, mark options={solid, mycolor4}}

            \addlegendentry{~$\max_i \|\mathbf{r}_{i+1}\|_2 / \|\mathbf{r}_i\|_2$ (V-cycle)}
            \addlegendentry{~$\| \mathcal{E}_{rFCF}^{n_\ell,\Delta} \|_2$}
            \addlegendentry{~$\sqrt{ \| \mathcal{E}_{rFCF}^{n_\ell,\Delta} \|_1 \| \mathcal{E}_{rFCF}^{n_\ell,\Delta} \|_\infty }$}
            \addlegendentry{~Approximation \ref{cf-nl-F-app} / \ref{cf-nl-FCF-app}}
            \addlegendentry{~$\max_i \|\mathbf{r}_{i+1}\|_2 / \|\mathbf{r}_i\|_2$ (F-cycle)}
            \addlegendentry{~$\| \mathcal{F}_{rFCF}^{n_\ell,\Delta} \|_2$}
            \addlegendentry{~$\sqrt{ \| \mathcal{F}_{rFCF}^{n_\ell,\Delta} \|_1 \| \mathcal{F}_{rFCF}^{n_\ell,\Delta} \|_\infty }$}
            \end{axis}
            \end{tikzpicture}%

        }
    \end{subfigure}
    \caption{Anisotropic diffusion: Comparison of V- and F-cycle MGRIT with F-relaxation ($r = 0$)
    and FCF-relaxation ($r = 1$).
    With a growing number of time grids, the convergence factor increases relatively quickly
    for V-cycle MGRIT.
    On the other hand, F-cycle MGRIT yields a nearly constant convergence factor, and thus,
    a more robust algorithm.}
    \label{diffusion-anisotropic-V-F-cycle-nt1024-F-FCF-relaxation-nn11x11-suppfig}
\end{figure}
\FloatBarrier

\subsection{Wave equation}
\label{wave-equation-additional-numerical-results-suppsec}
~
\FloatBarrier
\begin{figure}[h!]
    \centering
    \hspace{-1.75cm}
    \setlength{\figurewidth}{0.3\linewidth}
    \setlength{\figureheight}{0.25\linewidth}
    \begin{subfigure}[b]{0.4\textwidth}
        \hspace{0.75cm}
        \raisebox{1.25cm}{
            \definecolor{mycolor1}{rgb}{0.00000,0.44700,0.74100}%
            \definecolor{mycolor2}{rgb}{0.85000,0.32500,0.09800}%
            \definecolor{mycolor3}{rgb}{0.92900,0.69400,0.12500}%
            \definecolor{mycolor4}{rgb}{0.49400,0.18400,0.55600}%
            \definecolor{mycolor5}{rgb}{0.46600,0.67400,0.18800}%
            \begin{tikzpicture}

            \begin{axis}[%
            width=0.8\figurewidth,
            height=\figureheight,
            at={(0\figurewidth,0\figureheight)},
            hide axis,
            unbounded coords=jump,
            xmin=2,
            xmax=6,
            ymin=0.1,
            ymax=1,
            axis background/.style={fill=white},
            legend style={at={(0.7,0.5)}, anchor=center, legend cell align=left, align=left, fill=none, draw=none}
            ]

            \addlegendimage{color=mycolor1, line width=1pt, mark size=1.5pt, mark=*}
            \addlegendimage{color=mycolor2, line width=1pt, mark size=1.5pt, mark=*}
            \addlegendimage{color=mycolor3, line width=1pt, mark size=1.5pt, mark=*}
            \addlegendimage{color=mycolor4, line width=1.5pt, mark size=2.5pt, dotted, mark=x}

            \addlegendentry{\footnotesize~$\max_i \|\mathbf{r}_{i+1}\|_2 / \|\mathbf{r}_i\|_2$}
            \addlegendentry{\footnotesize~$\| \mathcal{E}_{F}^{n_\ell,\Delta} \|_2$}
            \addlegendentry{\footnotesize~$\sqrt{ \| \mathcal{E}_{F}^{n_\ell,\Delta} \|_1 \| \mathcal{E}_{F}^{n_\ell,\Delta} \|_\infty }$}
            \addlegendentry{\footnotesize~Approximation \ref{cf-nl-F-app}}

            \end{axis}
            \end{tikzpicture}%

        }
    \end{subfigure}\qquad
    \setlength{\figurewidth}{0.4\linewidth}
    \setlength{\figureheight}{0.25\linewidth}
    \begin{subfigure}[b]{0.4\textwidth}
        \definecolor{mycolor1}{rgb}{0.00000,0.44700,0.74100}%
        \definecolor{mycolor2}{rgb}{0.85000,0.32500,0.09800}%
        \definecolor{mycolor3}{rgb}{0.92900,0.69400,0.12500}%
        \definecolor{mycolor4}{rgb}{0.49400,0.18400,0.55600}%
        \begin{tikzpicture}

        \begin{axis}[%
        width=\figurewidth,
        height=0.849\figureheight,
        at={(0\figurewidth,0\figureheight)},
        scale only axis,
        xmin=2,
        xmax=6,
        xtick={2, 3, 4, 5, 6},
        xlabel style={font=\color{white!15!black}},
        xlabel={Number of levels},
        ymode=log,
        ymin=0.1,
        ymax=10,
        yminorticks=true,
        ylabel style={font=\color{white!15!black}},
        ylabel={Convergence factor},
        axis background/.style={fill=white},
        title style={font=\bfseries},
        title={L-stable SDIRK1}
        ]
        \addplot [color=mycolor1, line width=1pt, mark size=1.5pt, mark=*, mark options={solid, mycolor1}, forget plot]
          table[row sep=crcr]{%
        2	0.427041546965819\\
        3	0.697356193732166\\
        4	0.885290322467072\\
        5	1.08908034542197\\
        6	1.29142901736183\\
        };
        \addplot [color=mycolor2, line width=1pt, mark size=1.5pt, mark=*, mark options={solid, mycolor2}, forget plot]
          table[row sep=crcr]{%
        2	0.468801233109855\\
        3	0.889830599573516\\
        4	1.38316523757003\\
        5	2.03278922931191\\
        6	2.86114562876203\\
        };
        \addplot [color=mycolor3, line width=1pt, mark size=1.5pt, mark=*, mark options={solid, mycolor3}, forget plot]
          table[row sep=crcr]{%
        2	0.499699153253868\\
        3	0.978986883120193\\
        4	1.59184027534479\\
        5	2.52276574550005\\
        6	3.76051723319965\\
        };
        \addplot [color=mycolor4, line width=1.5pt, mark size=2.5pt, dotted, mark=x, mark options={solid, mycolor4}, forget plot]
          table[row sep=crcr]{%
        2	0.499701146340221\\
        3	0.754307195458225\\
        4	0.929947341797284\\
        5	1.17414336526808\\
        6	1.5111453625552\\
        };
        \end{axis}
        \end{tikzpicture}%
    \end{subfigure}\\[1ex]
    \setlength{\figurewidth}{0.4\linewidth}
    \setlength{\figureheight}{0.25\linewidth}
    \begin{subfigure}[b]{0.4\textwidth}
        \definecolor{mycolor1}{rgb}{0.00000,0.44700,0.74100}%
        \definecolor{mycolor2}{rgb}{0.85000,0.32500,0.09800}%
        \definecolor{mycolor3}{rgb}{0.92900,0.69400,0.12500}%
        \definecolor{mycolor4}{rgb}{0.49400,0.18400,0.55600}%
        \begin{tikzpicture}

        \begin{axis}[%
        width=\figurewidth,
        height=0.849\figureheight,
        at={(0\figurewidth,0\figureheight)},
        scale only axis,
        xmin=2,
        xmax=6,
        xtick={2, 3, 4, 5, 6},
        xlabel style={font=\color{white!15!black}},
        xlabel={Number of levels},
        ymode=log,
        ymin=0.01,
        ymax=10,
        yminorticks=true,
        ylabel style={font=\color{white!15!black}},
        ylabel={Convergence factor},
        axis background/.style={fill=white},
        title style={font=\bfseries},
        title={A-stable SDIRK2}
        ]
        \addplot [color=mycolor1, line width=1pt, mark size=1.5pt, mark=*, mark options={solid, mycolor1}, forget plot]
          table[row sep=crcr]{%
        2	0.0127833327193592\\
        3	0.0633750542431069\\
        4	0.264323009683932\\
        5	1.10102585961521\\
        6	3.09062025058423\\
        };
        \addplot [color=mycolor2, line width=1pt, mark size=1.5pt, mark=*, mark options={solid, mycolor2}, forget plot]
          table[row sep=crcr]{%
        2	0.026150104576929\\
        3	0.166703421292912\\
        4	0.942116135122626\\
        5	5.02957509119582\\
        6	23.1939291234184\\
        };
        \addplot [color=mycolor3, line width=1pt, mark size=1.5pt, mark=*, mark options={solid, mycolor3}, forget plot]
          table[row sep=crcr]{%
        2	0.0410363975297065\\
        3	0.274074608708593\\
        4	1.57453590776184\\
        5	8.46584955261564\\
        6	39.434832395859\\
        };
        \addplot [color=mycolor4, line width=1.5pt, mark size=2.5pt, dotted, mark=x, mark options={solid, mycolor4}, forget plot]
          table[row sep=crcr]{%
        2	0.0410764525880668\\
        3	0.204305763273798\\
        4	0.847155825344775\\
        5	3.28038580657159\\
        6	11.1792076528893\\
        };
        \end{axis}
        \end{tikzpicture}%
    \end{subfigure}\qquad\qquad\qquad
    \begin{subfigure}[b]{0.4\textwidth}
        \definecolor{mycolor1}{rgb}{0.00000,0.44700,0.74100}%
        \definecolor{mycolor2}{rgb}{0.85000,0.32500,0.09800}%
        \definecolor{mycolor3}{rgb}{0.92900,0.69400,0.12500}%
        \definecolor{mycolor4}{rgb}{0.49400,0.18400,0.55600}%
        \begin{tikzpicture}

        \begin{axis}[%
        width=\figurewidth,
        height=0.849\figureheight,
        at={(0\figurewidth,0\figureheight)},
        scale only axis,
        xmin=2,
        xmax=6,
        xtick={2, 3, 4, 5, 6},
        xlabel style={font=\color{white!15!black}},
        xlabel={Number of levels},
        ymode=log,
        ymin=0.01,
        ymax=10,
        yminorticks=true,
        yticklabels={,,},
        axis background/.style={fill=white},
        title style={font=\bfseries},
        title={L-stable SDIRK2}
        ]
        \addplot [color=mycolor1, line width=1pt, mark size=1.5pt, mark=*, mark options={solid, mycolor1}, forget plot]
          table[row sep=crcr]{%
        2	0.0247939919863588\\
        3	0.122424609666652\\
        4	0.468977720660123\\
        5	1.72165372668413\\
        6	3.21939983261641\\
        };
        \addplot [color=mycolor2, line width=1pt, mark size=1.5pt, mark=*, mark options={solid, mycolor2}, forget plot]
          table[row sep=crcr]{%
        2	0.0506518870945083\\
        3	0.320174560708435\\
        4	1.72249877520988\\
        5	7.1367790651213\\
        6	15.784410020929\\
        };
        \addplot [color=mycolor3, line width=1pt, mark size=1.5pt, mark=*, mark options={solid, mycolor3}, forget plot]
          table[row sep=crcr]{%
        2	0.079473689509342\\
        3	0.525823098162095\\
        4	2.85448945936512\\
        5	11.3537624919868\\
        6	21.6914646429558\\
        };
        \addplot [color=mycolor4, line width=1.5pt, mark size=2.5pt, dotted, mark=x, mark options={solid, mycolor4}, forget plot]
          table[row sep=crcr]{%
        2	0.0795511981633453\\
        3	0.392020146479185\\
        4	1.5363420718136\\
        5	4.3883904014382\\
        6	6.03423051914282\\
        };
        \end{axis}
        \end{tikzpicture}%
    \end{subfigure}\\[1ex]
    \setlength{\figurewidth}{0.4\linewidth}
    \setlength{\figureheight}{0.25\linewidth}
    \begin{subfigure}[b]{0.4\textwidth}
        \definecolor{mycolor1}{rgb}{0.00000,0.44700,0.74100}%
        \definecolor{mycolor2}{rgb}{0.85000,0.32500,0.09800}%
        \definecolor{mycolor3}{rgb}{0.92900,0.69400,0.12500}%
        \definecolor{mycolor4}{rgb}{0.49400,0.18400,0.55600}%
        \begin{tikzpicture}

        \begin{axis}[%
        width=\figurewidth,
        height=0.849\figureheight,
        at={(0\figurewidth,0\figureheight)},
        scale only axis,
        xmin=2,
        xmax=6,
        xtick={2, 3, 4, 5, 6},
        xlabel style={font=\color{white!15!black}},
        xlabel={Number of levels},
        ymode=log,
        ymin=0.001,
        ymax=10,
        yminorticks=true,
        ylabel style={font=\color{white!15!black}},
        ylabel={Convergence factor},
        axis background/.style={fill=white},
        title style={font=\bfseries},
        title={A-stable SDIRK3}
        ]
        \addplot [color=mycolor1, line width=1pt, mark size=1.5pt, mark=*, mark options={solid, mycolor1}, forget plot]
          table[row sep=crcr]{%
        2	0.0108559985932353\\
        3	0.0967428174584231\\
        4	0.523105838294067\\
        5	1.39549312268753\\
        6	3.24629846629937\\
        };
        \addplot [color=mycolor2, line width=1pt, mark size=1.5pt, mark=*, mark options={solid, mycolor2}, forget plot]
          table[row sep=crcr]{%
        2	0.021907406569714\\
        3	0.226521508479596\\
        4	1.37081640073273\\
        5	4.42595741436157\\
        6	12.3980503126809\\
        };
        \addplot [color=mycolor3, line width=1pt, mark size=1.5pt, mark=*, mark options={solid, mycolor3}, forget plot]
          table[row sep=crcr]{%
        2	0.0342512610093387\\
        3	0.355989831447296\\
        4	1.93691321862813\\
        5	5.328235814984\\
        6	15.167467938439\\
        };
        \addplot [color=mycolor4, line width=1.5pt, mark size=2.5pt, dotted, mark=x, mark options={solid, mycolor4}, forget plot]
          table[row sep=crcr]{%
        2	0.0342840400100247\\
        3	0.260112541411072\\
        4	1.0110995623275\\
        5	2.00293926883872\\
        6	4.2892737336012\\
        };
        \end{axis}
        \end{tikzpicture}%
    \end{subfigure}\qquad\qquad\qquad
    \begin{subfigure}[b]{0.4\textwidth}
        \definecolor{mycolor1}{rgb}{0.00000,0.44700,0.74100}%
        \definecolor{mycolor2}{rgb}{0.85000,0.32500,0.09800}%
        \definecolor{mycolor3}{rgb}{0.92900,0.69400,0.12500}%
        \definecolor{mycolor4}{rgb}{0.49400,0.18400,0.55600}%
        \begin{tikzpicture}

        \begin{axis}[%
        width=\figurewidth,
        height=0.849\figureheight,
        at={(0\figurewidth,0\figureheight)},
        scale only axis,
        xmin=2,
        xmax=6,
        xtick={2, 3, 4, 5, 6},
        xlabel style={font=\color{white!15!black}},
        xlabel={Number of levels},
        ymode=log,
        ymin=0.001,
        ymax=10,
        yminorticks=true,
        yticklabels={,,},
        axis background/.style={fill=white},
        title style={font=\bfseries},
        title={L-stable SDIRK3}
        ]
        \addplot [color=mycolor1, line width=1pt, mark size=1.5pt, mark=*, mark options={solid, mycolor1}, forget plot]
          table[row sep=crcr]{%
        2	0.00267532614256879\\
        3	0.0276828337210245\\
        4	0.230992577840743\\
        5	0.748016324348825\\
        6	1.52680015494383\\
        };
        \addplot [color=mycolor2, line width=1pt, mark size=1.5pt, mark=*, mark options={solid, mycolor2}, forget plot]
          table[row sep=crcr]{%
        2	0.00646766728042624\\
        3	0.0736806876184319\\
        4	0.626182742109267\\
        5	2.61667414948779\\
        6	6.00752107585315\\
        };
        \addplot [color=mycolor3, line width=1pt, mark size=1.5pt, mark=*, mark options={solid, mycolor3}, forget plot]
          table[row sep=crcr]{%
        2	0.0101383706114157\\
        3	0.117835592447524\\
        4	0.959319393506406\\
        5	3.33462864135975\\
        6	6.85261153240434\\
        };
        \addplot [color=mycolor4, line width=1.5pt, mark size=2.5pt, dotted, mark=x, mark options={solid, mycolor4}, forget plot]
          table[row sep=crcr]{%
        2	0.0101482091636581\\
        3	0.0860623337701928\\
        4	0.500400895580734\\
        5	1.2378587340703\\
        6	1.87805859163489\\
        };
        \end{axis}
        \end{tikzpicture}%
    \end{subfigure}\\[1ex]
    \setlength{\figurewidth}{0.4\linewidth}
    \setlength{\figureheight}{0.25\linewidth}
    \begin{subfigure}[b]{0.4\textwidth}
        \definecolor{mycolor1}{rgb}{0.00000,0.44700,0.74100}%
        \definecolor{mycolor2}{rgb}{0.85000,0.32500,0.09800}%
        \definecolor{mycolor3}{rgb}{0.92900,0.69400,0.12500}%
        \definecolor{mycolor4}{rgb}{0.49400,0.18400,0.55600}%
        \begin{tikzpicture}

        \begin{axis}[%
        width=\figurewidth,
        height=0.849\figureheight,
        at={(0\figurewidth,0\figureheight)},
        scale only axis,
        xmin=2,
        xmax=6,
        xtick={2, 3, 4, 5, 6},
        xlabel style={font=\color{white!15!black}},
        xlabel={Number of levels},
        ymode=log,
        ymin=1e-07,
        ymax=10,
        yminorticks=true,
        ylabel style={font=\color{white!15!black}},
        ylabel={Convergence factor},
        axis background/.style={fill=white},
        title style={font=\bfseries},
        title={A-stable SDIRK4}
        ]
        \addplot [color=mycolor1, line width=1pt, mark size=1.5pt, mark=*, mark options={solid, mycolor1}, forget plot]
          table[row sep=crcr]{%
        2	0.00416116422867001\\
        3	0.0532092243096273\\
        4	0.454233567486202\\
        5	1.05312219362023\\
        6	2.35611179484748\\
        };
        \addplot [color=mycolor2, line width=1pt, mark size=1.5pt, mark=*, mark options={solid, mycolor2}, forget plot]
          table[row sep=crcr]{%
        2	0.00728731413082763\\
        3	0.139960613173144\\
        4	1.17084555216597\\
        5	3.31452111263316\\
        6	8.71163837167512\\
        };
        \addplot [color=mycolor3, line width=1pt, mark size=1.5pt, mark=*, mark options={solid, mycolor3}, forget plot]
          table[row sep=crcr]{%
        2	0.0114298022668852\\
        3	0.220527433646191\\
        4	1.61395749312456\\
        5	3.75065099775646\\
        6	10.2911501033451\\
        };
        \addplot [color=mycolor4, line width=1.5pt, mark size=2.5pt, dotted, mark=x, mark options={solid, mycolor4}, forget plot]
          table[row sep=crcr]{%
        2	0.0114409281435788\\
        3	0.159281724090351\\
        4	0.830915931104776\\
        5	1.3979905460933\\
        6	2.93933857474756\\
        };
        \end{axis}
        \end{tikzpicture}%
    \end{subfigure}\qquad\qquad\qquad
    \begin{subfigure}[b]{0.4\textwidth}
        \definecolor{mycolor1}{rgb}{0.00000,0.44700,0.74100}%
        \definecolor{mycolor2}{rgb}{0.85000,0.32500,0.09800}%
        \definecolor{mycolor3}{rgb}{0.92900,0.69400,0.12500}%
        \definecolor{mycolor4}{rgb}{0.49400,0.18400,0.55600}%
        \begin{tikzpicture}

        \begin{axis}[%
        width=\figurewidth,
        height=0.849\figureheight,
        at={(0\figurewidth,0\figureheight)},
        scale only axis,
        xmin=2,
        xmax=6,
        xtick={2, 3, 4, 5, 6},
        xlabel style={font=\color{white!15!black}},
        xlabel={Number of levels},
        ymode=log,
        ymin=1e-07,
        ymax=10,
        yminorticks=true,
        yticklabels={,,},
        axis background/.style={fill=white},
        title style={font=\bfseries},
        title={L-stable SDIRK4}
        ]
        \addplot [color=mycolor1, line width=1pt, mark size=1.5pt, mark=*, mark options={solid, mycolor1}, forget plot]
          table[row sep=crcr]{%
        2	5.98107521206688e-07\\
        3	9.98580114902018e-06\\
        4	0.00533762570255451\\
        5	0.0949956677157784\\
        6	0.790319425321937\\
        };
        \addplot [color=mycolor2, line width=1pt, mark size=1.5pt, mark=*, mark options={solid, mycolor2}, forget plot]
          table[row sep=crcr]{%
        2	3.9550907870132e-05\\
        3	0.000919829838534665\\
        4	0.0205280468717682\\
        5	0.432818497245124\\
        6	6.86639617199081\\
        };
        \addplot [color=mycolor3, line width=1pt, mark size=1.5pt, mark=*, mark options={solid, mycolor3}, forget plot]
          table[row sep=crcr]{%
        2	6.20657832496286e-05\\
        3	0.00146450009421774\\
        4	0.0327299778185826\\
        5	0.690307526976447\\
        6	10.9464437968763\\
        };
        \addplot [color=mycolor4, line width=1.5pt, mark size=2.5pt, dotted, mark=x, mark options={solid, mycolor4}, forget plot]
          table[row sep=crcr]{%
        2	6.21263647978498e-05\\
        3	0.00105588530497039\\
        4	0.016808693765926\\
        5	0.253854679058847\\
        6	2.92444417966326\\
        };
        \end{axis}
        \end{tikzpicture}%
    \end{subfigure}
    \caption{Wave equation: Observed convergence and predicted upper bounds on convergence
        of MGRIT with V-cycles and F-relaxation shows very similar trends as for MGRIT
        with V-cycles and FCF-relaxation,
        see Figure \ref{wave-V-cycle-nt1024-FCF-relaxation-nn11x11-suppfig}.
        This shows that switching from F-relaxation to FCF-relaxation alone is not sufficient
        to yield a robust MGRIT algorithm for the wave equation (and likely, other hyperbolic
        PDEs).}
    \label{wave-V-cycle-nt1024-F-relaxation-nn11x11-suppfig}
\end{figure}
\FloatBarrier
\FloatBarrier
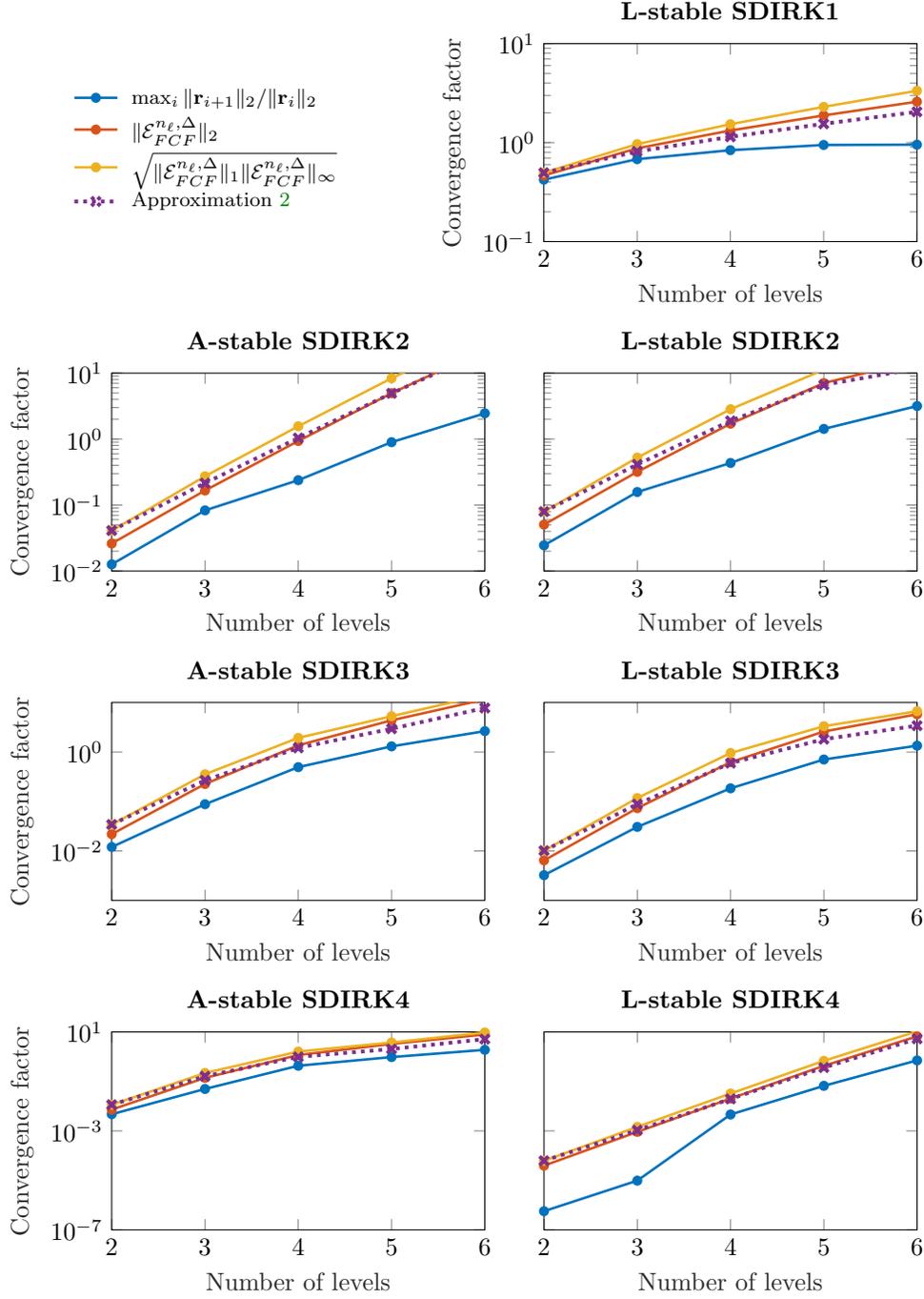
\begin{figure}[h!]
    \centering
    \hspace{-1.75cm}
    \setlength{\figurewidth}{0.3\linewidth}
    \setlength{\figureheight}{0.25\linewidth}
    \begin{subfigure}[b]{0.4\textwidth}
        \hspace{0.75cm}
        \raisebox{1.25cm}{
            \definecolor{mycolor1}{rgb}{0.00000,0.44700,0.74100}%
            \definecolor{mycolor2}{rgb}{0.85000,0.32500,0.09800}%
            \definecolor{mycolor3}{rgb}{0.92900,0.69400,0.12500}%
            \definecolor{mycolor4}{rgb}{0.49400,0.18400,0.55600}%
            \definecolor{mycolor5}{rgb}{0.46600,0.67400,0.18800}%
            \begin{tikzpicture}

            \begin{axis}[%
            width=0.8\figurewidth,
            height=\figureheight,
            at={(0\figurewidth,0\figureheight)},
            hide axis,
            unbounded coords=jump,
            xmin=2,
            xmax=6,
            ymin=0.1,
            ymax=1,
            axis background/.style={fill=white},
            legend style={at={(0.7,0.5)}, anchor=center, legend cell align=left, align=left, fill=none, draw=none}
            ]

            \addlegendimage{color=mycolor1, line width=1pt, mark size=1.5pt, mark=*}
            \addlegendimage{color=mycolor2, line width=1pt, mark size=1.5pt, mark=*}
            \addlegendimage{color=mycolor3, line width=1pt, mark size=1.5pt, mark=*}
            \addlegendimage{color=mycolor4, line width=1.5pt, mark size=2.5pt, dotted, mark=x}

            \addlegendentry{\footnotesize~$\max_i \|\mathbf{r}_{i+1}\|_2 / \|\mathbf{r}_i\|_2$}
            \addlegendentry{\footnotesize~$\| \mathcal{E}_{FCF}^{n_\ell,\Delta} \|_2$}
            \addlegendentry{\footnotesize~$\sqrt{ \| \mathcal{E}_{FCF}^{n_\ell,\Delta} \|_1 \| \mathcal{E}_{FCF}^{n_\ell,\Delta} \|_\infty }$}
            \addlegendentry{\footnotesize~Approximation \ref{cf-nl-FCF-app}}

            \end{axis}
            \end{tikzpicture}%

        }
    \end{subfigure}\qquad
    \setlength{\figurewidth}{0.4\linewidth}
    \setlength{\figureheight}{0.25\linewidth}
    \begin{subfigure}[b]{0.4\textwidth}
        \definecolor{mycolor1}{rgb}{0.00000,0.44700,0.74100}%
        \definecolor{mycolor2}{rgb}{0.85000,0.32500,0.09800}%
        \definecolor{mycolor3}{rgb}{0.92900,0.69400,0.12500}%
        \definecolor{mycolor4}{rgb}{0.49400,0.18400,0.55600}%
        \begin{tikzpicture}

        \begin{axis}[%
        width=\figurewidth,
        height=0.849\figureheight,
        at={(0\figurewidth,0\figureheight)},
        scale only axis,
        xmin=2,
        xmax=6,
        xtick={2, 3, 4, 5, 6},
        xlabel style={font=\color{white!15!black}},
        xlabel={Number of levels},
        ymode=log,
        ymin=0.1,
        ymax=10,
        yminorticks=true,
        ylabel style={font=\color{white!15!black}},
        ylabel={Convergence factor},
        axis background/.style={fill=white},
        title style={font=\bfseries},
        title={L-stable SDIRK1}
        ]
        \addplot [color=mycolor1, line width=1pt, mark size=1.5pt, mark=*, mark options={solid, mycolor1}, forget plot]
          table[row sep=crcr]{%
        2	0.422707263086206\\
        3	0.68091720393957\\
        4	0.839232893133704\\
        5	0.94729820659592\\
        6	0.953971343349235\\
        };
        \addplot [color=mycolor2, line width=1pt, mark size=1.5pt, mark=*, mark options={solid, mycolor2}, forget plot]
          table[row sep=crcr]{%
        2	0.465239435266433\\
        3	0.87618639321187\\
        4	1.32606370436559\\
        5	1.88461267759253\\
        6	2.59363057276779\\
        };
        \addplot [color=mycolor3, line width=1pt, mark size=1.5pt, mark=*, mark options={solid, mycolor3}, forget plot]
          table[row sep=crcr]{%
        2	0.496077405867888\\
        3	0.966454780947916\\
        4	1.53431528464857\\
        5	2.29931912375723\\
        6	3.34323662110172\\
        };
        \addplot [color=mycolor4, line width=1.5pt, mark size=2.5pt, dotted, mark=x, mark options={solid, mycolor4}, forget plot]
          table[row sep=crcr]{%
        2	0.49608703352612\\
        3	0.814029564157985\\
        4	1.14464911743607\\
        5	1.54968560489753\\
        6	2.04914673076731\\
        };
        \end{axis}
        \end{tikzpicture}%
    \end{subfigure}\\[1ex]
    \setlength{\figurewidth}{0.4\linewidth}
    \setlength{\figureheight}{0.25\linewidth}
    \begin{subfigure}[b]{0.4\textwidth}
        \definecolor{mycolor1}{rgb}{0.00000,0.44700,0.74100}%
        \definecolor{mycolor2}{rgb}{0.85000,0.32500,0.09800}%
        \definecolor{mycolor3}{rgb}{0.92900,0.69400,0.12500}%
        \definecolor{mycolor4}{rgb}{0.49400,0.18400,0.55600}%
        \begin{tikzpicture}

        \begin{axis}[%
        width=\figurewidth,
        height=0.849\figureheight,
        at={(0\figurewidth,0\figureheight)},
        scale only axis,
        xmin=2,
        xmax=6,
        xtick={2, 3, 4, 5, 6},
        xlabel style={font=\color{white!15!black}},
        xlabel={Number of levels},
        ymode=log,
        ymin=0.01,
        ymax=10,
        yminorticks=true,
        ylabel style={font=\color{white!15!black}},
        ylabel={Convergence factor},
        axis background/.style={fill=white},
        title style={font=\bfseries},
        title={A-stable SDIRK2}
        ]
        \addplot [color=mycolor1, line width=1pt, mark size=1.5pt, mark=*, mark options={solid, mycolor1}, forget plot]
          table[row sep=crcr]{%
        2	0.0126237948958803\\
        3	0.0829917403735648\\
        4	0.238225196343646\\
        5	0.897536255881032\\
        6	2.45683465444267\\
        };
        \addplot [color=mycolor2, line width=1pt, mark size=1.5pt, mark=*, mark options={solid, mycolor2}, forget plot]
          table[row sep=crcr]{%
        2	0.0260990800226153\\
        3	0.166092070101794\\
        4	0.935281356159273\\
        5	4.95670620876972\\
        6	22.5251279779128\\
        };
        \addplot [color=mycolor3, line width=1pt, mark size=1.5pt, mark=*, mark options={solid, mycolor3}, forget plot]
          table[row sep=crcr]{%
        2	0.0409562483157834\\
        3	0.273086683554486\\
        4	1.56342657983375\\
        5	8.34717901296332\\
        6	38.3423423076198\\
        };
        \addplot [color=mycolor4, line width=1.5pt, mark size=2.5pt, dotted, mark=x, mark options={solid, mycolor4}, forget plot]
          table[row sep=crcr]{%
        2	0.0410764525880669\\
        3	0.215015300138658\\
        4	1.03993298906818\\
        5	4.93670208755283\\
        6	21.3433329825369\\
        };
        \end{axis}
        \end{tikzpicture}%
    \end{subfigure}\qquad\qquad\qquad
    \begin{subfigure}[b]{0.4\textwidth}
        \definecolor{mycolor1}{rgb}{0.00000,0.44700,0.74100}%
        \definecolor{mycolor2}{rgb}{0.85000,0.32500,0.09800}%
        \definecolor{mycolor3}{rgb}{0.92900,0.69400,0.12500}%
        \definecolor{mycolor4}{rgb}{0.49400,0.18400,0.55600}%
        \begin{tikzpicture}

        \begin{axis}[%
        width=\figurewidth,
        height=0.849\figureheight,
        at={(0\figurewidth,0\figureheight)},
        scale only axis,
        xmin=2,
        xmax=6,
        xtick={2, 3, 4, 5, 6},
        xlabel style={font=\color{white!15!black}},
        xlabel={Number of levels},
        ymode=log,
        ymin=0.01,
        ymax=10,
        yminorticks=true,
        yticklabels={,,},
        axis background/.style={fill=white},
        title style={font=\bfseries},
        title={L-stable SDIRK2}
        ]
        \addplot [color=mycolor1, line width=1pt, mark size=1.5pt, mark=*, mark options={solid, mycolor1}, forget plot]
          table[row sep=crcr]{%
        2	0.0244843517742587\\
        3	0.157756502262385\\
        4	0.434494005491729\\
        5	1.42660225264305\\
        6	3.16908975550067\\
        };
        \addplot [color=mycolor2, line width=1pt, mark size=1.5pt, mark=*, mark options={solid, mycolor2}, forget plot]
          table[row sep=crcr]{%
        2	0.0505531000941099\\
        3	0.319006721670251\\
        4	1.71050760210563\\
        5	7.05669105132868\\
        6	15.5856601797381\\
        };
        \addplot [color=mycolor3, line width=1pt, mark size=1.5pt, mark=*, mark options={solid, mycolor3}, forget plot]
          table[row sep=crcr]{%
        2	0.0793185638122103\\
        3	0.523940814964597\\
        4	2.83541336422652\\
        5	11.243491032787\\
        6	21.5975279921208\\
        };
        \addplot [color=mycolor4, line width=1.5pt, mark size=2.5pt, dotted, mark=x, mark options={solid, mycolor4}, forget plot]
          table[row sep=crcr]{%
        2	0.0795511657749488\\
        3	0.412655086367942\\
        4	1.888039464055\\
        5	6.63891655230384\\
        6	11.6946879181383\\
        };
        \end{axis}
        \end{tikzpicture}%
    \end{subfigure}\\[1ex]
    \setlength{\figurewidth}{0.4\linewidth}
    \setlength{\figureheight}{0.25\linewidth}
    \begin{subfigure}[b]{0.4\textwidth}
        \definecolor{mycolor1}{rgb}{0.00000,0.44700,0.74100}%
        \definecolor{mycolor2}{rgb}{0.85000,0.32500,0.09800}%
        \definecolor{mycolor3}{rgb}{0.92900,0.69400,0.12500}%
        \definecolor{mycolor4}{rgb}{0.49400,0.18400,0.55600}%
        \begin{tikzpicture}

        \begin{axis}[%
        width=\figurewidth,
        height=0.849\figureheight,
        at={(0\figurewidth,0\figureheight)},
        scale only axis,
        xmin=2,
        xmax=6,
        xtick={2, 3, 4, 5, 6},
        xlabel style={font=\color{white!15!black}},
        xlabel={Number of levels},
        ymode=log,
        ymin=0.001,
        ymax=10,
        yminorticks=true,
        ylabel style={font=\color{white!15!black}},
        ylabel={Convergence factor},
        axis background/.style={fill=white},
        title style={font=\bfseries},
        title={A-stable SDIRK3}
        ]
        \addplot [color=mycolor1, line width=1pt, mark size=1.5pt, mark=*, mark options={solid, mycolor1}, forget plot]
          table[row sep=crcr]{%
        2	0.0120191052548178\\
        3	0.0881109705762433\\
        4	0.494484190643019\\
        5	1.30024226662209\\
        6	2.64021506714484\\
        };
        \addplot [color=mycolor2, line width=1pt, mark size=1.5pt, mark=*, mark options={solid, mycolor2}, forget plot]
          table[row sep=crcr]{%
        2	0.0218651189923281\\
        3	0.225745865869666\\
        4	1.36422939994418\\
        5	4.37124060956429\\
        6	11.5418733517547\\
        };
        \addplot [color=mycolor3, line width=1pt, mark size=1.5pt, mark=*, mark options={solid, mycolor3}, forget plot]
          table[row sep=crcr]{%
        2	0.0341853297088356\\
        3	0.354825636460333\\
        4	1.93007325863185\\
        5	5.2866312044556\\
        6	14.3261301528964\\
        };
        \addplot [color=mycolor4, line width=1.5pt, mark size=2.5pt, dotted, mark=x, mark options={solid, mycolor4}, forget plot]
          table[row sep=crcr]{%
        2	0.0342837021430505\\
        3	0.267642581959773\\
        4	1.2103377789531\\
        5	2.9520555831536\\
        6	7.74257041631726\\
        };
        \end{axis}
        \end{tikzpicture}%
    \end{subfigure}\qquad\qquad\qquad
    \begin{subfigure}[b]{0.4\textwidth}
        \definecolor{mycolor1}{rgb}{0.00000,0.44700,0.74100}%
        \definecolor{mycolor2}{rgb}{0.85000,0.32500,0.09800}%
        \definecolor{mycolor3}{rgb}{0.92900,0.69400,0.12500}%
        \definecolor{mycolor4}{rgb}{0.49400,0.18400,0.55600}%
        \begin{tikzpicture}

        \begin{axis}[%
        width=\figurewidth,
        height=0.849\figureheight,
        at={(0\figurewidth,0\figureheight)},
        scale only axis,
        xmin=2,
        xmax=6,
        xtick={2, 3, 4, 5, 6},
        xlabel style={font=\color{white!15!black}},
        xlabel={Number of levels},
        ymode=log,
        ymin=0.001,
        ymax=10,
        yminorticks=true,
        yticklabels={,,},
        axis background/.style={fill=white},
        title style={font=\bfseries},
        title={L-stable SDIRK3}
        ]
        \addplot [color=mycolor1, line width=1pt, mark size=1.5pt, mark=*, mark options={solid, mycolor1}, forget plot]
          table[row sep=crcr]{%
        2	0.00323837021802482\\
        3	0.0306767589026215\\
        4	0.184671537558734\\
        5	0.705728012554802\\
        6	1.34261892340324\\
        };
        \addplot [color=mycolor2, line width=1pt, mark size=1.5pt, mark=*, mark options={solid, mycolor2}, forget plot]
          table[row sep=crcr]{%
        2	0.00645508811734074\\
        3	0.0734107737467412\\
        4	0.622366339997128\\
        5	2.59591843529659\\
        6	5.78381548842482\\
        };
        \addplot [color=mycolor3, line width=1pt, mark size=1.5pt, mark=*, mark options={solid, mycolor3}, forget plot]
          table[row sep=crcr]{%
        2	0.0101186545764888\\
        3	0.117412976139701\\
        4	0.954050838565406\\
        5	3.32075989425305\\
        6	6.67739771296237\\
        };
        \addplot [color=mycolor4, line width=1.5pt, mark size=2.5pt, dotted, mark=x, mark options={solid, mycolor4}, forget plot]
          table[row sep=crcr]{%
        2	0.0101481801423545\\
        3	0.088458362080275\\
        4	0.596409913745118\\
        5	1.81613403476017\\
        6	3.40515258321172\\
        };
        \end{axis}
        \end{tikzpicture}%
    \end{subfigure}\\[1ex]
    \setlength{\figurewidth}{0.4\linewidth}
    \setlength{\figureheight}{0.25\linewidth}
    \begin{subfigure}[b]{0.4\textwidth}
        \definecolor{mycolor1}{rgb}{0.00000,0.44700,0.74100}%
        \definecolor{mycolor2}{rgb}{0.85000,0.32500,0.09800}%
        \definecolor{mycolor3}{rgb}{0.92900,0.69400,0.12500}%
        \definecolor{mycolor4}{rgb}{0.49400,0.18400,0.55600}%
        \begin{tikzpicture}

        \begin{axis}[%
        width=\figurewidth,
        height=0.849\figureheight,
        at={(0\figurewidth,0\figureheight)},
        scale only axis,
        xmin=2,
        xmax=6,
        xtick={2, 3, 4, 5, 6},
        xlabel style={font=\color{white!15!black}},
        xlabel={Number of levels},
        ymode=log,
        ymin=1e-07,
        ymax=10,
        yminorticks=true,
        ylabel style={font=\color{white!15!black}},
        ylabel={Convergence factor},
        axis background/.style={fill=white},
        title style={font=\bfseries},
        title={A-stable SDIRK4}
        ]
        \addplot [color=mycolor1, line width=1pt, mark size=1.5pt, mark=*, mark options={solid, mycolor1}, forget plot]
          table[row sep=crcr]{%
        2	0.00473339681471841\\
        3	0.0496473619635444\\
        4	0.430149516221153\\
        5	0.963588686863819\\
        6	1.89448510357372\\
        };
        \addplot [color=mycolor2, line width=1pt, mark size=1.5pt, mark=*, mark options={solid, mycolor2}, forget plot]
          table[row sep=crcr]{%
        2	0.00727312389153252\\
        3	0.1394504667967\\
        4	1.16588157962976\\
        5	3.27561930133136\\
        6	7.90505053733317\\
        };
        \addplot [color=mycolor3, line width=1pt, mark size=1.5pt, mark=*, mark options={solid, mycolor3}, forget plot]
          table[row sep=crcr]{%
        2	0.0114075353317223\\
        3	0.219741898411467\\
        4	1.60928334200213\\
        5	3.72715994499783\\
        6	9.53557040503556\\
        };
        \addplot [color=mycolor4, line width=1.5pt, mark size=2.5pt, dotted, mark=x, mark options={solid, mycolor4}, forget plot]
          table[row sep=crcr]{%
        2	0.0114409240180685\\
        3	0.161758661007896\\
        4	0.981989167910772\\
        5	2.04295246135027\\
        6	5.12541331961503\\
        };
        \end{axis}
        \end{tikzpicture}%
    \end{subfigure}\qquad\qquad\qquad
    \begin{subfigure}[b]{0.4\textwidth}
        \definecolor{mycolor1}{rgb}{0.00000,0.44700,0.74100}%
        \definecolor{mycolor2}{rgb}{0.85000,0.32500,0.09800}%
        \definecolor{mycolor3}{rgb}{0.92900,0.69400,0.12500}%
        \definecolor{mycolor4}{rgb}{0.49400,0.18400,0.55600}%
        \begin{tikzpicture}

        \begin{axis}[%
        width=\figurewidth,
        height=0.849\figureheight,
        at={(0\figurewidth,0\figureheight)},
        scale only axis,
        xmin=2,
        xmax=6,
        xtick={2, 3, 4, 5, 6},
        xlabel style={font=\color{white!15!black}},
        xlabel={Number of levels},
        ymode=log,
        ymin=1e-07,
        ymax=10,
        yminorticks=true,
        yticklabels={,,},
        axis background/.style={fill=white},
        title style={font=\bfseries},
        title={L-stable SDIRK4}
        ]
        \addplot [color=mycolor1, line width=1pt, mark size=1.5pt, mark=*, mark options={solid, mycolor1}, forget plot]
          table[row sep=crcr]{%
        2	5.68527005677997e-07\\
        3	9.77911728432672e-06\\
        4	0.00462349723864827\\
        5	0.0663626090077182\\
        6	0.708013452181577\\
        };
        \addplot [color=mycolor2, line width=1pt, mark size=1.5pt, mark=*, mark options={solid, mycolor2}, forget plot]
          table[row sep=crcr]{%
        2	3.94737354375397e-05\\
        3	0.000916292932387609\\
        4	0.0203700306558133\\
        5	0.426149053815399\\
        6	6.65825730591849\\
        };
        \addplot [color=mycolor3, line width=1pt, mark size=1.5pt, mark=*, mark options={solid, mycolor3}, forget plot]
          table[row sep=crcr]{%
        2	6.19445610380176e-05\\
        3	0.0014589019032384\\
        4	0.032479789254962\\
        5	0.679755121901639\\
        6	10.6197919278232\\
        };
        \addplot [color=mycolor4, line width=1.5pt, mark size=2.5pt, dotted, mark=x, mark options={solid, mycolor4}, forget plot]
          table[row sep=crcr]{%
        2	6.21263647964648e-05\\
        3	0.00106978228605693\\
        4	0.0196460705547786\\
        5	0.362316645782796\\
        6	5.26753742222048\\
        };
        \end{axis}
        \end{tikzpicture}%
    \end{subfigure}
    \caption{Wave equation: The convergence factor of MGRIT with V-cycles and FCF-relaxation
        increases substantially with a growing number of time grid levels
        and eventually exceeds $1$. This means that in the worst case, MGRIT V-cycles yields
        a divergent algorithm, which is in line with observations for hyperbolic PDEs
        in the literature \cite{FarhatChandesris2003,
        FarhatCortial2006,HessenthalerNordslettenRoehrleSchroderFalgout2018}.}
    \label{wave-V-cycle-nt1024-FCF-relaxation-nn11x11-suppfig}
\end{figure}
\FloatBarrier
\FloatBarrier
\begin{figure}[h!]
    \centering
    \hspace{-1.75cm}
    \setlength{\figurewidth}{0.3\linewidth}
    \setlength{\figureheight}{0.25\linewidth}
    \begin{subfigure}[b]{0.4\textwidth}
        \hspace{0.75cm}
        \raisebox{1.25cm}{
            \definecolor{mycolor1}{rgb}{0.00000,0.44700,0.74100}%
            \definecolor{mycolor2}{rgb}{0.85000,0.32500,0.09800}%
            \definecolor{mycolor3}{rgb}{0.92900,0.69400,0.12500}%
            \definecolor{mycolor4}{rgb}{0.49400,0.18400,0.55600}%
            \definecolor{mycolor5}{rgb}{0.46600,0.67400,0.18800}%
            \begin{tikzpicture}

            \begin{axis}[%
            width=\figurewidth,
            height=\figureheight,
            at={(0\figurewidth,0\figureheight)},
            hide axis,
            unbounded coords=jump,
            xmin=2,
            xmax=6,
            ymin=0.1,
            ymax=1,
            axis background/.style={fill=white},
            legend style={at={(0.9,0.5)}, anchor=center, legend cell align=left, align=left, fill=none, draw=none}
            ]

            \addlegendimage{color=mycolor1, line width=1.2pt}, mark size=2.0pt, mark=*
            \addlegendimage{color=mycolor2, line width=1pt, mark size=1.5pt, mark=*}
            \addlegendimage{color=mycolor3, line width=1pt, mark size=1.5pt, mark=*}

            \addlegendentry{\footnotesize~$\max_i \|\mathbf{r}_{i+1}\|_2 / \|\mathbf{r}_i\|_2$}
            \addlegendentry{\footnotesize~$\| \mathcal{F}_{F}^{n_\ell,\Delta} \|_2$}
            \addlegendentry{\footnotesize~$\sqrt{ \| \mathcal{F}_{F}^{n_\ell,\Delta} \|_1 \| \mathcal{F}_{F}^{n_\ell,\Delta} \|_\infty }$}

            \end{axis}
            \end{tikzpicture}%

        }
    \end{subfigure}\qquad
    \setlength{\figurewidth}{0.4\linewidth}
    \setlength{\figureheight}{0.25\linewidth}
    \begin{subfigure}[b]{0.4\textwidth}
        \definecolor{mycolor1}{rgb}{0.00000,0.44700,0.74100}%
        \definecolor{mycolor2}{rgb}{0.85000,0.32500,0.09800}%
        \definecolor{mycolor3}{rgb}{0.92900,0.69400,0.12500}%
        \definecolor{mycolor4}{rgb}{0.49400,0.18400,0.55600}%
        \begin{tikzpicture}

        \begin{axis}[%
        width=\figurewidth,
        height=0.849\figureheight,
        at={(0\figurewidth,0\figureheight)},
        scale only axis,
        unbounded coords=jump,
        xmin=2,
        xmax=6,
        xtick={2, 3, 4, 5, 6},
        xlabel style={font=\color{white!15!black}},
        xlabel={Number of levels},
        ymode=log,
        ymin=0.1,
        ymax=10,
        yminorticks=true,
        ylabel style={font=\color{white!15!black}},
        ylabel={Convergence factor},
        axis background/.style={fill=white},
        title style={font=\bfseries},
        title={L-stable SDIRK1}
        ]
        \addplot [color=mycolor1, line width=1pt, mark size=1.5pt, mark=*, mark options={solid, mycolor1}, forget plot]
          table[row sep=crcr]{%
        2	0.427041546965819\\
        3	0.552316480717758\\
        4	0.687744520379039\\
        5	0.874386584120249\\
        6	1.02149172990219\\
        };
        \addplot [color=mycolor2, line width=1pt, mark size=1.5pt, mark=*, mark options={solid, mycolor2}, forget plot]
          table[row sep=crcr]{%
        2	0.468801233109855\\
        3	0.651520746808723\\
        4	0.931635205408301\\
        5	1.4015239272078\\
        6	2.16759636474516\\
        };
        \addplot [color=mycolor3, line width=1pt, mark size=1.5pt, mark=*, mark options={solid, mycolor3}, forget plot]
          table[row sep=crcr]{%
        2	0.499699153253868\\
        3	0.745866415018615\\
        4	1.14386861917292\\
        5	1.90626681135816\\
        6	3.23286611837149\\
        };
        \addplot [color=mycolor4, line width=1pt, mark size=1.5pt, mark=*, mark options={solid, mycolor4}, forget plot]
          table[row sep=crcr]{%
        2	nan\\
        3	nan\\
        4	nan\\
        5	nan\\
        6	nan\\
        };
        \end{axis}
        \end{tikzpicture}%
    \end{subfigure}\\[1ex]
    \setlength{\figurewidth}{0.4\linewidth}
    \setlength{\figureheight}{0.25\linewidth}
    \begin{subfigure}[b]{0.4\textwidth}
        \definecolor{mycolor1}{rgb}{0.00000,0.44700,0.74100}%
        \definecolor{mycolor2}{rgb}{0.85000,0.32500,0.09800}%
        \definecolor{mycolor3}{rgb}{0.92900,0.69400,0.12500}%
        \definecolor{mycolor4}{rgb}{0.49400,0.18400,0.55600}%
        \begin{tikzpicture}

        \begin{axis}[%
        width=\figurewidth,
        height=0.849\figureheight,
        at={(0\figurewidth,0\figureheight)},
        scale only axis,
        unbounded coords=jump,
        xmin=2,
        xmax=6,
        xtick={2, 3, 4, 5, 6},
        xlabel style={font=\color{white!15!black}},
        xlabel={Number of levels},
        ymode=log,
        ymin=0.01,
        ymax=10,
        yminorticks=true,
        ylabel style={font=\color{white!15!black}},
        ylabel={Convergence factor},
        axis background/.style={fill=white},
        title style={font=\bfseries},
        title={A-stable SDIRK2}
        ]
        \addplot [color=mycolor1, line width=1pt, mark size=1.5pt, mark=*, mark options={solid, mycolor1}, forget plot]
          table[row sep=crcr]{%
        2	0.0127833327193592\\
        3	0.0124222236699561\\
        4	0.0136967889095167\\
        5	0.0560292016989289\\
        6	1.27907789994245\\
        };
        \addplot [color=mycolor2, line width=1pt, mark size=1.5pt, mark=*, mark options={solid, mycolor2}, forget plot]
          table[row sep=crcr]{%
        2	0.026150104576929\\
        3	0.028899086810459\\
        4	0.0856350329567948\\
        5	3.82860468741843\\
        6	1016.0639297697\\
        };
        \addplot [color=mycolor3, line width=1pt, mark size=1.5pt, mark=*, mark options={solid, mycolor3}, forget plot]
          table[row sep=crcr]{%
        2	0.0410363975297065\\
        3	0.0497918592135254\\
        4	0.185343827756128\\
        5	7.97348834867009\\
        6	2079.55893291217\\
        };
        \addplot [color=mycolor4, line width=1pt, mark size=1.5pt, mark=*, mark options={solid, mycolor4}, forget plot]
          table[row sep=crcr]{%
        2	nan\\
        3	nan\\
        4	nan\\
        5	nan\\
        6	nan\\
        };
        \end{axis}
        \end{tikzpicture}%
    \end{subfigure}\qquad\qquad\qquad
    \begin{subfigure}[b]{0.4\textwidth}
        \definecolor{mycolor1}{rgb}{0.00000,0.44700,0.74100}%
        \definecolor{mycolor2}{rgb}{0.85000,0.32500,0.09800}%
        \definecolor{mycolor3}{rgb}{0.92900,0.69400,0.12500}%
        \definecolor{mycolor4}{rgb}{0.49400,0.18400,0.55600}%
        \begin{tikzpicture}

        \begin{axis}[%
        width=\figurewidth,
        height=0.849\figureheight,
        at={(0\figurewidth,0\figureheight)},
        scale only axis,
        unbounded coords=jump,
        xmin=2,
        xmax=6,
        xtick={2, 3, 4, 5, 6},
        xlabel style={font=\color{white!15!black}},
        xlabel={Number of levels},
        ymode=log,
        ymin=0.01,
        ymax=10,
        yminorticks=true,
        yticklabels={,,},
        axis background/.style={fill=white},
        title style={font=\bfseries},
        title={L-stable SDIRK2}
        ]
        \addplot [color=mycolor1, line width=1pt, mark size=1.5pt, mark=*, mark options={solid, mycolor1}, forget plot]
          table[row sep=crcr]{%
        2	0.0247939919863588\\
        3	0.0259242661268218\\
        4	0.0308708880704343\\
        5	0.477981772762115\\
        6	9.60477035533039\\
        };
        \addplot [color=mycolor2, line width=1pt, mark size=1.5pt, mark=*, mark options={solid, mycolor2}, forget plot]
          table[row sep=crcr]{%
        2	0.0506518870945083\\
        3	0.0689150136877586\\
        4	0.480280269268822\\
        5	25.5090788177879\\
        6	1377.22122170059\\
        };
        \addplot [color=mycolor3, line width=1pt, mark size=1.5pt, mark=*, mark options={solid, mycolor3}, forget plot]
          table[row sep=crcr]{%
        2	0.079473689509342\\
        3	0.127921875779653\\
        4	0.999992695319344\\
        5	51.5191919939072\\
        6	2712.76856900295\\
        };
        \addplot [color=mycolor4, line width=1pt, mark size=1.5pt, mark=*, mark options={solid, mycolor4}, forget plot]
          table[row sep=crcr]{%
        2	nan\\
        3	nan\\
        4	nan\\
        5	nan\\
        6	nan\\
        };
        \end{axis}
        \end{tikzpicture}%
    \end{subfigure}\\[1ex]
    \setlength{\figurewidth}{0.4\linewidth}
    \setlength{\figureheight}{0.25\linewidth}
    \begin{subfigure}[b]{0.4\textwidth}
        \definecolor{mycolor1}{rgb}{0.00000,0.44700,0.74100}%
        \definecolor{mycolor2}{rgb}{0.85000,0.32500,0.09800}%
        \definecolor{mycolor3}{rgb}{0.92900,0.69400,0.12500}%
        \definecolor{mycolor4}{rgb}{0.49400,0.18400,0.55600}%
        \begin{tikzpicture}

        \begin{axis}[%
        width=\figurewidth,
        height=0.851\figureheight,
        at={(0\figurewidth,0\figureheight)},
        scale only axis,
        unbounded coords=jump,
        xmin=2,
        xmax=6,
        xtick={2, 3, 4, 5, 6},
        xlabel style={font=\color{white!15!black}},
        xlabel={Number of levels},
        ymode=log,
        ymin=0.001,
        ymax=10,
        yminorticks=true,
        ylabel style={font=\color{white!15!black}},
        ylabel={Convergence factor},
        axis background/.style={fill=white},
        title style={font=\bfseries},
        title={A-stable SDIRK3}
        ]
        \addplot [color=mycolor1, line width=1pt, mark size=1.5pt, mark=*, mark options={solid, mycolor1}, forget plot]
          table[row sep=crcr]{%
        2	0.0108559985932353\\
        3	0.0161489075848049\\
        4	0.0779086101948694\\
        5	1.41237398166846\\
        6	39.2279000832822\\
        };
        \addplot [color=mycolor2, line width=1pt, mark size=1.5pt, mark=*, mark options={solid, mycolor2}, forget plot]
          table[row sep=crcr]{%
        2	0.021907406569714\\
        3	0.0398611902177169\\
        4	0.425876738629977\\
        5	13.6424722003456\\
        6	1082.90830690613\\
        };
        \addplot [color=mycolor3, line width=1pt, mark size=1.5pt, mark=*, mark options={solid, mycolor3}, forget plot]
          table[row sep=crcr]{%
        2	0.0342512610093387\\
        3	0.0736708121238084\\
        4	0.802689899758003\\
        5	23.1889832481197\\
        6	1923.05085949413\\
        };
        \addplot [color=mycolor4, line width=1pt, mark size=1.5pt, mark=*, mark options={solid, mycolor4}, forget plot]
          table[row sep=crcr]{%
        2	nan\\
        3	nan\\
        4	nan\\
        5	nan\\
        6	nan\\
        };
        \end{axis}
        \end{tikzpicture}%
    \end{subfigure}\qquad\qquad\qquad
    \begin{subfigure}[b]{0.4\textwidth}
        \definecolor{mycolor1}{rgb}{0.00000,0.44700,0.74100}%
        \definecolor{mycolor2}{rgb}{0.85000,0.32500,0.09800}%
        \definecolor{mycolor3}{rgb}{0.92900,0.69400,0.12500}%
        \definecolor{mycolor4}{rgb}{0.49400,0.18400,0.55600}%
        \begin{tikzpicture}

        \begin{axis}[%
        width=\figurewidth,
        height=0.849\figureheight,
        at={(0\figurewidth,0\figureheight)},
        scale only axis,
        unbounded coords=jump,
        xmin=2,
        xmax=6,
        xtick={2, 3, 4, 5, 6},
        xlabel style={font=\color{white!15!black}},
        xlabel={Number of levels},
        ymode=log,
        ymin=0.001,
        ymax=10,
        yminorticks=true,
        yticklabels={,,},
        axis background/.style={fill=white},
        title style={font=\bfseries},
        title={L-stable SDIRK3}
        ]
        \addplot [color=mycolor1, line width=1pt, mark size=1.5pt, mark=*, mark options={solid, mycolor1}, forget plot]
          table[row sep=crcr]{%
        2	0.00267532614256879\\
        3	0.00356285695887558\\
        4	0.00579998036153973\\
        5	0.14763966782276\\
        6	2.83231213033025\\
        };
        \addplot [color=mycolor2, line width=1pt, mark size=1.5pt, mark=*, mark options={solid, mycolor2}, forget plot]
          table[row sep=crcr]{%
        2	0.00646766728042624\\
        3	0.00814891828744455\\
        4	0.0405740869214743\\
        5	1.35719963896617\\
        6	33.4843328824972\\
        };
        \addplot [color=mycolor3, line width=1pt, mark size=1.5pt, mark=*, mark options={solid, mycolor3}, forget plot]
          table[row sep=crcr]{%
        2	0.0101383706114157\\
        3	0.0143802652129805\\
        4	0.0831328717706557\\
        5	2.50965441247269\\
        6	52.6395190016887\\
        };
        \addplot [color=mycolor4, line width=1pt, mark size=1.5pt, mark=*, mark options={solid, mycolor4}, forget plot]
          table[row sep=crcr]{%
        2	nan\\
        3	nan\\
        4	nan\\
        5	nan\\
        6	nan\\
        };
        \end{axis}
        \end{tikzpicture}%
    \end{subfigure}\\[1ex]
    \setlength{\figurewidth}{0.4\linewidth}
    \setlength{\figureheight}{0.25\linewidth}
    \begin{subfigure}[b]{0.4\textwidth}
        \definecolor{mycolor1}{rgb}{0.00000,0.44700,0.74100}%
        \definecolor{mycolor2}{rgb}{0.85000,0.32500,0.09800}%
        \definecolor{mycolor3}{rgb}{0.92900,0.69400,0.12500}%
        \definecolor{mycolor4}{rgb}{0.49400,0.18400,0.55600}%
        \begin{tikzpicture}

        \begin{axis}[%
        width=\figurewidth,
        height=0.849\figureheight,
        at={(0\figurewidth,0\figureheight)},
        scale only axis,
        unbounded coords=jump,
        xmin=2,
        xmax=6,
        xtick={2, 3, 4, 5, 6},
        xlabel style={font=\color{white!15!black}},
        xlabel={Number of levels},
        ymode=log,
        ymin=1e-07,
        ymax=10,
        yminorticks=true,
        ylabel style={font=\color{white!15!black}},
        ylabel={Convergence factor},
        axis background/.style={fill=white},
        title style={font=\bfseries},
        title={A-stable SDIRK4}
        ]
        \addplot [color=mycolor1, line width=1pt, mark size=1.5pt, mark=*, mark options={solid, mycolor1}, forget plot]
          table[row sep=crcr]{%
        2	0.00416116422867001\\
        3	0.00578722061337781\\
        4	0.0474549722307326\\
        5	0.85284134695853\\
        6	22.9665231410466\\
        };
        \addplot [color=mycolor2, line width=1pt, mark size=1.5pt, mark=*, mark options={solid, mycolor2}, forget plot]
          table[row sep=crcr]{%
        2	0.00728731413082763\\
        3	0.0149873634708708\\
        4	0.275022929920467\\
        5	4.96303794251344\\
        6	238.437190880878\\
        };
        \addplot [color=mycolor3, line width=1pt, mark size=1.5pt, mark=*, mark options={solid, mycolor3}, forget plot]
          table[row sep=crcr]{%
        2	0.0114298022668852\\
        3	0.0279622571620679\\
        4	0.503268528651206\\
        5	7.50948670358275\\
        6	374.175881672276\\
        };
        \addplot [color=mycolor4, line width=1pt, mark size=1.5pt, mark=*, mark options={solid, mycolor4}, forget plot]
          table[row sep=crcr]{%
        2	nan\\
        3	nan\\
        4	nan\\
        5	nan\\
        6	nan\\
        };
        \end{axis}
        \end{tikzpicture}%
    \end{subfigure}\qquad\qquad\qquad
    \begin{subfigure}[b]{0.4\textwidth}
        \definecolor{mycolor1}{rgb}{0.00000,0.44700,0.74100}%
        \definecolor{mycolor2}{rgb}{0.85000,0.32500,0.09800}%
        \definecolor{mycolor3}{rgb}{0.92900,0.69400,0.12500}%
        \definecolor{mycolor4}{rgb}{0.49400,0.18400,0.55600}%
        \begin{tikzpicture}

        \begin{axis}[%
        width=\figurewidth,
        height=0.849\figureheight,
        at={(0\figurewidth,0\figureheight)},
        scale only axis,
        unbounded coords=jump,
        xmin=2,
        xmax=6,
        xtick={2, 3, 4, 5, 6},
        xlabel style={font=\color{white!15!black}},
        xlabel={Number of levels},
        ymode=log,
        ymin=1e-07,
        ymax=10,
        yminorticks=true,
        yticklabels={,,},
        axis background/.style={fill=white},
        title style={font=\bfseries},
        title={L-stable SDIRK4}
        ]
        \addplot [color=mycolor1, line width=1pt, mark size=1.5pt, mark=*, mark options={solid, mycolor1}, forget plot]
          table[row sep=crcr]{%
        2	5.98107521206688e-07\\
        3	5.98169902075979e-07\\
        4	6.04573732683324e-07\\
        5	4.14473975789086e-06\\
        6	0.0028441846995894\\
        };
        \addplot [color=mycolor2, line width=1pt, mark size=1.5pt, mark=*, mark options={solid, mycolor2}, forget plot]
          table[row sep=crcr]{%
        2	3.9550907870132e-05\\
        3	3.95570861444827e-05\\
        4	3.99507545154069e-05\\
        5	0.000241189308926225\\
        6	2.10107450072211\\
        };
        \addplot [color=mycolor3, line width=1pt, mark size=1.5pt, mark=*, mark options={solid, mycolor3}, forget plot]
          table[row sep=crcr]{%
        2	6.20657832496286e-05\\
        3	6.20857062223685e-05\\
        4	6.40267515104835e-05\\
        5	0.0005662953483806\\
        6	4.15616840610289\\
        };
        \addplot [color=mycolor4, line width=1pt, mark size=1.5pt, mark=*, mark options={solid, mycolor4}, forget plot]
          table[row sep=crcr]{%
        2	nan\\
        3	nan\\
        4	nan\\
        5	nan\\
        6	nan\\
        };
        \end{axis}
        \end{tikzpicture}%
    \end{subfigure}
    \caption{Wave equation: Observed convergence and predicted upper bounds on convergence
        of MGRIT with F-cycles and F-relaxation shows very similar trends as for MGRIT
        with F-cycles and FCF-relaxation,
        see Figure \ref{wave-F-cycle-nt1024-FCF-relaxation-nn11x11-suppfig}.
        This shows that switching from F-relaxation to FCF-relaxation alone is not sufficient
        to yield a robust MGRIT algorithm for the wave equation (and likely, other hyperbolic
        PDEs).}
    \label{wave-F-cycle-nt1024-F-relaxation-nn11x11-suppfig}
\end{figure}
\FloatBarrier
\FloatBarrier
\begin{figure}[h!]
    \centering
    \hspace{-1.75cm}
    \setlength{\figurewidth}{0.3\linewidth}
    \setlength{\figureheight}{0.25\linewidth}
    \begin{subfigure}[b]{0.4\textwidth}
        \hspace{0.75cm}
        \raisebox{1.25cm}{
            \definecolor{mycolor1}{rgb}{0.00000,0.44700,0.74100}%
            \definecolor{mycolor2}{rgb}{0.85000,0.32500,0.09800}%
            \definecolor{mycolor3}{rgb}{0.92900,0.69400,0.12500}%
            \definecolor{mycolor4}{rgb}{0.49400,0.18400,0.55600}%
            \definecolor{mycolor5}{rgb}{0.46600,0.67400,0.18800}%
            \begin{tikzpicture}

            \begin{axis}[%
            width=\figurewidth,
            height=\figureheight,
            at={(0\figurewidth,0\figureheight)},
            hide axis,
            unbounded coords=jump,
            xmin=2,
            xmax=6,
            ymin=0.1,
            ymax=1,
            axis background/.style={fill=white},
            legend style={at={(0.9,0.5)}, anchor=center, legend cell align=left, align=left, fill=none, draw=none}
            ]

            \addlegendimage{color=mycolor1, line width=1pt, mark size=1.5pt, mark=*}
            \addlegendimage{color=mycolor2, line width=1pt, mark size=1.5pt, mark=*}
            \addlegendimage{color=mycolor3, line width=1pt, mark size=1.5pt, mark=*}

            \addlegendentry{\footnotesize~$\max_i \|\mathbf{r}_{i+1}\|_2 / \|\mathbf{r}_i\|_2$}
            \addlegendentry{\footnotesize~$\| \mathcal{F}_{FCF}^{n_\ell,\Delta} \|_2$}
            \addlegendentry{\footnotesize~$\sqrt{ \| \mathcal{F}_{FCF}^{n_\ell,\Delta} \|_1 \| \mathcal{F}_{FCF}^{n_\ell,\Delta} \|_\infty }$}

            \end{axis}
            \end{tikzpicture}%

        }
    \end{subfigure}\qquad
    \setlength{\figurewidth}{0.4\linewidth}
    \setlength{\figureheight}{0.25\linewidth}
    \begin{subfigure}[b]{0.4\textwidth}
        \definecolor{mycolor1}{rgb}{0.00000,0.44700,0.74100}%
        \definecolor{mycolor2}{rgb}{0.85000,0.32500,0.09800}%
        \definecolor{mycolor3}{rgb}{0.92900,0.69400,0.12500}%
        \definecolor{mycolor4}{rgb}{0.49400,0.18400,0.55600}%
        \begin{tikzpicture}

        \begin{axis}[%
        width=\figurewidth,
        height=0.849\figureheight,
        at={(0\figurewidth,0\figureheight)},
        scale only axis,
        unbounded coords=jump,
        xmin=2,
        xmax=6,
        xtick={2, 3, 4, 5, 6},
        xlabel style={font=\color{white!15!black}},
        xlabel={Number of levels},
        ymode=log,
        ymin=0.01,
        ymax=10,
        yminorticks=true,
        ylabel style={font=\color{white!15!black}},
        ylabel={Convergence factor},
        axis background/.style={fill=white},
        title style={font=\bfseries},
        title={L-stable SDIRK1}
        ]
        \addplot [color=mycolor1, line width=1pt, mark size=1.5pt, mark=*, mark options={solid, mycolor1}, forget plot]
          table[row sep=crcr]{%
        2	0.422707263086206\\
        3	0.539747136999146\\
        4	0.634461665803552\\
        5	0.682397487793319\\
        6	0.659073801612238\\
        };
        \addplot [color=mycolor2, line width=1pt, mark size=1.5pt, mark=*, mark options={solid, mycolor2}, forget plot]
          table[row sep=crcr]{%
        2	0.465239435266433\\
        3	0.634185343198594\\
        4	0.837573420995158\\
        5	1.0631100028354\\
        6	1.24930273317484\\
        };
        \addplot [color=mycolor3, line width=1pt, mark size=1.5pt, mark=*, mark options={solid, mycolor3}, forget plot]
          table[row sep=crcr]{%
        2	0.496077405867889\\
        3	0.731282007225176\\
        4	1.06384740887511\\
        5	1.53715486550355\\
        6	2.0953971574138\\
        };
        \addplot [color=mycolor4, line width=1pt, mark size=1.5pt, mark=*, mark options={solid, mycolor4}, forget plot]
          table[row sep=crcr]{%
        2	nan\\
        3	nan\\
        4	nan\\
        5	nan\\
        6	nan\\
        };
        \end{axis}
        \end{tikzpicture}%
    \end{subfigure}\\[1ex]
    \setlength{\figurewidth}{0.4\linewidth}
    \setlength{\figureheight}{0.25\linewidth}
    \begin{subfigure}[b]{0.4\textwidth}
        \definecolor{mycolor1}{rgb}{0.00000,0.44700,0.74100}%
        \definecolor{mycolor2}{rgb}{0.85000,0.32500,0.09800}%
        \definecolor{mycolor3}{rgb}{0.92900,0.69400,0.12500}%
        \definecolor{mycolor4}{rgb}{0.49400,0.18400,0.55600}%
        \begin{tikzpicture}

        \begin{axis}[%
        width=\figurewidth,
        height=0.849\figureheight,
        at={(0\figurewidth,0\figureheight)},
        scale only axis,
        unbounded coords=jump,
        xmin=2,
        xmax=6,
        xtick={2, 3, 4, 5, 6},
        xlabel style={font=\color{white!15!black}},
        xlabel={Number of levels},
        ymode=log,
        ymin=0.01,
        ymax=10,
        yminorticks=true,
        ylabel style={font=\color{white!15!black}},
        ylabel={Convergence factor},
        axis background/.style={fill=white},
        title style={font=\bfseries},
        title={A-stable SDIRK2}
        ]
        \addplot [color=mycolor1, line width=1pt, mark size=1.5pt, mark=*, mark options={solid, mycolor1}, forget plot]
          table[row sep=crcr]{%
        2	0.0126237948958803\\
        3	0.0125223995273946\\
        4	0.0138907722246968\\
        5	0.0173169651195102\\
        6	0.261363667588777\\
        };
        \addplot [color=mycolor2, line width=1pt, mark size=1.5pt, mark=*, mark options={solid, mycolor2}, forget plot]
          table[row sep=crcr]{%
        2	0.0260990800226155\\
        3	0.0287320602092014\\
        4	0.0795579489943677\\
        5	2.52506377273171\\
        6	213.94573099681\\
        };
        \addplot [color=mycolor3, line width=1pt, mark size=1.5pt, mark=*, mark options={solid, mycolor3}, forget plot]
          table[row sep=crcr]{%
        2	0.0409562483157829\\
        3	0.0494020068666622\\
        4	0.172720594494204\\
        5	5.34199571628685\\
        6	443.1092172607\\
        };
        \addplot [color=mycolor4, line width=1pt, mark size=1.5pt, mark=*, mark options={solid, mycolor4}, forget plot]
          table[row sep=crcr]{%
        2	nan\\
        3	nan\\
        4	nan\\
        5	nan\\
        6	nan\\
        };
        \end{axis}
        \end{tikzpicture}%
    \end{subfigure}\qquad\qquad\qquad
    \begin{subfigure}[b]{0.4\textwidth}
        \definecolor{mycolor1}{rgb}{0.00000,0.44700,0.74100}%
        \definecolor{mycolor2}{rgb}{0.85000,0.32500,0.09800}%
        \definecolor{mycolor3}{rgb}{0.92900,0.69400,0.12500}%
        \definecolor{mycolor4}{rgb}{0.49400,0.18400,0.55600}%
        \begin{tikzpicture}

        \begin{axis}[%
        width=\figurewidth,
        height=0.851\figureheight,
        at={(0\figurewidth,0\figureheight)},
        scale only axis,
        unbounded coords=jump,
        xmin=2,
        xmax=6,
        xtick={2, 3, 4, 5, 6},
        xlabel style={font=\color{white!15!black}},
        xlabel={Number of levels},
        ymode=log,
        ymin=0.01,
        ymax=10,
        yminorticks=true,
        yticklabels={,,},
        axis background/.style={fill=white},
        title style={font=\bfseries},
        title={L-stable SDIRK2}
        ]
        \addplot [color=mycolor1, line width=1pt, mark size=1.5pt, mark=*, mark options={solid, mycolor1}, forget plot]
          table[row sep=crcr]{%
        2	0.0244843517742587\\
        3	0.0269869081251273\\
        4	0.0301905224229283\\
        5	0.133896959791831\\
        6	1.02280578017734\\
        };
        \addplot [color=mycolor2, line width=1pt, mark size=1.5pt, mark=*, mark options={solid, mycolor2}, forget plot]
          table[row sep=crcr]{%
        2	0.0505531000941094\\
        3	0.0681299045743775\\
        4	0.436411623452794\\
        5	17.5611723921844\\
        6	479.767273137187\\
        };
        \addplot [color=mycolor3, line width=1pt, mark size=1.5pt, mark=*, mark options={solid, mycolor3}, forget plot]
          table[row sep=crcr]{%
        2	0.0793185638122101\\
        3	0.126352243452018\\
        4	0.914937020783199\\
        5	35.7980221535127\\
        6	971.134880663888\\
        };
        \addplot [color=mycolor4, line width=1pt, mark size=1.5pt, mark=*, mark options={solid, mycolor4}, forget plot]
          table[row sep=crcr]{%
        2	nan\\
        3	nan\\
        4	nan\\
        5	nan\\
        6	nan\\
        };
        \end{axis}
        \end{tikzpicture}%
    \end{subfigure}\\[1ex]
    \setlength{\figurewidth}{0.4\linewidth}
    \setlength{\figureheight}{0.25\linewidth}
    \begin{subfigure}[b]{0.4\textwidth}
        \definecolor{mycolor1}{rgb}{0.00000,0.44700,0.74100}%
        \definecolor{mycolor2}{rgb}{0.85000,0.32500,0.09800}%
        \definecolor{mycolor3}{rgb}{0.92900,0.69400,0.12500}%
        \definecolor{mycolor4}{rgb}{0.49400,0.18400,0.55600}%
        \begin{tikzpicture}

        \begin{axis}[%
        width=\figurewidth,
        height=0.851\figureheight,
        at={(0\figurewidth,0\figureheight)},
        scale only axis,
        unbounded coords=jump,
        xmin=2,
        xmax=6,
        xtick={2, 3, 4, 5, 6},
        xlabel style={font=\color{white!15!black}},
        xlabel={Number of levels},
        ymode=log,
        ymin=0.001,
        ymax=10,
        yminorticks=true,
        ylabel style={font=\color{white!15!black}},
        ylabel={Convergence factor},
        axis background/.style={fill=white},
        title style={font=\bfseries},
        title={A-stable SDIRK3}
        ]
        \addplot [color=mycolor1, line width=1pt, mark size=1.5pt, mark=*, mark options={solid, mycolor1}, forget plot]
          table[row sep=crcr]{%
        2	0.0120191052548178\\
        3	0.0165052986144696\\
        4	0.0444104372344145\\
        5	0.508156235704172\\
        6	0.862606008941354\\
        };
        \addplot [color=mycolor2, line width=1pt, mark size=1.5pt, mark=*, mark options={solid, mycolor2}, forget plot]
          table[row sep=crcr]{%
        2	0.0218651189923282\\
        3	0.0393255270184763\\
        4	0.395489713972101\\
        5	10.6034741261669\\
        6	263.983844330425\\
        };
        \addplot [color=mycolor3, line width=1pt, mark size=1.5pt, mark=*, mark options={solid, mycolor3}, forget plot]
          table[row sep=crcr]{%
        2	0.0341853297088354\\
        3	0.0727254688617068\\
        4	0.751249168207012\\
        5	18.6561048932775\\
        6	515.194321709269\\
        };
        \addplot [color=mycolor4, line width=1pt, mark size=1.5pt, mark=*, mark options={solid, mycolor4}, forget plot]
          table[row sep=crcr]{%
        2	nan\\
        3	nan\\
        4	nan\\
        5	nan\\
        6	nan\\
        };
        \end{axis}
        \end{tikzpicture}%
    \end{subfigure}\qquad\qquad\qquad
    \begin{subfigure}[b]{0.4\textwidth}
        \definecolor{mycolor1}{rgb}{0.00000,0.44700,0.74100}%
        \definecolor{mycolor2}{rgb}{0.85000,0.32500,0.09800}%
        \definecolor{mycolor3}{rgb}{0.92900,0.69400,0.12500}%
        \definecolor{mycolor4}{rgb}{0.49400,0.18400,0.55600}%
        \begin{tikzpicture}

        \begin{axis}[%
        width=\figurewidth,
        height=0.849\figureheight,
        at={(0\figurewidth,0\figureheight)},
        scale only axis,
        unbounded coords=jump,
        xmin=2,
        xmax=6,
        xtick={2, 3, 4, 5, 6},
        xlabel style={font=\color{white!15!black}},
        xlabel={Number of levels},
        ymode=log,
        ymin=0.001,
        ymax=10,
        yminorticks=true,
        yticklabels={,,},
        axis background/.style={fill=white},
        title style={font=\bfseries},
        title={L-stable SDIRK3}
        ]
        \addplot [color=mycolor1, line width=1pt, mark size=1.5pt, mark=*, mark options={solid, mycolor1}, forget plot]
          table[row sep=crcr]{%
        2	0.00323837021802482\\
        3	0.0039684847043296\\
        4	0.00528749545419481\\
        5	0.0672154739136417\\
        6	0.0737210978122094\\
        };
        \addplot [color=mycolor2, line width=1pt, mark size=1.5pt, mark=*, mark options={solid, mycolor2}, forget plot]
          table[row sep=crcr]{%
        2	0.00645508811734188\\
        3	0.00808701241725648\\
        4	0.0378390479386906\\
        5	1.06665561426293\\
        6	15.1743204420361\\
        };
        \addplot [color=mycolor3, line width=1pt, mark size=1.5pt, mark=*, mark options={solid, mycolor3}, forget plot]
          table[row sep=crcr]{%
        2	0.0101186545764903\\
        3	0.0142554459353865\\
        4	0.0779607239156913\\
        5	2.0242369238712\\
        6	27.6867508167403\\
        };
        \addplot [color=mycolor4, line width=1pt, mark size=1.5pt, mark=*, mark options={solid, mycolor4}, forget plot]
          table[row sep=crcr]{%
        2	nan\\
        3	nan\\
        4	nan\\
        5	nan\\
        6	nan\\
        };
        \end{axis}
        \end{tikzpicture}%
    \end{subfigure}\\[1ex]
    \setlength{\figurewidth}{0.4\linewidth}
    \setlength{\figureheight}{0.25\linewidth}
    \begin{subfigure}[b]{0.4\textwidth}
        \definecolor{mycolor1}{rgb}{0.00000,0.44700,0.74100}%
        \definecolor{mycolor2}{rgb}{0.85000,0.32500,0.09800}%
        \definecolor{mycolor3}{rgb}{0.92900,0.69400,0.12500}%
        \definecolor{mycolor4}{rgb}{0.49400,0.18400,0.55600}%
        \begin{tikzpicture}

        \begin{axis}[%
        width=\figurewidth,
        height=0.849\figureheight,
        at={(0\figurewidth,0\figureheight)},
        scale only axis,
        unbounded coords=jump,
        xmin=2,
        xmax=6,
        xtick={2, 3, 4, 5, 6},
        xlabel style={font=\color{white!15!black}},
        xlabel={Number of levels},
        ymode=log,
        ymin=1e-07,
        ymax=10,
        yminorticks=true,
        ylabel style={font=\color{white!15!black}},
        ylabel={Convergence factor},
        axis background/.style={fill=white},
        title style={font=\bfseries},
        title={A-stable SDIRK4}
        ]
        \addplot [color=mycolor1, line width=1pt, mark size=1.5pt, mark=*, mark options={solid, mycolor1}, forget plot]
          table[row sep=crcr]{%
        2	0.00473339681471841\\
        3	0.00570012359076608\\
        4	0.042275299622819\\
        5	0.393489733360033\\
        6	0.24164775242275\\
        };
        \addplot [color=mycolor2, line width=1pt, mark size=1.5pt, mark=*, mark options={solid, mycolor2}, forget plot]
          table[row sep=crcr]{%
        2	0.00727312389153332\\
        3	0.0147625314609372\\
        4	0.256344803985121\\
        5	4.07424782102752\\
        6	51.3717336100215\\
        };
        \addplot [color=mycolor3, line width=1pt, mark size=1.5pt, mark=*, mark options={solid, mycolor3}, forget plot]
          table[row sep=crcr]{%
        2	0.0114075353317239\\
        3	0.0275644643958277\\
        4	0.47238551855669\\
        5	6.49731471331314\\
        6	96.8873295153152\\
        };
        \addplot [color=mycolor4, line width=1pt, mark size=1.5pt, mark=*, mark options={solid, mycolor4}, forget plot]
          table[row sep=crcr]{%
        2	nan\\
        3	nan\\
        4	nan\\
        5	nan\\
        6	nan\\
        };
        \end{axis}
        \end{tikzpicture}%

    \end{subfigure}\qquad\qquad\qquad
    \begin{subfigure}[b]{0.4\textwidth}
        \definecolor{mycolor1}{rgb}{0.00000,0.44700,0.74100}%
        \definecolor{mycolor2}{rgb}{0.85000,0.32500,0.09800}%
        \definecolor{mycolor3}{rgb}{0.92900,0.69400,0.12500}%
        \definecolor{mycolor4}{rgb}{0.49400,0.18400,0.55600}%
        \begin{tikzpicture}

        \begin{axis}[%
        width=\figurewidth,
        height=0.849\figureheight,
        at={(0\figurewidth,0\figureheight)},
        scale only axis,
        unbounded coords=jump,
        xmin=2,
        xmax=6,
        xtick={2, 3, 4, 5, 6},
        xlabel style={font=\color{white!15!black}},
        xlabel={Number of levels},
        ymode=log,
        ymin=1e-07,
        ymax=10,
        yminorticks=true,
        yticklabels={,,},
        axis background/.style={fill=white},
        title style={font=\bfseries},
        title={L-stable SDIRK4}
        ]
        \addplot [color=mycolor1, line width=1pt, mark size=1.5pt, mark=*, mark options={solid, mycolor1}, forget plot]
          table[row sep=crcr]{%
        2	5.68527005677997e-07\\
        3	5.6842020787866e-07\\
        4	5.73358923892536e-07\\
        5	2.58410883628764e-06\\
        6	0.000243520560478492\\
        };
        \addplot [color=mycolor2, line width=1pt, mark size=1.5pt, mark=*, mark options={solid, mycolor2}, forget plot]
          table[row sep=crcr]{%
        2	3.94737354378682e-05\\
        3	3.94796994891262e-05\\
        4	3.9862557389292e-05\\
        5	0.000176618155009953\\
        6	0.378508156908287\\
        };
        \addplot [color=mycolor3, line width=1pt, mark size=1.5pt, mark=*, mark options={solid, mycolor3}, forget plot]
          table[row sep=crcr]{%
        2	6.19445610379228e-05\\
        3	6.19510424022491e-05\\
        4	6.37369536838193e-05\\
        5	0.000419439521747319\\
        6	0.777120455015235\\
        };
        \addplot [color=mycolor4, line width=1pt, mark size=1.5pt, mark=*, mark options={solid, mycolor4}, forget plot]
          table[row sep=crcr]{%
        2	nan\\
        3	nan\\
        4	nan\\
        5	nan\\
        6	nan\\
        };
        \end{axis}
        \end{tikzpicture}%
    \end{subfigure}
    \caption{Wave equation: The convergence of MGRIT with F-cycles and FCF-relaxation
        deteriorates with a larger number of time grid levels compared to MGRIT with V-cycles,
        see Figure \ref{wave-V-cycle-nt1024-FCF-relaxation-nn11x11-suppfig}.
        Generally, convergent algorithms are given for a larger range of time grid levels
        and observed convergence is better than the predictions from the upper bounds.
        This shows that the choice of F-cycles over V-cycles is one likely ingredient
        for future improvements of MGRIT for hyperbolic-type PDEs.}
    \label{wave-F-cycle-nt1024-FCF-relaxation-nn11x11-suppfig}
\end{figure}
\FloatBarrier

\end{document}